\documentclass[10pt]{article}

\usepackage[english]{babel}
\usepackage[utf8]{inputenc}
\usepackage{amsmath}
\usepackage{amssymb}
\usepackage{amsthm}
\usepackage[all,cmtip]{xy}
\usepackage[bookmarks=true]{hyperref}
\usepackage{tikz}
\usepackage[disable]{todonotes}
\usepackage{slashed}
\usepackage{mathtools}
\usepackage{dsfont}
\usepackage{tensor}
\usepackage{verbatim}
\usepackage{graphicx}
\usepackage[shortlabels]{enumitem}
\usepackage{xstring}
\usepackage[left=2.5cm,right=2.5cm,top=2cm,bottom=3.3cm]{geometry}
\usepackage[capitalize]{cleveref}
\usepackage{tensor}
\usepackage{mathrsfs}
\usepackage[ddmmyy,24hr]{datetime}

\definecolor{darkolivegreen}{rgb}{0.33, 0.42, 0.18} 
\definecolor{cobalt}{rgb}{0.0, 0.24, 0.43}

\newcounter{comments}
\newenvironment{displaycomment}{\begin{list}{}{\rightmargin=1cm\leftmargin=1cm}\item\sf\begin{small}}{\end{small}\end{list}}

\newcommand{\C}{\mathbb{C}}
\newcommand{\R}{\mathbb{R}}

\newcommand{\mc}[1]{\mathcal{#1}}
\def\smooth{\infty} 
\newcommand{\Clsm}{\Cl(V)^{\smooth}} 

\newcommand*{\defeq}{\mathrel{\vcenter{\baselineskip0.5ex \lineskiplimit0pt
                        \hbox{\scriptsize.}\hbox{\scriptsize.}}}%
        =}

\newcommand{\ph}{\varphi}

\newcommand{\eps}{\varepsilon}

\newcommand{\Aut}{\operatorname{Aut}}
\newcommand{\SO}{\operatorname{SO}}
\newcommand{\Spin}{\operatorname{Spin}}
\newcommand{\U}{\operatorname{U}}

\newcommand{\Imp}{\operatorname{Imp}}
\newcommand{\Fr}{\operatorname{Fr}}

\newcommand{\Cl}{\operatorname{Cl}}
\newcommand{\lie}[1]{\mathfrak{#1}}

\newcommand{\opp}{\operatorname{opp}}
\newcommand{\pr}{\operatorname{pr}}

\newcommand{\lact}{\triangleright}

\newcommand{\Diff}{\operatorname{Diff}}

\newcommand{\VBdl}{\mathcal{RHB}dl}

\def\twist{\mathrm{tw}}

\newcommand{\BCE}{\smash{\widetilde{L\Spin}}(d)}

\renewcommand{\O}{\operatorname{O}}
\renewcommand{\d}{\operatorname{d}}
\newcommand{\der}[2]{\frac{\d #1}{\d #2}} 

\newcommand{\Triv}{\mathcal{T}\!riv}
\newcommand{\Lift}{\mathcal{L}i\!ft}
\newcommand{\Mod}[1]{#1\text{-}\hspace{-0.04em}\VBdl}

\newcommand{\EqImp}{\Imp}

\newcommand{\Lag}{\text{Lag}}
\newcommand{\D}{\slashed{D}} 
\newcommand{\dd}[2]{\frac{\text{d} #1}{\text{d} #2}}

\def\res{L}

\def\quot#1{,,#1``}

\def\cstar{C$^{*}\!$}

\def\pol#1{^{\infty}(#1)^{\mathrm{pol}}}
\def\frm#1{^{\infty}(#1)^{\mathrm{frm}}}

\theoremstyle{definition}
\newtheorem{definition}{Definition}[subsection]
\newtheorem{remark}[definition]{Remark}
\newtheorem{example}[definition]{Example}

\theoremstyle{plain}
\newtheorem{theorem}[definition]{Theorem}
\newtheorem{proposition}[definition]{Proposition}
\newtheorem{lemma}[definition]{Lemma}
\newtheorem{corollary}[definition]{Corollary}

\setlist{topsep=-0.2em,itemsep=0em}

\crefname{enumi}{\unskip}{\unskip}
\crefname{equation}{}{}

\title{Smooth Fock bundles, and spinor bundles on loop space}

\author{Peter Kristel and Konrad Waldorf}
\date{}

\makeatletter
\hypersetup{
colorlinks=true,
urlcolor=cobalt,
linkcolor=cobalt,
citecolor=darkolivegreen,
pdfmenubar=false,
pdftoolbar=true,
pdftitle = {\@title},
pdfauthor = {\@author},
}
\makeatother

\begin{document}

\maketitle

\begin{abstract}
\noindent
We address the construction of smooth bundles of fermionic Fock spaces, a problem that appears frequently in fermionic gauge theories. Our  main motivation is the spinor bundle on the free loop space of a string manifold, a structure anticipated by Killingback, with a construction outlined by Stolz-Teichner. We develop a general framework for constructing smooth Fock bundles, and obtain as an application a complete and well-founded construction of  spinor bundles on loop space.
         
\medskip

\end{abstract}

\listoftodos

     \setlength{\parskip}{0ex}
     \begingroup
     \tableofcontents
     \endgroup
     \setlength{\parskip}{1.5ex}

\section{Introduction}

\thispagestyle{empty}

A fermionic Fock space $\mathcal{F}_L$ is the Hilbert completion of the exterior algebra of a Lagrangian subspace $L$ of a   Hilbert space $\mathcal{H}$ (complex, but equipped with a real structure),
and it carries an irreducible representation of a Clifford \cstar-algebra $\Cl(\mathcal{H})$. A typical situation that appears in fermionic gauge theories is that the Hilbert space is the fibre $\mathcal{H}_p$ of a continuous bundle of Hilbert spaces over a parameter space $\mathcal{M}$, but Lagrangians $L_p \subset \mathcal{H}_p$ cannot be chosen to depend continuously on the parameter $p \in \mathcal{M}$.
Consequently, the Fock spaces $\mathcal{F}_{L_p}$ do not automatically form a continuous bundle of Hilbert spaces over $\mathcal{M}$. It is well-understood that in general only a projective bundle can be defined, or, in more modern terms, a  bundle twisted by a bundle gerbe over $\mathcal{M}$. 

In case of an anticipated spinor bundle on loop space, the parameter space $\mathcal{M}$ is the loop space $LM$ of a spin manifold $M$, and a proposal of Stolz-Teichner \cite{stolz3} is to consider for the Hilbert space $\mathcal{H}_{\gamma}$ over a loop $\gamma\in LM$  a certain space of  sections of the pullback bundle $\gamma^{*}TM \to S^1$; further, to consider as a typical Lagrangian, roughly speaking, a space of sections that \quot{extend to anti-holomorphic functions on the disk}. A bundle gerbe over $LM$ that obstructs the existence of a continuous Fock bundle on $LM$, together with a twisted Fock bundle has been constructed by Ambler \cite{Ambler2012}.  

In this article, we set up a general theory for constructing bundles of Fock spaces, which reproduces, completes, and generalizes the treatment of Stolz-Teichner and Ambler. However, we go one step beyond the construction of \emph{continuous} Fock bundles, and construct \emph{smooth} ones. Our main motivation for this is that at some point one may want to study differential operators acting on sections of such bundles.
Let us point out that the difference between \quot{continuous} and \quot{smooth}  is more than a mere technicality.
In infinite-dimensional representation theory, one typically encounters
unitary representations $\mathcal{G} \to \U(\mathcal{H})$ of perfectly well-behaved Banach Lie groups $\mathcal{G}$ on a Hilbert space $\mathcal{H}$ which are continuous with respect to the strong topology on $\U(\mc{H})$, but not with respect to its norm topology.
Consequently, the  action map
$\mathcal{G} \times \mathcal{H} \to \mathcal{H}$
is continuous, but not smooth.
This means that to such representations and, say, a principal $\mathcal{G}$-bundle, one can easily associate Hilbert space bundles with continuous transition functions, but not with smooth ones. Our conclusion is that smooth bundles of Fock spaces cannot be defined in any naive way, and require a more elaborate framework, in particular, concerning the representation-theoretical aspects. In fact, it was remarked by Henriques in a Mathoverflow question \cite{Henriques2015} that it is actually not at all clear what a sensible concept of a smooth Hilbert space bundle is. In a comment to this question, Tao proposed to use the concept of a \emph{rigged Hilbert space}, i.e., to equip Hilbert spaces with  Fr\'echet subspaces with continuous inclusion maps with dense image. 
In this article we follow  this idea and we show how well this fits into the situation of bundles of Fock spaces. 

Let us start with describing the representation-theoretical part of our work. Let $V$ be a complex Hilbert space with a real structure, and let $L\subset V$ be a Lagrangian subspace. The work of Araki \cite{Ara74,Ara85} and Segal \cite{PS86} introduced a  central extension
\begin{equation}
\label{intro:ce}
1\to \U(1) \to \Imp_L(V) \to \O_L(V) \to 1
\end{equation}
of Banach Lie groups. The \emph{restricted orthogonal group} $\O_L(V)$ acts on the Clifford algebra $\Cl(V)$ by Bogoliubov automorphisms, and the group of \emph{implementers} $\Imp_L(V)$ acts on the Fock space $\mathcal{F}_L$ by unitary transformations (however, the action  $\Imp_L(V) \times \mathcal{F}_L \to \mathcal{F}_L$ is not smooth). These two actions are compatible with the Clifford multiplication in the sense that the famous \emph{implementability condition}
\begin{equation*}
U(\theta_g(a)\lact v )= a\lact Uv
\end{equation*} 
is satisfied for all $a\in \Cl(V)$, $v\in \mathcal{F}_L$, and $U \in \Imp_L(V)$ projecting to $g\in \O_L(V)$, where $\theta_g$ is the Bogoliubov automorphism associated to $g$, and $\lact$ denotes the Clifford multiplication. The first set of new results of this article is a complete upgrade of  these representations to a \emph{rigged} setting:
\begin{itemize}

\item 
We construct a subspace $\mathcal{F}_L^{\infty} \subset \mathcal{F}_L$ that equips the Fock space $\mathcal{F}_L$ with the structure  of a rigged Hilbert space, such that the action $\Imp_L(V) \times \mathcal{F}_L^{\infty} \to \mathcal{F}_L^{\infty}$ is smooth (\cref{prop:smoothFockSpace}). 

\item
We construct a Fr\'echet subalgebra $\Cl(V)^{\infty}\subset \Cl(V)$, in such a way  that the action by Bogoliubov automorphisms restricts to a smooth action $\O_L(V) \times \Cl(V)^{\infty} \to \Cl(V)^{\infty}$  (\cref{prop:CliffordFrechetAlgebra}).
We will say that $\Cl(V)^{\infty}$ is a \emph{rigged \cstar-algebra}, a concept that apparently as not been found before, but turns out to be very useful for our purposes.

\item
We prove that Clifford multiplication restricts to a smooth action $\Cl(V)^{\infty} \times \mathcal{F}_L^{\infty} \to \mathcal{F}_L^{\infty}$ (\cref{prop:SmoothCliffordActionOnF}). 
\end{itemize}

Above representation-theoretic results describe the situation in the typical fibre of the rigged Fock bundles we are going to construct:  $\mathcal{F}_L^{\infty}\subset \mathcal{F}_L$ and $\Cl(V)^{\infty}\subset \Cl(V)$ are the typical fibres of a rigged Fock bundle, and a rigged \cstar-algebra bundle, respectively.
Taking care of the bundle-theoretic, global aspects is the second set of results of this article. 

\begin{itemize}

\item 
We develop a theory of \emph{rigged Hilbert space bundles} over Fr\'echet manifolds (\cref{def:HilbertBundle}). In short, this is a continuous Hilbert space bundle $\mathcal{H}$ together with a Fr\'echet vector bundle $\mathcal{E}$ with a continuous dense inclusion $\mathcal{E} \to \mathcal{H}$.  We prove that rigged Hilbert space bundles can be constructed by associating a  unitary representation with smooth action map (like the one of $\Imp_L(V)$ on $\mathcal{F}_L^{\infty}$) to a principal Fr\'echet bundle (\cref{lem:AssociatedHilbertBundle}). 
\item
In order to properly keep track of the Clifford algebra action, we develop a corresponding theory of \emph{rigged \cstar-algebra bundles} (\cref{def:cstarBundle}), and show 
how to construct examples by associating a   representation with smooth action map (like the one of $\O_L(V)$ on $\Cl(V)^{\infty}$) to Fr\'echet principal bundles (\cref{lem:RiggedCStarBundle}). 

\end{itemize}

Now, the geometric basis of our general framework is a Fr\'echet Lie group $\mathcal{G}$ and a Fr\'echet principal $\mathcal{G}$-bundle $\mathcal{E}$ over a Fr\'echet manifold $\mathcal{M}$. The relation between geometry and representation theory is established by requiring a smooth group homomorphism $\omega: \mathcal{G} \to \O_L(V)$. We call the triple $(V,L,\omega)$ a \emph{Fock extension} for the Fr\'echet Lie group $\mathcal{G}$ (\cref{def:fockextension}). The reason for this terminology is that the central extension \cref{intro:ce} pulls back along $\omega$ to a central extension
\begin{equation}
\label{intro:cepullback}
1\to \U(1) \to \widetilde{\mathcal{G}} \to \mathcal{G} \to 1\text{,}
\end{equation}  
with induced smooth representations of $\mathcal{G}$ on $\Cl(V)^{\infty}$ and of $\widetilde{\mathcal{G}}$ on $\mathcal{F}_L^{\infty}$. Using our results listed above, we form the associated rigged \cstar-algebra bundle $\Cl^{\infty}_V(\mathcal{E})\defeq(\mathcal{E} \times \Cl(V)^{\infty})/\mathcal{G}$ over $\mathcal{M}$ with typical fibre $\Cl(V)^{\infty}$. Our main result concerning the construction of smooth Fock bundles (see \cref{sec:smoothfockbundles}) is the following:

\def\thedefinition{\Alph{definition}}
\begin{theorem}
\label{intro:A}
Let $\mathcal{E}$ be a Fr\'echet principal $\mathcal{G}$-bundle over $\mathcal{M}$, and let $(V,L,\omega)$ be a Fock extension of $\mathcal{G}$.
Suppose that $\widetilde{\mathcal{E}}$ is a lift of the structure group of $\mathcal{E}$ from $\mathcal{G}$ to the central extension $\widetilde{\mathcal{G}}$. Then, the associated bundle 
\begin{equation*}
\mathcal{F}^{\infty}_L(\widetilde{\mathcal{E}})\defeq(\widetilde{\mathcal{E}} \times \mathcal{F}_L^{\infty})/\widetilde{\mathcal{G}}
\end{equation*}
is rigged Hilbert space bundle over $\mathcal{M}$ with typical fibre the rigged Fock space $\mathcal{F}_L^{\infty}$. Moreover, it is a rigged $\Cl^{\infty}_V(\mathcal{E})$-module bundle.  
\end{theorem}

In our application to spinor bundles on loop space, we consider a real vector bundle $E$ over a smooth manifold $M$, equipped with a spin structure $\Spin(E)$. For instance, $E=TM$ and $\Spin(E)$ is a spin structure on $M$. We consider the free loop space $\mathcal{M}\defeq LM$ and the looped bundle $\mathcal{E} := L\Spin(E)$, which is a Fr\'echet principal bundle over $\mathcal{M}$ for the loop group $\mathcal{G}=L\Spin(d)$, where $d$ is the dimension of $M$. We provide two possible Fock extensions $(V_{ev},L_{ev},\omega_{ev})$ and $(V_{odd},L_{odd},\omega_{odd})$ of $\mathcal{G}$, based on spaces of even resp.~odd spinors on the circle. We show in both cases (\cref{prop:ceodd,prop:EvenBasicCE}) that the  extension \cref{intro:cepullback} is the basic central extension
\begin{equation*}
1\to \U(1) \to \widetilde{L\Spin}(d) \to L\Spin(d) \to 1
\end{equation*}
of the loop group studied, e.g., by Pressley-Segal \cite{PS86}. This way, a lift $\widetilde{L\Spin}(E)$ as required in \cref{intro:A} is precisely what Killingback called a \emph{string structure} on $E$ \cite{killingback1}.
Applying \cref{intro:A} for our two Fock extensions now leads to two rigged Fock bundles
\begin{equation*}
\mathcal{F}_{L_{ev}}^{\infty}(\widetilde{L\Spin}(E))
\quad\text{ and }\quad 
\mathcal{F}_{L_{odd}}^{\infty}(\widetilde{L\Spin}(E))
\text{,}
\end{equation*}  
which we call the \emph{even} and the \emph{odd loop spinor bundle} of the string vector bundle $E$. Further, these loop spinor bundles are equipped with a Clifford multiplication by the rigged Clifford algebra bundles $\Cl^{\infty}_{V_{ev}}(L\Spin(E))$ and $\Cl^{\infty}_{V_{odd}}(L\Spin(E))$, respectively. We believe that our construction achieves, for the first time, a rigorous construction of loop spinor bundles in a continuous (and now even smooth) setting.
We describe below  that our construction is consistent with the work of Stolz-Teichner \cite{stolz3} and Ambler \cite{Ambler2012}, and in how far it improves these.

The last big group of results of this article is concerned with the situation where  lifts $\widetilde{\mathcal{E}}$  as required in \cref{intro:A} are not available, so that the construction of rigged Fock bundles (already of continuous bundles) fails. In our application to loop spaces this corresponds to the situation that $E$ does not admit  string structures.  A common slogan one often finds in the literature is that one should then aim at constructing \emph{projective} bundles. Instead, we use the finer notion of \emph{twisted} bundles, where the twist appears as a  separate geometric object,  a bundle gerbe. Twisted bundles have been introduced by Lupercio-Uribe \cite{Lupercio2004}, Mackaay \cite{Mackaay2003}, and by Bouwknegt-Carey-Mathai-Murray-Stevenson  under the name of \quot{bundle gerbe modules} \cite{Bouwknegt2002}. In this article, we generalize this theory to Fr\'echet bundle gerbes, and \emph{twisted rigged Hilbert space bundles} (\cref{def:trhsb}), and also take an additional action of a rigged \cstar-algebra bundle into account (\cref{def:trmb}).  We prove a number of useful results;
most importantly, we show that twisted rigged space bundles can be untwisted by trivializations of the twisting bundle gerbe (\cref{lem:twistingmod}).  

The principal $\mathcal{G}$-bundle $\mathcal{E}$  and the central extension \cref{intro:cepullback} induce a \emph{lifting bundle gerbe} $\mathscr{L}_{\mathcal{E}}$ in the sense of Murray \cite{Murray1996}, whose trivializations correspond precisely to the lifts $\widetilde{\mathcal{E}}$. We construct a canonical $\mathscr{L}_{\mathcal{E}}$-twisted rigged Hilbert space bundle $\mathcal{F}^{\infty}(\mathcal{E})^{\twist}$ which we call the \emph{twisted Fock bundle} (\cref{def:twistedfock}), which exists no matter if lifts $\widetilde{\mathcal{E}}$ exist or not. The terminology comes from the fact that the twisted Fock bundle is a twisted version of the Fock bundle  of \cref{intro:A}, as our next main result shows (see \cref{th:twistedspinor}):

\begin{theorem}
\label{th:B}
Suppose $\mathcal{T}$ is a trivialization of the lifting gerbe $\mathscr{L}_{\mathcal{E}}$, corresponding  to a lift $\widetilde{\mathcal{E}}$ of the structure group of $\mathcal{E}$ along the central extension \cref{intro:cepullback}. Then,  untwisting the twisted Fock bundle $\mathcal{F}^{\infty}(\mathcal{E})^{\twist}$ using the trivialization $\mathcal{T}$ yields precisely the Fock bundle $\mathcal{F}^{\infty}_L(\widetilde{\mathcal{E}})$.
\end{theorem}

\Cref{th:B} allows to work with the twisted Fock bundle $\mathcal{F}^{\infty}(\mathcal{E})^{\twist}$ in situations where no lift $\widetilde{\mathcal{E}}$ exists, or one wishes to leave the choice of a lift open. In case of the spinor bundle on loop space, the lifts $\widetilde{\mathcal{E}}$ are  the string structures on $\mathcal{E}$, and not to chose one a priori   may be useful in applications to M-theory (see \cref{sec:loopspin}). 
It also allows to compare our construction of the spinor bundle on loop space with existing proposals, as these end with projective bundles and do not ultimately perform the untwisting.

In the general frame work, we consider -- for the purpose of comparison -- the  Hilbert space bundle $\mathcal{H}$ with fibre $V$ obtained by associating to $\mathcal{E}$ the representation $\mathcal{G} \to \O_L(V)$. This brings us into the situation mentioned at the very beginning: a continuous Hilbert space bundle with the problem that a continuous choice of Lagrangians $L_p \subset \mathcal{H}_p$ is obstructed. We parameterize this problem in two ways: 
\begin{enumerate}[(a)]

\item 
In a direct way by examining  the Fr\'echet fibre bundle of equivalence classes of Lagrangians in $\mathcal{H}$. The obstruction is then represented by a Fr\'echet bundle gerbe $\mathscr{G}_{\mathcal{H}}$ over $\mathcal{M}$, which we call the \emph{Lagrangian Gra{\ss}mannian gerbe}. We equip it with a twisted rigged Hilbert space bundle $\mathcal{F}\pol{\mathcal{H}}$, which we call the \emph{polarization-dependent Fock bundle} (\cref{sec:LagrangianTwistedSpinorBundle}).

\item
In an indirect way by considering certain restricted orthogonal frames in $\mathcal{H}$. The obstruction is then represented by a lifting bundle gerbe $\mathscr{L}_{\mathcal{H}}$ over $\mathcal{M}$, which we call the \emph{implementer lifting gerbe}. We equip it with a twisted rigged Hilbert space bundle $\mathcal{F}\frm{\mathcal{H}}$, which we call the \emph{frame-dependent Fock bundle} (\cref{sec:impLiftingGerbe}).

\end{enumerate}
We prove that these structures related to the bundle $\mathcal{H}$ are completely equivalent to the lifting gerbe $\mathscr{L}_{\mathcal{E}}$ and the twisted Fock bundle $\mathcal{F}^{\infty}(\mathcal{E})^{\twist}$. More precisely, we construct a sequence of canonical isomorphisms of Fr\'echet  bundle gerbes
\begin{equation*}
\mathscr{L}_{\mathcal{E}} \to \mathscr{L}_{\mathcal{H}} \to \mathscr{G}_{\mathcal{H}}
\end{equation*}
over $\mathcal{M}$. Under these isomorphisms, the polarization-dependent Fock bundle $\mathcal{F}\pol{\mathcal{H}}$ of $\mathscr{G}_{\mathcal{H}}$ pulls back to the frame-dependent Fock bundle $\mathcal{F}\frm{\mathcal{H}}$ of $\mathscr{L}_{\mathcal{H}}$ (\cref{prop:equivambler}), which in turn pulls back to the twisted Fock bundle  $\mathcal{F}^{\infty}(\mathcal{E})^{\twist}$ of $\mathscr{L}_{\mathcal{E}}$ (\cref{prop:refinement}). In particular, untwisting any of these twisted rigged Hilbert space bundles yields the the ordinary Fock bundle $\mathcal{F}^{\infty}(\widetilde{\mathcal{E}})$ as rigged $\Cl_V^{\infty}(\mathcal{E})$-module bundles (\cref{co:untwisting1,co:untwisting2}).

In the following we describe how these equivalent structures -- in the application to spinor bundles on loop space -- make  contact to the existing discussions of Stolz-Teichner \cite{stolz3} and Ambler \cite{Ambler2012}. First of all, both approaches can be subsumed in our general framework, in such a way that they correspond to  the two Fock extensions $(V_{odd},L_{odd},\omega_{odd})$ and $(V_{ev},L_{ev},\omega_{ev})$ mentioned above. The Hilbert space bundle $\mathcal{H}$ obtains concrete interpretations in terms of spinors on the circle. In case of the odd Fock extension, its fibre over a loop $\gamma\in LM$ is the Hilbert space sections  $L^2(S^1,\mathbb{S} \otimes \gamma^{*}E)$ of the tensor product of the odd spinor bundle $\mathbb{S}$ on $S^1$ and the pullback of the given bundle $E$ along the loop $\gamma:S^1 \to M$  (\cref{lem:fibresoddspinor}). This Hilbert space is the basis of the proposal of Stolz-Teichner. They consider a canonical equivalence class of Lagrangians in each fibre $\mathcal{H}_p$, given by the positive Eigenspaces of the twisted Dirac operator $\D_{\gamma}$ acting on $\mathcal{H}_{\gamma}$, and form the corresponding projective bundle of Fock spaces. We prove (\cref{prop:equivalenceclasslag}) that our construction of the frame-dependent Fock bundle $\mathcal{F}\frm{\mathcal{H}}$ leads to the same equivalence class of Lagrangians (though we never used $\D_{\gamma}$). As a result, our frame-dependent Fock bundle is a non-projective, smooth version of the projective Fock bundle of Stolz-Teichner.

In case of the even Fock extension, the fibre of $\mathcal{H}$ over $\gamma$ is the Hilbert space $L^2(S^1,\gamma^{*}E)$, which is the basis of the proposal of Ambler  (\cref{lem:fibresevenspinor}). It turns further out that Ambler's work is parallel to our discussion of the polarization-dependent Fock bundle. Thus, our work provides an upgrade of Ambler's from a purely topological version to a smooth version in the setting of rigged Hilbert spaces.   

In both cases, our treatment provides a neat explanation how string structures on $E$ untwist these twisted  bundles to yield true Fock bundles, via \cref{th:B}. This link was missing so far in the literature.
The following table summarizes the matches between the structures we introduced here, and existing structures:
\begin{center}
\setlength\tabcolsep{4pt}
\begin{tabular}{|l|c|c|}\hline
\hspace{\tabcolsep}\textbf{General framework}\begin{tabular}{c}\strut\\\strut\end{tabular} & \textbf{Odd spinor bundle} & \textbf{Even spinor bundle}
\\\hline\hline
\hspace{\tabcolsep}Central extension\begin{tabular}{c}\strut\\\strut\end{tabular} & \multicolumn{2}{|c|}{The basic central extension of $L\Spin(d)$} \\\hline
\hspace{\tabcolsep}Lifting gerbe $\mathscr{L}_{\mathcal{E}}\begin{tabular}{c}\strut\\\strut\end{tabular}$&  \multicolumn{2}{|c|}{The spin lifting gerbe on $LM$}\\\hline
\hspace{\tabcolsep}Hilbert space bundle $\mathcal{H}\begin{tabular}{c}\strut\\\strut\end{tabular}$ & \begin{tabular}{c}Twisted odd spinors on the circle\\$\mathcal{H}_{\gamma}=L^2(S^1,\mathbb{S} \otimes \gamma^{*}E)$\end{tabular} & \begin{tabular}{c}Twisted spinors on the circle\\$\mathcal{H}_{\gamma}=L^2(S^1,\gamma^{*}E_{\C})$\end{tabular} \\\hline
\begin{tabular}{l}Implementer lifting\\gerbe $\mathscr{L}_{\mathcal{H}}$\end{tabular} & \begin{tabular}{c}Represents the projectivity\\of Stolz-Teichner's Fock bundle\end{tabular} & -- \\\hline
\begin{tabular}{l}Frame-dependent\\Fock bundle $\mathcal{F}\frm{\mathcal{H}}$\end{tabular} & \begin{tabular}{c}A non-projective, twisted version\\of Stolz-Teichner's Fock bundle\end{tabular} & -- \\\hline
\begin{tabular}{l}Lagrangian Gra\ss mannian\\gerbe $\mathscr{G}_{\mathcal{H}}$\end{tabular} & -- & \begin{tabular}{c}
A Fr\'echet version of\\Ambler's bundle gerbe\end{tabular} \\\hline
\begin{tabular}{l}Polarization-dependent\\Fock bundle $\mathcal{F}\pol{\mathcal{H}}$\end{tabular} & -- & \begin{tabular}{c}
A smooth version of\\Ambler's Fock space bundle\end{tabular} \\\hline
\end{tabular}
\end{center}

Concerning the organization of this paper, we set up our theory of rigged (twisted) Hilbert space bundles in \cref{sec:HilbertBundle},  and address the representation-theoretic problems in \cref{sec:freefermions}. The general construction of Fock bundles is located in \cref{sec:smoothfockbundles}, and the twisted versions in \cref{sec:twistedspinorbundles}. The application of these structures to spinor bundles on loop space is performed in \cref{sec:spinorbundles}.

In a future paper we intend to apply our general framework to the construction of Fock bundles over the moduli space of gauge connections, which has been described by Segal \cite{Segal1984}, Carey-Murray \cite{Carey1996}, and Mickelsson \cite{Carey1997,Mickelsson1989}. In this context, the bundle gerbe that obstructs the existence of non-projective Fock bundles is called the \emph{Faddeev-Mickelsson-Shatashvili bundle gerbe}, and it represent a field-theoretical anomaly called the \emph{Hamiltonian anomaly}.  
It turns out that in order to apply our framework to this situation,  a slight generalization of the notion of Fock extension is necessary to do this, and this is why we decided to move this into another paper.

\paragraph{Acknowledgements. } This project was funded by the German Research Foundation (DFG) under project code WA 3300/1-1. We would like to thank Matthias Ludewig, Jouko Mickelsson, and Peter Teichner for helpful discussions.

\section{Smooth bundles of Hilbert spaces and \texorpdfstring{\cstar}{C*}-algebras} 
\def\thedefinition{\thesubsection.\arabic{definition}}

\label{sec:HilbertBundle}

In this section we set up our framework for smooth bundles of Hilbert spaces and \cstar-algebras over Fr\'echet manifolds. All  Fock bundles we construct will be placed in this framework.
The main idea is to use  the well-known notion of a rigged Hilbert space, which provides a neat combination of the smooth and the functional analytical setting. We introduce a new, analogous definition of a rigged \cstar-algebra, and  introduce corresponding notions  of locally trivial smooth bundles of those. Though these notions seem to be very natural ones, we have not seen them written up anywhere.
We show how to construct examples by associating certain representations to Fr\'echet principal bundles. We also discuss how rigged \cstar-algebra bundles act on rigged Hilbert space bundles; this will be our framework for Clifford multiplication. Finally, we set up the framework for twisted bundles, where the twist is represented by a bundle gerbe. This will be used in \cref{sec:twistedspinorbundles} for the discussion of various versions of twisted Fock bundles.

\subsection{Rigged Hilbert space bundles}
\label{sec:HilbertBundle:1}

The notion of a rigged Hilbert space originates from quantum mechanics, where it provides Dirac's bra-ket formalism with a rigorous mathematical basis, see \cite{Roberts1966,Antoine1969}. Our Fr\'echet spaces will always be over $\R$ or over $\C$, and by \emph{morphism of Fr\'{e}chet spaces} we mean a continuous linear map.

 \begin{definition}\label{def:RiggedHilbertSpace}
 A \emph{rigged Hilbert space}  is a Fr\'echet space equipped with a continuous (sesquilinear) inner product.
 A \emph{morphism of rigged Hilbert spaces} is simply a morphism of Fr\'echet spaces.
 A morphism of rigged Hilbert spaces is called \emph{bounded}/\emph{isometric} if it is bounded/isometric with respect to the  inner products.
\end{definition}

Explicitly, a rigged Hilbert space is a pair $(F, \langle \cdot, \cdot \rangle)$; however, we shall usually suppress the inner product $\langle \cdot, \cdot \rangle$ from the notation.
 Given a rigged Hilbert space $F$ one obtains a Hilbert space $F^{\langle \cdot, \cdot \rangle}$ by taking the completion with respect to the inner product.
 A  morphism $f: F_{1} \rightarrow F_{2}$ can be seen as an unbounded linear map $F_{1}^{\langle \cdot, \cdot \rangle} \rightarrow F_{2}^{\langle \cdot, \cdot \rangle}$ with  domain $F_1$, and when  $f$ is bounded/isometric, that map is bounded/isometric. 
 In this way, a rigged Hilbert space is the same thing as pair $(F,H)$, where $H$ is a Hilbert space, and $F$ is a Fr\'echet space equipped with a continuous linear injection $\iota: F \rightarrow H$ such that $\iota(F)$ is dense in $H$.
 Likewise, a morphism $(F_{1},H_{1}) \rightarrow (F_{2},H_{2})$ of rigged Hilbert spaces is an unbounded linear map $H_{1} \rightarrow H_{2}$ with domain $F_{1}$, that restricts to a morphism of Fr\'{e}chet spaces $F_{1} \rightarrow F_{2}$.

\begin{example}\label{ex:SchwartzFrechet}
The Schwartz space of the real line, $\mc{S}(\R)$, is a Fr\'echet space, and dense in $L^{2}(\R)$. The restriction of the $L^{2}$-inner product is continuous on $\mc{S}(\R)$, and hence turns $\mc{S}(\R)$ into a rigged Hilbert space.
\end{example}
One of the reasons we work with rigged Hilbert spaces is that they give good control over differential operators.
Consider \cref{ex:SchwartzFrechet}, there we see that differential operators are extremely well-behaved on $\mc{S}(\R)$, even though they behave very poorly on $L^{2}(\R)$, i.e.~they are merely densely defined and unbounded.

Let $\mc{M}$ be a Fr\'echet manifold. We shall consider several different notions of bundles over $\mc{M}$.
We start by recalling the definition of Fr\'echet principal bundles from \cite[Definition 4.6.5]{Hamilton82}.
Let $\mc{G}$ be a Fr\'echet Lie group.

\begin{definition}
 A \emph{Fr\'echet principal $\mc{G}$-bundle} over $\mc{M}$ is a Fr\'echet manifold $\mathcal{P}$, equipped with
a smooth surjection $\pi: \mathcal{P} \rightarrow \mc{M}$ and a  fibre-preserving  right action of $\mc{G}$ on $\mathcal{P}$, such that for each $x \in \mc{M}$ there exists an open neighbourhood $U$ of $x$, and a fibre-preserving diffeomorphism $\Phi_U: \mathcal{P}|_{U}\defeq \pi^{-1}(U) \rightarrow U \times \mc{G}$, which intertwines the  action of $\mc{G}$ on $\mathcal{P}$ with  multiplication from the right.
\end{definition}

Morphisms between Fr\'echet principal bundles are defined in the usual way, and 
it is clear that Fr\'echet principal bundles can be pulled back along smooth maps of Fr\'echet manifolds. 
In the following \lcnamecref{lem:bundleext} (required in \cref{sec:AnotherBundle}) we note that Fr\'echet principal bundles can be extended along  Fr\'echet Lie group homomorphisms, just like ordinary, finite-dimensional principal bundles. 

\begin{lemma}
\label{lem:bundleext}
Let $\pi:\mathcal{P} \to \mathcal{M}$ be a Fr\'echet principal $\mathcal{G}$-bundle over $\mathcal{M}$, and let $\phi:\mathcal{G} \to \mathcal{H}$ be a homomorphism of Fr\'echet Lie groups. Then, the quotient $\mathcal{P} \times_{\mathcal{G}} \mathcal{H} \defeq(\mathcal{P} \times \mathcal{H})/\mathcal{G}$, where $\mathcal{G}$ acts by $(p,h) g \defeq (pg,\phi(g)^{-1}h)$, carries a unique Fr\'echet manifold structure for which the projection $\mathcal{P} \times \mathcal{H} \to (\mathcal{P}\times \mathcal{H})/\mathcal{G}$ is a submersion. Moreover, equipped with the $\mathcal{H}$-action $(p,h)h'\defeq (p,hh')$ and the projection $(p,h) \mapsto \pi(p)$, $\mathcal{P} \times_{\mathcal{G}} \mathcal{H}$ is a Fr\'echet principal $\mathcal{H}$-bundle over $\mathcal{M}$.  
\end{lemma}

Next we let $E$ be a Fr\'echet space, and define the notion of a Fr\'echet vector bundle with fibre $E$, following \cite[Definition 4.3.1]{Hamilton82}.
\begin{definition}
    A \emph{Fr\'echet vector bundle} over $\mc{M}$ with fibre $E$ is a Fr\'echet manifold $\mc{E}$ equipped with  a smooth surjection $\pi : \mc{E} \rightarrow \mc{M}$ and vector space structures on the fibres $\mc{E}_{x} \defeq \pi^{-1}(x)$,  for each $x \in \mc{M}$, such that for each $x \in \mc{M}$ there exists an open neighbourhood $U$ of $x$, and a diffeomorphism $\Phi_{U}:\mc{E}|_{U} \defeq \pi^{-1}(U) \rightarrow U \times E$ which is fibre-preserving and fibrewise linear. Morphisms of Fr\'echet vector bundles are smooth maps that are fibre-preserving and fibrewise linear. 
\end{definition}

As usual, the maps $\Phi_{U}$ are called \emph{local trivializations}.  
Using a local trivialization $\Phi_U$ defined around $x$, we equip the fibre $\mc{E}_{x}$ with the structure of Fr\'echet space, such that $\Phi_{U}|_{x}: \mc{E}_{x} \rightarrow E$ is an isomorphism of Fr\'echet spaces.
This Fr\'echet space structure is independent of the choice of $\Phi_U$.
The notions of a morphism between Fr\'echet vector bundles and of the pullback of a Fr\'echet vector bundle are the usual ones.

\begin{example}
The tangent bundle of a Fr\'echet manifold is a Fr\'echet vector bundle over that manifold, see \cite[Example 4.3.2]{Hamilton82}. The cotangent bundle, however, is not. 
\end{example}
Next, we discuss the operation of taking associated bundles in the setting of  Fr\'echet spaces.
Since this is the basis of our central definitions (\cref{def:fockbundle,def:CliffordBundle}), and we have been unable to find a reference, we will be explicit.
Let $\mathcal{P}$ be a Fr\'echet principal $\mc{G}$-bundle over $\mathcal{M}$, and suppose that $\mc{G}$ acts on a Fr\'echet space $E$ such that the map $\mc{G} \times E \rightarrow E$ is smooth.
We consider the quotient $(\mathcal{P} \times E)/\mc{G}$, where $\mathcal{G}$ acts on $\mathcal{P} \times E$ by $(p,v)g\defeq(pg,g^{-1}v)$, equipped with the projection map to $\mc{M}$ inherited from $\mathcal{P}$.
For each $x \in \mc{M}$ we equip $((\mathcal{P} \times E)/\mc{G})_{x}$ with the structure of vector space as follows.
 The map
    \begin{align*}
     ((\mathcal{P} \times E)/\mc{G})_{x} \times ((\mathcal{P} \times E)/\mc{G})_{x} \rightarrow ((\mathcal{P} \times E)/\mc{G})_{x}, \quad ([p,v],[p,w]) \mapsto [p,v+w],
    \end{align*}
    defines an addition. Scalar multiplication is defined by $(\lambda, [p,v]) \mapsto [p, \lambda v]$ for all $\lambda \in \C$.

\begin{lemma}\label{lem:AssociatedFrechetBundle}
    The quotient $(\mathcal{P} \times E)/\mc{G}$ has a unique structure of Fr\'echet manifold, such that the map $\mathcal{P} \times E \rightarrow (\mathcal{P} \times E)/\mc{G}$ is a surjective submersion.
    When the fibres over $\mathcal{M}$ are equipped with the vector space structure as above, the Fr\'echet manifold $(\mathcal{P} \times E)/\mc{G}$ is a Fr\'echet vector bundle over $\mc{M}$.
\end{lemma}
\begin{proof}
    There is at most one structure of Fr\'echet manifold on $(\mathcal{P} \times E)/\mc{G}$ such that the projection $\mathcal{P} \times E \rightarrow (\mathcal{P} \times E)/\mc{G}$ is a surjective submersion (see \cite[Lemma 1.9]{Glockner2015}).
    Hence, all we need to do is to construct one.
    Let us write $F$ for the Fr\'echet space on which $\mc{M}$ is modelled.
    Let $\{ U_{\alpha}, \psi_{\alpha} \}_{\alpha \in I}$ be an atlas of $\mc{M}$ that trivializes $\mathcal{P}$. Let $\alpha \in I$ be arbitrary, let $\ph_{\alpha} : \mathcal{P}|_{U_{\alpha}} \rightarrow U_{\alpha} \times \mc{G}$ be the corresponding trivialization.
    Let us write $g_{\alpha}: \mathcal{P}|_{U_{\alpha}} \rightarrow \mc{G}$ for $\ph_{\alpha}$ followed by projection onto $\mc{G}$.
    We have
        \begin{equation*}
                \left( ( \mathcal{P} \times E) / \mc{G})\right)|_{U_{\alpha}} = \{ [p,v] \mid p \in \mathcal{P}|_{U_{\alpha}}, v \in E \}.
        \end{equation*}
        Define the local chart
        \begin{align}\label{eq:LocalTrivializations}
                \Psi_{\alpha}: \left( ( \mathcal{P} \times E) / \mc{G})\right)|_{U_{\alpha}}  \rightarrow F \times E, \quad
                [p,v]  \mapsto (\psi_{\alpha}(\pi(p)),\, g_{\alpha}(p) v ).
        \end{align}
        Using the right $\mc{G}$-equivariance of the map $g_{\alpha}$, it is easy to check that this map is well-defined, fibrewise linear and bijective.
        A standard argument then proves that the transition functions are smooth.
Finally, a computation shows that the projection map $\mathcal{P} \times E \rightarrow (\mathcal{P} \times E)/\mc{G}$ locally becomes the projection map $F \times \mathcal{G} \times E \to F \times E$ onto the first and third component; hence, it is a submersion.
\end{proof}

The following result shows that taking associated bundles is compatible with Fr\'echet Lie group homomorphisms; the proof is completely analogous to the finite-dimensional case.

\begin{lemma}\label{prop:InvarianceUnderLifts}
Let $\mathcal{G}_1$ and $\mathcal{G}_2$ be Fr\'echet Lie groups acting on a Fr\'echet space $E$ in such a way that the maps $\mathcal{G}_i \times E \to E$ are smooth for $i=1,2$.   
 Let $q: \mc{G}_{1} \rightarrow \mc{G}_{2}$ be a smooth group homomorphism such that $g v = q(g) v$ for all $g \in \mc{G}_{1}$ and all $v \in E$. Let $\mathcal{P}_i$ be a principal $\mathcal{G}_i$-bundle over $\mathcal{M}$, for $i=1,2$, and let $\hat q: \mathcal{P}_1 \to \mathcal{P}_2$ be a $q$-equivariant smooth bundle morphism. 
  Then, there exists a unique isomorphism of Fr\'{e}chet vector bundles
  \begin{equation*}
   f: (\mathcal{P}_{1} \times E)/\mc{G}_{1} \rightarrow (\mathcal{P}_{2} \times E)/\mc{G}_{2}
  \end{equation*}
  such that $f([p,v]) = [\hat{q}(p),v]$ for all $p \in \mathcal{P}_{1}$ and all $v \in E$.
 \end{lemma}

Now we are in position to discuss bundles of rigged Hilbert spaces.

\begin{definition}\label{def:HilbertBundle}
Let $E$ be a rigged Hilbert space. A \emph{rigged Hilbert space bundle} over $\mc{M}$ with fibre $E$ is a Fr\'echet vector bundle  $\mc{E} \to \mathcal{M}$  with fibre $E$ equipped with a map $g:\mc{E} \times_{\mathcal{M}} \mc{E} \rightarrow \C$, such that the following conditions hold for each $x \in \mc{M}$:
        \begin{itemize}
         \item The map $g_{x}: \mc{E}_{x} \times \mc{E}_{x} \rightarrow \C$ is an inner product.
         \item There exists an open neighbourhood $U \subset \mc{M}$ of $x$ and a diffeomorphism $\Phi_{U}: \mc{E}|_{U} \rightarrow U \times E$, which is fibre-preserving, fibrewise linear, and fibrewise an isometry.
        \end{itemize}
        The map $g$ is called the \emph{Hermitian metric} of $\mc{E}$.
        The map $\Phi_{U}$ is called a \emph{local trivialization of $\mc{E}$ over $U$}.
        A \emph{morphism of rigged Hilbert space bundles} is simply a morphism of the underlying Fr\'{e}chet vector bundles. A morphism is called \emph{locally bounded} if it is bounded with respect to the Hermitian metrics over small enough open subsets of $\mathcal{M}$, and it is called \emph{isometric} if it is isometric with respect to the Hermitian metrics.
\end{definition}

\begin{remark}\label{rem:HermitianMetricIsSmooth}
 Let $\mc{E}$ be a rigged Hilbert space bundle over $\mc{M}$ with fibre $E$ and with Hermitian metric $g$. Let $\Phi_{U}:~ \mc{E}|_{U} \rightarrow U \times E$ be a local trivialization, in particular fibrewise linear and an isometry. The diagram
 \begin{equation*}
  \xymatrix{
   \mc{E}|_{U} \times_{U} \mc{E}|_{U} \ar[r]^-{g} \ar[d]_{\Phi_{U} \times \Phi_{U}} & \C \\
   U \times E \times E \ar[ur] &
  }
 \end{equation*}
 then commutes, which proves that $g$ is smooth. In particular, it is fibrewise continuous; this means that each fibre $\mc{E}_{x}$ is a rigged Hilbert space, and each local trivialization is fibrewise an isometric isomorphism of rigged Hilbert spaces.
 Finally, a (locally bounded/isometric) morphism of rigged Hilbert space bundles is fibrewise a (bounded/isometric) morphism of rigged Hilbert spaces.
 \end{remark}

\begin{remark}
 If the rigged Hilbert space $E$ is finite-dimensional, then a rigged Hilbert space bundle with fibre $E$ is the same as a Hermitian vector bundle.  
 If the base manifold $\mc{M}$ is a point, then a rigged Hilbert space bundle over $\mc{M}$ is a rigged Hilbert space. If $E$ is a rigged Hilbert space, then the trivial bundle $\mathcal{M} \times E$ is a rigged Hilbert space bundle.
\end{remark}

\begin{remark}
\label{re:presheafstack}
Rigged Hilbert space bundles over $\mathcal{M}$ form a category $\VBdl(\mathcal{M})$, whose morphisms we define to be the locally bounded morphisms.
Together with the natural pullback operation, the assignment
\begin{equation*}
\mathcal{M} \mapsto \VBdl(\mathcal{M})
\end{equation*}
is a presheaf of categories on the category of Fr\'echet manifolds; this just means that one can consistently pull back rigged Hilbert space bundles along smooth maps. In \cref{sec:twistedriggedhb} it will be important at one point that the presheaf  $\VBdl$ is a \emph{sheaf} of categories, a.k.a. a stack (when the gluing maps are  isometric isomorphisms). This means that one can glue locally defined rigged Hilbert space bundle along isometric isomorphisms on the overlaps, if these satisfy the cocycle condition on triple overlaps. This is just a small modification of the fact that Fr\'echet vector bundles form a stack, which is well known, though we have not been able to find a reference for this. 
\end{remark}

 Just as one can pass from a rigged Hilbert space $E$ to the Hilbert space $E^{\left \langle \cdot,\cdot  \right \rangle}$, one can pass from a rigged Hilbert space bundle to a continuous Hilbert space bundle. Here, a continuous Hilbert space bundle $\mathcal{H}$ with fibre $H$ means that it has  continuous local trivializations
\begin{equation*}
\mathcal{H}|_U \cong U \times H\text{,}
\end{equation*} 
which are fibrewise isometric isomorphisms. 
Equivalently, its transition functions are continuous with values in $\U(\mathcal{H})$, equipped with the strong operator topology. We remark that the strong operator topology coincides with the compact open topology on $\U(\mathcal{H})$, and that $\U(\mathcal{H})$ is contractible in these topologies \cite{Schottenloher,Espinoza}.
 
 \begin{lemma}\label{rem:ContinuousHilbertBundle}
Let $\mc{E} \rightarrow \mc{M}$ be a rigged Hilbert space bundle.
 Then, the fibrewise completion of $\mc{E}$, denoted $\mc{E}^{g}$, has a unique structure of  locally trivial strongly continuous Hilbert space bundle over $\mc{M}$, such that the inclusion $\mc{E} \rightarrow \mc{E}^{g}$ is continuous. Any locally bounded morphism $\phi: \mathcal{E}_1 \to \mathcal{E}_2$ of rigged Hilbert space bundles extends uniquely to a continuous morphism $\mathcal{E}_1^{g} \to \mathcal{E}_2^{g}$ of continuous Hilbert space bundles.

\end{lemma}

\begin{proof}
It is straightforward to turn a local trivialization $\Phi_U$ of $\mathcal{E}$ into a fibre-preserving, fibrewise linear, and fibrewise isometric bijection $\Phi_{U}^{g}: \mc{E}|_{U}^{g} \rightarrow U \times E^{g}$. A standard exercise shows that the corresponding transition functions $U \times E^{g} \to E^{g}$ are continuous. 
\end{proof}

Next we discuss representations of Fr\'echet Lie groups on rigged Hilbert spaces.

\begin{definition}
\label{def:smoothreprhs}
A \emph{smooth representation} of a Fr\'echet Lie group $\mathcal{G}$ on a rigged Hilbert space $E$ is an action  of $\mathcal{G}$ on $E$ by isometric morphisms of rigged Hilbert spaces, such that the map $\mathcal{G} \times E \to E$ is smooth. 
\end{definition}

 Let $\mathcal{P}$ be a Fr\'echet principal $\mc{G}$-bundle over $\mathcal{M}$, and consider a smooth  representation of $\mathcal{G}$ on a rigged Hilbert space $E$.  We define a Hermitian metric $g$ on the associated vector bundle $(\mathcal{P} \times E)/\mc{G}$ as follows.
 Let
 \begin{equation*}
    (x,y) \in (\mathcal{P} \times E)/\mc{G} \times_{\mathcal{M}} (\mathcal{P} \times E)/\mc{G}.
 \end{equation*}
 Then, there exist a $p \in \mathcal{P}$ and $v,w \in E$ such that $x = [p,v]$ and $y = [p,w]$. We set $g(x,y) = \langle v, w \rangle_{E}$. Using the fact that $\mc{G}$ acts by unitary transformations it becomes easy to check that  $g$ is well-defined and that the local trivializations $\Psi_{\alpha}$ of $(\mathcal{P} \times E)/\mathcal{G}$ defined in Equation \cref{eq:LocalTrivializations} are fibrewise unitary.
 Hence, we conclude that the following result holds.

\begin{proposition}\label{lem:AssociatedHilbertBundle}
Let $\mathcal{P}$ be a Fr\'echet principal $\mc{G}$-bundle over $\mathcal{M}$, and consider a smooth  representation of $\mathcal{G}$ on a rigged Hilbert space $E$. 
Then,  the associated bundle $(\mathcal{P} \times E)/\mc{G}$ together with the Hermitian metric constructed above is a rigged Hilbert space bundle with  fibre $E$.
\end{proposition}

Finally, we observe that the Fr\'echet vector bundle isomorphism of \cref{prop:InvarianceUnderLifts} is isometric with respect to the Hermitian metrics on associated rigged Hilbert space bundles, and hence obtain the following result.
 
\begin{proposition}\label{prop:InvarianceUnderLiftsHilbert}
Let $\mathcal{G}_1$ and $\mathcal{G}_2$ be Fr\'echet Lie groups with smooth representations on a rigged Hilbert space $E$.    
 Let $q: \mc{G}_{1} \rightarrow \mc{G}_{2}$ be a smooth group homomorphism such that $g v = q(g) v$ for all $g \in \mc{G}_{1}$ and all $v \in E$. Let $\mathcal{P}_i$ be a principal $\mathcal{G}_i$-bundle over $\mathcal{M}$, for $i=1,2$, and let $\hat q: \mathcal{P}_1 \to \mathcal{P}_2$ be a $q$-equivariant smooth bundle morphism. 
  Then, there exists a unique isometric isomorphism of rigged Hilbert space bundles
  \begin{equation*}
   f: (\mathcal{P}_{1} \times E)/\mc{G}_{1} \rightarrow (\mathcal{P}_{2} \times E)/\mc{G}_{2}
  \end{equation*}
  such that $f([p,v]) = [\hat{  q}(p),v]$ for all $p \in \mathcal{P}_{1}$ and all $v \in E$.
 \end{proposition}

\subsection{Rigged \texorpdfstring{\cstar}{C*}-algebra bundles}
\label{sec:HilbertBundle:2}

In this section we discuss bundles of infinite-dimensional algebras, in close analogy to rigged Hilbert space bundles.
First, we recall that a \emph{Fr\'{e}chet algebra} is a Fr\'{e}chet space, which is at the same time an algebra, with the property that multiplication is continuous (and thus smooth).
A \emph{morphism of Fr\'{e}chet algebras} is a  morphism of the underlying Fr\'{e}chet spaces and at the same time an algebra homomorphism.

\begin{definition}
\label{def:riggedcstaralgebra}
 A \emph{rigged \cstar-algebra} is a Fr\'echet algebra $A$, equipped with a continuous norm $\| \cdot \|: A \rightarrow \R_{\geqslant 0}$ and a continuous complex anti-linear involution $*:A \rightarrow A$ such that its completion  with respect to the norm is a \cstar-algebra.
 A \emph{morphism of rigged \cstar-algebras} is a morphism of Fr\'echet algebras that is bounded with respect to the norms and intertwines the involutions. A morphism of rigged \cstar-algebras is called \emph{isometric}, if it is an isometry with respect to the norms.
\end{definition}

\begin{example}
 We equip $C^{\infty}(S^{1},\C)$ with its usual Fr\'echet structure. Then, equipped with the supremum norm and with pointwise complex conjugation, $C^{\infty}(S^{1},\C)$ becomes a rigged \cstar-algebra.
\end{example}

\begin{remark}
 If $A$ is a rigged \cstar-algebra, then its opposite algebra, $A^{\mathrm{opp}}$, is a rigged \cstar-algebra in a natural way.
\end{remark}

\begin{definition}\label{def:cstarBundle}
Let $A$ be a rigged \cstar-algebra.
 A \emph{rigged \cstar-algebra bundle} over $\mc{M}$ with fibre $A$ is a Fr\'echet vector bundle $\mc{A} \rightarrow \mc{M}$, equipped with the following structure:
 \begin{itemize}
 \item A map $\| \cdot \|: \mc{A} \rightarrow \R_{\geqslant 0}$.
 \item A fibre-preserving map $m: \mc{A} \times_{\mathcal{M}} \mc{A} \rightarrow \mc{A}$.
 \item A fibre-preserving map $*: \mc{A} \rightarrow \mc{A}$.
 \end{itemize}
 such that the following conditions hold for each $x \in \mc{M}$
 \begin{itemize}
    \item The map $\| \cdot \|_{x} : \mc{A}_{x} \rightarrow \R_{\geqslant 0}$ is a norm.
    \item The map $m_{x} : \mc{A}_{x} \times \mc{A}_{x} \rightarrow \mc{A}_{x}$ turns the vector space $\mc{A}_{x}$ into an associative algebra.
    \item The map $*_{x}: \mc{A}_{x} \rightarrow \mc{A}_{x}$ turns the algebra $\mc{A}_{x}$ into a $*$-algebra.
    \item There exists an open neighbourhood $U \subset \mc{M}$ of $x$ and a diffeomorphism $\Phi_{U}: \mc{A}|_{U} \rightarrow U \times A$ which is fibrewise an isometric isomorphism of $*$-algebras.
 \end{itemize}
 The map $\Phi_{U}$ is called a \emph{local trivialization of $\mc{A}$ over $U$}. A \emph{morphism of rigged \cstar-algebra bundles} over $\mathcal{M}$ is a morphism $\phi: \mathcal{A}_1 \to \mathcal{A}_2$ of Fr\'echet vector bundles that is fibrewise a morphism of $\ast$-algebras and locally bounded with respect to the norms.
A morphism is called \emph{isometric} if it is isometric with respect to the norms.\end{definition}
\begin{remark}
 Let $\mc{A}$ be a rigged \cstar-algebra bundle with fibre $A$.
 A standard argument, similar to  \cref{rem:HermitianMetricIsSmooth}, proves that the maps $*$ and $m$ are smooth and that $\| \cdot \|$ is smooth away from the image of the zero section.
 Further, it follows that on each fibre $\mathcal{A}_x$, the involution is continuous and an isometry, and the norm is continuous, submultiplicative, and satisfies the \cstar-identity; in other words, all fibres $\mathcal{A}_x$ are rigged \cstar-algebras. Likewise, an (isometric) morphism of rigged \cstar-algebra bundles restricts on each fibre to an (isometric) morphism of rigged \cstar-algebras.   

It is possible to pass from a rigged \cstar-algebra bundle to a locally trivial continuous \cstar-algebra bundle, just like it is possible to pass from a rigged Hilbert space bundle to a continuous Hilbert space bundle.
Here, a continuous bundle of \cstar-algebras with fibre $A$ has continuous local transition functions $U \times A \rightarrow A$, or equivalently, the induced maps $U \rightarrow \mathrm{Aut}(A)$ are continuous when $\mathrm{Aut}(A)$ is equipped with the strong operator topology.
\end{remark}

This follows in a straightforward way, analogously to \cref{rem:ContinuousHilbertBundle}. 

\begin{lemma}\label{rem:fibrewiseCstar}
 If $\mc{A}$ is a rigged \cstar-algebra bundle over $\mathcal{M}$, then its fibrewise completion with respect to the norm, denoted $\mathcal{A}^{\|\cdot\|}$, carries a unique structure of a continuous bundle of \cstar-algebras, such that the inclusion $\mathcal{A} \to \mathcal{A}^{\|\cdot\|}$ is continuous.  Likewise, any morphism of rigged \cstar-algebra bundles extends uniquely to a continuous morphism of continuous bundles of \cstar-algebras.
\end{lemma}

The next step is to introduce representations of Fr\'echet Lie groups on rigged \cstar-algebras, which leads to a construction of rigged \cstar-algebra bundles we will use later.

\begin{definition}
\label{def:smoothrepcstar}
A \emph{smooth representation} of a Fr\'echet Lie group $\mathcal{G}$ on a rigged \cstar-algebra $A$ is an action of $\mathcal{G}$ on $A$ by isometric morphisms of rigged \cstar-algebras, such that the map $\mathcal{G} \times A \to A$ is smooth. 
\end{definition}

Now, let $\mathcal{P}$ be a Fr\'echet principal $\mc{G}$-bundle, and consider a smooth representation of  $\mc{G}$  on a rigged \cstar-algebra $A$. 
 We consider the Fr\'echet vector bundle $\mc{E} \defeq (\mathcal{P} \times A)/ \mc{G}$. The maps 
\begin{align*}
&\| \cdot \|: \mathcal{P} \times A \rightarrow \R_{\geqslant 0},\quad (p,v) \mapsto \| v\|
\\
&m: (\mathcal{P} \times A) \times_{\pi} (\mathcal{P} \times A) \rightarrow \mc{E}, \quad ([p,v_{1}],[p,v_{2}]) \mapsto [p, v_{1}v_{2}]
\\
&*: \mathcal{P} \times A \rightarrow \mathcal{P} \times A, \quad (p,v) \mapsto (p,v^{*})
\end{align*}
descend to well-defined maps on $\mathcal{E}$, providing on $\mathcal{E}$  the data for a rigged \cstar -algebra bundle. It is a routine exercise to check that these maps satisfy the properties laid out in \cref{def:cstarBundle}.
 Moreover, it follows from a routine verification that the trivializations $\Psi_{\alpha}$ from the proof of \cref{lem:AssociatedFrechetBundle} are fibrewise isometric $*$-homomorphisms.
 We conclude that the following result holds.

\begin{proposition}\label{lem:RiggedCStarBundle}
Let $\mathcal{P}$ be a Fr\'echet principal $\mc{G}$-bundle, and consider a smooth representation of  $\mc{G}$  on a rigged \cstar-algebra $A$. 
  Then, the associated bundle  $(\mathcal{P} \times A)/\mc{G}$ is a rigged \cstar-algebra bundle over $\mc{M}$ with fibre $A$.
\end{proposition}

One may now check that the Fr\'echet vector bundle isomorphism of
\cref{prop:InvarianceUnderLifts} is isometric with respect to the norms on associated rigged \cstar-algebra bundles, and that it is a $\ast$-homomorphism. Thus, the following result holds.
 
\begin{proposition}\label{prop:InvarianceUnderLiftsCStar}
Let $\mathcal{G}_1$ and $\mathcal{G}_2$ be Fr\'echet Lie groups with smooth representations on a rigged \cstar-algebra  $A$.    
 Let $q: \mc{G}_{1} \rightarrow \mc{G}_{2}$ be a smooth group homomorphism such that $g a = q(g) a$ for all $g \in \mc{G}_{1}$ and all $a \in A$. Let $\mathcal{P}_i$ be a principal $\mathcal{G}_i$-bundle over $\mathcal{M}$, for $i=1,2$, and let $\hat q: \mathcal{P}_1 \to \mathcal{P}_2$ be a $q$-equivariant smooth bundle morphism. 
  Then, there exists a unique isometric isomorphism of rigged \cstar-algebra bundles
  \begin{equation*}
   f: (\mathcal{P}_{1} \times A)/\mc{G}_{1} \rightarrow (\mathcal{P}_{2} \times A)/\mc{G}_{2}
  \end{equation*}
  such that $f([p,a]) = [\hat{q}(p),a]$ for all $p \in \mathcal{P}_{1}$ and all $a \in A$.
 \end{proposition}

Finally,  we  discuss modules for rigged \cstar-algebras,  and module bundles for rigged \cstar-algebra bundles.

\begin{definition}\label{def:RiggedRepresentation}
Let $A$ be a rigged \cstar-algebra. A \emph{rigged $A$-module} is a rigged Hilbert space $E$ together with a  representation $\rho$
 of (the underlying algebra of) $A$ on (the underlying vector space of) $E$, such that the map $\rho:A \times E \to E$ is smooth and the following conditions hold for all $a \in A$ and all $v,w \in E$
 \begin{equation}\label{eq:BoundednessCriteria}
  \langle \rho(a,v), \rho(a,v) \rangle \leqslant \| a \|^{2} \langle v, v \rangle, \quad \text{ and } \quad \langle \rho(a,v),w \rangle = \langle v, \rho(a^{*},w) \rangle.
 \end{equation}
 
\end{definition}

\begin{remark}
\label{re:repriggedcstar}
In order to explain the conditions \cref{eq:BoundednessCriteria}, let $E$ be a rigged $A$-module with representation $\rho$.
For each $a \in A$ we write $\rho_a:E \rightarrow E$ for the morphism of rigged Hilbert spaces defined by $\rho_a(v) \defeq \rho(a,v)$. The inequality in \cref{eq:BoundednessCriteria} then implies that $\rho_a$ is   bounded. 
 In fact, more is true, a standard argument using  \cref{eq:BoundednessCriteria} shows that  $\rho$ extends to a representation of the \cstar-algebra $A^{\| \cdot \|}$ on the Hilbert space $E^{\langle \cdot, \cdot \rangle}$.
 We thus obtain a morphism of \cstar-algebras $\rho^{\vee}: A^{\| \cdot \|} \rightarrow \mc{B}(E^{\langle \cdot, \cdot \rangle})$.
Conversely, suppose we are given a representation of the \cstar-algebra $A^{\| \cdot \|}$ on the Hilbert space $E^{\langle \cdot, \cdot \rangle}$ that restricts to a continuous map $A \times E \to E$. Then, $E$ is automatically a rigged $A$-module, i.e.~the conditions \cref{eq:BoundednessCriteria} are automatically satisfied. 
\end{remark}

\begin{remark}
\label{re:inducedcstarrep}
Rigged modules can be induced along morphisms of rigged \cstar-algebras. More precisely, if $E$ is a rigged module for a \cstar-algebra $A'$, representation $\rho':A' \times E \to E$,  and $f:A \to A'$ is a morphism of rigged \cstar-algebras, then $\rho(a,v) := \rho'(f(a),v)$ turns $E$ into a rigged $A$-module.

\end{remark}

\begin{remark}\label{ex:OppositeRep}
 Let $A$ be a rigged \cstar-algebra, and $E$ a rigged $A$-module.
 The inner product on $E$ gives us a complex anti-linear injection $\iota: E \rightarrow E^{*}$ mapping $E$ into its continuous linear dual $E^{*}$.
 Denote the image of $\iota$ by $E^{\sharp}$.
 We turn $E^{\sharp}$ into a rigged Fr\'{e}chet space using the identification with $E$.
 If $\rho$ denotes the representation of $E$, the map 
 \begin{align*}
  \rho^{\sharp}: A^{\mathrm{opp}} \times E^{\sharp} \rightarrow E^{\sharp},
  (a,\ph) &\mapsto \ph \circ \rho^{\vee}(a),
 \end{align*}
 is a  representation of  $A^{\mathrm{opp}}$ on  $E^{\sharp}$, turning $E^{\sharp}$ into a rigged $A^{\opp}$-module.
\end{remark}

We recall that a central extension of Fr\'echet Lie groups is a sequence
$\U(1) \to \widetilde{\mathcal{G}} \to \mathcal{G}$
of Fr\'echet Lie groups and smooth group homomorphisms, such that $\widetilde{\mathcal{G}}$ is a principal $\U(1)$-bundle over $\mathcal{G}$.

\begin{definition}
\label{def:repofce}
Let $E$ be a rigged module for a rigged \cstar-algebra $A$. Then, a \emph{smooth representation of a central extension $\U(1) \to \widetilde{\mathcal{G}} \stackrel{q}{\to} \mathcal{G}$ on $E$} consists of smooth representations of $\mathcal{G}$ on $A$ and $\widetilde{\mathcal{G}}$ on $E$ such that the following two conditions are satisfied:
\begin{enumerate}[(a)]

\item 
The central subgroup $\U(1) \subset \widetilde{\mathcal{G}}$ acts by scalar multiplication.

\item
For all $\tilde{g} \in \widetilde{\mathcal{G}}$, $a\in A$, and $v\in E$ we have:
$\tilde g \cdot (a\cdot v) = (q(\tilde g) \cdot a)\cdot(\tilde g \cdot v)$.

\end{enumerate}
\end{definition}

The main example of a smooth representation of a central extension on a rigged module is the rigged version of the Fock space representation we introduce in \cref{sec:SmoothRepresentations}, see \cref{thm:repofimp}. In the general context, \cref{def:repofce} will appear later in \cref{prop:InducedBundleRep} and in the twisted context in \cref{lem:reptwistedmb}. 
Next we introduce the bundle version of a rigged module.

\begin{definition}
\label{def:repcstarbundle}
Let $A$ be a rigged \cstar-algebra and let $E$ be a rigged $A$-module, with representation $\rho_0$. Let $\mathcal{A}$ be a rigged \cstar-algebra bundle over $\mathcal{M}$ with typical fibre $A$. A \emph{rigged $\mathcal{A}$-module bundle with typical fibre $E$} is a rigged Hilbert space bundle $\mathcal{E}$ with typical fibre $E$, and  a fibre-preserving map
 \begin{equation*}
  \rho: \mc{A} \times_{\mathcal{M}} \mc{E} \rightarrow \mc{E}
 \end{equation*}
 with the property that for each $x \in \mc{M}$, there exist an open neighbourhood $U$ of $x$ and local trivializations $\Phi$ of $\mc{A}$ and $\Psi$ of $\mc{E}$ over $U$ such that the following diagram commutes:
 \begin{equation*}
  \xymatrix{
  \mc{A}_U \times_{\mathcal{M}} \mc{E}_{U} \ar[r]^-{\rho} \ar[d]_{\Phi \times \Psi} & \mc{E}|_{U} \ar[d]^{\Psi} \\
  (A \times U)\times_U (E \times U) \ar[d]_{\pr_A \times \pr_E} & E \times U \ar[d]^{\pr_E}\\A \times E \ar[r]_-{\rho_0} & E\text{.} 
  }
 \end{equation*}
\end{definition}

\begin{remark}
\label{re:repcstarbundle}
Let $\mc{E}$ be a rigged $\mathcal{A}$-module.  It is straightforward to see that its map $\rho: \mathcal{A} \times_{\mathcal{M}} \mathcal{E} \to \mathcal{E}$ is  a smooth morphism of Fr\'{e}chet vector bundles, and that its restriction over every point $x\in U$ turns $\mathcal{E}_x$ into a rigged $\mathcal{A}_x$-module. 
\end{remark}

A morphism between rigged $\mathcal{A}$-module bundles is a morphism of rigged Hilbert space bundles that commutes with the representations. We denote by $\VBdl^{\mathcal{A}}(\mathcal{M})$ the category of rigged $\mathcal{A}$-module bundles, with morphisms  the locally bounded ones.

We recall that the fibrewise norm completion $\mathcal{A}^{\|\cdot\|}$ of $\mathcal{A}$ is a continuous bundle of \cstar-algebras (see \cref{rem:fibrewiseCstar}), and that the fibrewise Hilbert completion $\mathcal{E}^{g}$ of $\mathcal{E}$ is a continuous Hilbert space bundle (see \cref{rem:ContinuousHilbertBundle}). \cref{re:repcstarbundle,re:repriggedcstar} provide a fibrewise representation of the \cstar-algebra $\mathcal{A}_x^{\|\cdot\|}$ on the Hilbert space $\mathcal{E}_x^g$. Again by \cref{re:repriggedcstar}, the rigged $A$-module $E$ extends to a representation of $A^{\|\cdot\|}$ on $E^{\left \langle  \cdot,\cdot \right \rangle}$. We note the following result, which uses all this structure and is straightforward to prove.

\begin{lemma}
\label{lem:continuous smoothrep}
Let $\mathcal{A}$ be a rigged \cstar-algebra bundle over $\mathcal{M}$ and let $\mathcal{E}$ be a rigged $\mathcal{A}$-module. Then, the fibrewise representations of the \cstar-algebras $\mathcal{A}^{\|\cdot\|}_x$ on the Hilbert spaces $\mathcal{E}_x^g$ form a continuous bundle homomorphism $\mathcal{A}^{\|\cdot \|} \times_{\mathcal{M}} \mathcal{E}^{g} \to \mathcal{E}^{g}$ which satisfies  a local triviality condition like in \cref{def:repcstarbundle} with typical fibre the representation of $A^{\|\cdot\|}$ on $E^{\left \langle  \cdot,\cdot \right \rangle}$.
\end{lemma}

In the following we explain an associated bundle construction for rigged module bundles.
This method is an important tool in \cref{sec:SpinorBundleOnLoopSpaceI}, where we use it to obtain the Clifford multiplication on Fock bundles.
We recall the following terminology. 

\begin{definition}
\label{def:lifts}
If $\U(1) \to \widetilde{\mathcal{G}} \stackrel{q}{\to} \mathcal{G}$ is a central extension of Fr\'echet Lie groups, and $\mathcal{P}$ is a principal $\mathcal{G}$-bundle over $\mathcal{M}$, then a \emph{lift} of the structure group of $\mathcal{P}$ to $\widetilde{\mathcal{G}}$ is a principal $\widetilde{\mathcal{G}}$-bundle $\widetilde{\mathcal{P}}$ over $\mathcal{M}$ together with a smooth  bundle map $\phi:\widetilde{\mathcal{P}} \to \mathcal{P}$ that  is $q$-equivariant, i.e.,  $\phi(\tilde p \cdot \tilde g) = \phi(\tilde p)\cdot q(\tilde g)$ for all $\tilde p \in \widetilde{\mathcal{P}}$ and $\tilde g \in \widetilde{\mathcal{G}}$.
\end{definition}  

\begin{proposition}\label{prop:InducedBundleRep}
Let $\U(1) \to \widetilde{\mathcal{G}} \stackrel{q}{\to} \mathcal{G}$ be a central extension of Fr\'echet Lie groups, equipped with a smooth representation on a rigged module $E$ for a rigged \cstar-algebra $A$. Let $\mathcal{P}$ be a principal $\mathcal{G}$-bundle over $\mathcal{M}$ and let $\phi:\widetilde{\mathcal{P}} \to \mathcal{P}$ be a lift of $\mc{G}$ to $\widetilde{\mathcal{G}}$.
Let  $\mc{A} \defeq (\mathcal{P} \times A)/\mc{G}$ and $\mc{E} \defeq (\widetilde{\mathcal{P}} \times E)/\widetilde{\mc{G}}$ be the associated bundles.
Then, there is a unique map $\rho: \mc{A} \times_{\mc{M}} \mc{E} \rightarrow \mc{E}$ satisfying   
 \begin{equation}\label{eq:FibreWiseRep}
  \rho([\phi(\tilde p),a],[\tilde p,v]) = [\tilde p, av]
 \end{equation}
 for all $\tilde p \in \widetilde{\mathcal{P}}$, all $a \in A$ and $v \in E$.
Moreover, $\rho$ equips $\mathcal{E}$ with the structure of a rigged $\mathcal{A}$-module bundle with typical fibre $E$. 
\end{proposition}

\begin{proof}
It is clear that \cref{eq:FibreWiseRep}  determines $\rho$ uniquely. In order to define $\rho$ as a map by \cref{eq:FibreWiseRep}, we infer from condition (b) of \cref{def:repofce} that
\begin{equation*}
\rho([\phi(\tilde p)g,g^{-1}a],[\tilde p\tilde g,\tilde g^{-1}v])=[\tilde p \tilde g,(q(\tilde g)^{-1}a)\cdot (\tilde g^{-1}v)]=[\tilde p \tilde g,q(\tilde g)^{-1}\cdot (av)]
\end{equation*}
holds for all $\tilde p\in \widetilde{\mathcal{P}}$, $\tilde g\in \widetilde{\mathcal{G}}$, $v\in E$ and $a\in A$.  This implies that defining $\rho$ by \cref{eq:FibreWiseRep} does not depend on the choice of representatives.
Now let $U\subset \mathcal{M}$ be an open set over which $\widetilde{\mathcal{P}}$ trivializes and let $\tilde \sigma: U \rightarrow \widetilde{\mathcal{P}}$ be a local section. Then, $\sigma\defeq \phi\circ \tilde\sigma$ is a local section of $\mathcal{P}$. 
 Let $\Phi$ and $\Psi$ be the corresponding local trivializations of $\mc{A}$ and $\mc{E}$ over $U$ respectively.
 We have $\mathrm{pr}_{A} (\Phi ([\sigma(x),a])) = a$ and $\mathrm{pr}_{E} (\Psi ([\sigma(x),v])) = v$ for all $x \in U$, all $a \in A$ and all $v \in E$.
 It then immediately follows that the diagram in \cref{def:repcstarbundle} commutes. We remark that condition (a) of \cref{def:repofce} has not been used here; it will become relevant later. 
\end{proof}

\subsection{Twists and twisted rigged Hilbert space bundles}

\label{sec:twistedriggedhb}

Twisted bundles will be used in \cref{sec:twisted} of this article, and the  present section establishes the basic definitions and results about twisted bundles in the infinite-dimensional rigged setting. 
A suitable   framework to discuss twisted vector bundles is provided by the theory of  bundle gerbes initiated in \cite{Murray1996}, also see \cite{johnson1,gomi3,waldorf13,Waldorfb}.  In this framework, bundle gerbes represent the twist. We recall this theory here, and generalize it at the same time to twisted rigged Hilbert space bundles.

In order to generalize the usual definition of a bundle gerbe to Fr\'echet manifolds, one  has to note two facts. Firstly, fibre products of surjective submersions exist (they are closed submanifolds of the direct product \cite[Thm. 4.4.10]{Hamilton82}). Secondly, tensor products between principal Fr\'echet $\U(1)$-bundles exist and are defined in the usual way. 

\begin{definition} 
A \emph{bundle gerbe} $\mathscr{G}=(\mathcal{Y},\mathcal{Q},\mu)$ over a Fr\'echet manifold $\mathcal{M}$ consists of a smooth surjective submersion $\pi:\mathcal{Y} \to \mathcal{M}$, of a Fr\'echet principal $\U(1)$-bundle $\mathcal{Q}$ over the 2-fold fibre product $\mathcal{Y}^{[2]}\defeq \mathcal{Y} \times_\mathcal{M} \mathcal{Y}$, and of a bundle isomorphism $\mu: \pr_{12}^{*}\mathcal{Q} \otimes \pr_{23}^{*}\mathcal{Q} \to \pr_{13}^{*}\mathcal{Q}$ over the 3-fold fibre product $\mathcal{Y}^{[3]}$, satisfying the following associativity condition over $\mathcal{Y}^{[4]}$,
\begin{equation*}
\mu(q_{12} \otimes \mu(q_{23} \otimes q_{34}) ) = \mu(\mu(q_{12} \otimes q_{23}) \otimes q_{34})
\end{equation*}
for all $q_{ij}\in \mathcal{Q}$ for which this expression is defined.
 A \emph{trivialization} of $\mathscr{G}$ is a  Fr\'echet principal $\U(1)$-bundle $\mathcal{T}$ over $\mathcal{Y}$ together with a bundle isomorphism $\tau: \mathcal{Q} \otimes \pr_2^{*}\mathcal{T} \to \pr_1^{*}\mathcal{T}$ over $\mathcal{Y}^{[2]}$ that is compatible with $\mu$ in the sense that 
\begin{equation*}
\tau(\mu(q_{12} \otimes q_{23}) \otimes t)=\tau(q_{12} \otimes \tau(q_{23} \otimes t))
\end{equation*}
for all $q_{ij}\in \mathcal{Q}$ and $t\in \mathcal{T}$ for which this expression is defined.
\end{definition}

Trivializations of a bundle gerbe $\mathscr{G}$ form a category, $\Triv(\mathscr{G})$, whose morphisms are bundle morphisms $\mathcal{T} \to \mathcal{T}'$ that are compatible with the isomorphisms $\tau$ and $\tau'$ in the obvious sense.

One source of examples of bundle gerbes are lifting problems. Suppose $\U(1) \to \widetilde {\mathcal{G}} \stackrel{q}{\to} \mathcal{G}$ is a central extension of Fr\'echet Lie groups and $\mathcal{P}$ is a principal $\mathcal{G}$-bundle over $\mathcal{M}$. Lifts of the structure group of $\mathcal{P}$ to $\widetilde{\mathcal{G}}$ (see \cref{def:lifts}) form a  category $\Lift(\mathcal{P})$ whose morphisms are bundle morphisms $\widetilde{\mathcal{P}}\to \widetilde{\mathcal{P}}'$ that preserve the maps to $\mathcal{P}$. Associated to any lifting problem is the \emph{lifting  gerbe} $\mathscr{L}_\mathcal{P}$. Its surjective submersion is the bundle projection $\mathcal{P} \to \mathcal{M}$. The 2-fold fibre product comes equipped with a smooth map $\delta:\mathcal{P}^{[2]} \to \mathcal{G}$ defined by $p'=p\delta(p,p')$. Now, $\mathcal{Q} \defeq\delta^{*}\widetilde {\mathcal{G}}$, where $\widetilde {\mathcal{G}}$ is considered as a principal $\U(1)$-bundle over $\mathcal{G}$, and the isomorphism $\mu$ is the multiplication of group elements.

As the name implies, a lifting gerbe is closely related to the corresponding lifting problem.
For any lift $\widetilde{\mathcal{P}}$, the map $\phi:\widetilde{\mathcal{P}} \to \mathcal{P}$ has the structure of a principal $\U(1)$-bundle $\mathcal{T}:=\widetilde{\mathcal{P}}$ over $\mathcal{P}$, with the action induced along $\U(1) \subset \widetilde{\mathcal{G}}$. The remaining  $\widetilde{\mathcal{G}}$-action on $\mathcal{T}$ determines a bundle morphism $\tau: \mathcal{Q} \otimes \pr_2^{*}\mathcal{T} \to \pr_1^{*}\mathcal{T}$ over $\mathcal{P}^{[2]}$, turning $\mathcal{T}$ into a trivialization of $\mathscr{L}_\mathcal{P}$. 
Conversely, if $\mathcal{T}$ is a trivialization of $\mathscr{L}_\mathcal{P}$, with isomorphism $\tau$, then the formula $t \cdot \tilde g \defeq \tau(\tilde g^{-1} \otimes t)$  determines a $\widetilde{\mathcal{G}}$-action on $\mathcal{T}$ such that the composite $\mathcal{T} \to \mathcal{P} \to \mathcal{M}$ has the structure of  a principal $\widetilde{\mathcal{G}}$-bundle $\widetilde{\mathcal{P}}$ over $\mathcal{M}$. 
The former bundle projection $\mathcal{T} \to \mathcal{P}$ gives the  map $\phi: \widetilde{\mathcal{P}} \to \mathcal{P}$ turning $\widetilde{\mathcal{P}}$ into a lift.
We now have the following fundamental result regarding lifting gerbes (see the references above).
\begin{proposition}
\label{prop:lifting}
The two constructions described above extend to functors 
\begin{equation*}
\xymatrix{\Triv(\mathscr{L}_\mathcal{P}) \ar@<2pt>[r]  & \Lift(\mathcal{P})  \ar@<2pt>[l]}
\end{equation*}
and establish an isomorphism of categories.
\end{proposition} 

Next we recall the notion of a twisted vector bundle, which in our infinite-dimensional setting will be a twisted rigged Hilbert space bundle. We will write $\mathcal{Q}_{\C} \defeq (\mathcal{Q} \times \C)/\U(1)$ for the hermitian line bundle associated to a Fr\'echet principal $\U(1)$-bundle $\mathcal{Q}$, which may be seen as an application of \cref{lem:AssociatedHilbertBundle} to the smooth representation $\U(1) \times \C \to \C$. We further write $\mu_{\C}$ for the unitary line bundle isomorphism induced from a principal bundle morphism $\mu$. 

\begin{definition}
\label{def:trhsb}
Let $\mathscr{G}=(\mathcal{Y},\mathcal{Q},\mu)$ be a bundle gerbe over a Fr\'echet manifold $\mathcal{M}$. A \emph{$\mathscr{G}$-twisted rigged Hilbert space bundle} is a rigged Hilbert space bundle $\mathcal{K}$ over $\mathcal{Y}$ together with an isometric isomorphism
\begin{equation*}
\kappa : \mathcal{Q}_{\C} \otimes \pr_2^{*}\mathcal{K} \to \pr_1^{*}\mathcal{K}
\end{equation*} 
of rigged Hilbert space bundles over $\mathcal{Y}^{[2]}$, such that
\begin{equation*}
\kappa (\mu_{\C}(\ell_{12}\otimes \ell_{23}) \otimes v)=\kappa(\ell_{12} \otimes \kappa(\ell_{23} \otimes v))
\end{equation*}
holds for all $\ell_{ij}\in \mathcal{Q}_{\C}$ and $v\in \mathcal{K}$ for which above expression is well-defined. 
\end{definition}

Morphisms between $\mathscr{G}$-twisted rigged Hilbert space bundles are locally bounded  rigged Hilbert space bundle morphisms over $\mathcal{Y}$ that commute with the isomorphisms over $\mathcal{Y}^{[2]}$ in the obvious way. Thus, $\mathscr{G}$-twisted rigged Hilbert space bundles form a category $\Mod{\mathscr{G}}$.

\begin{remark}
We recall that rigged Hilbert space bundles extend to continuous Hilbert space bundles, with transition functions in $\U(\mathcal{H})$ equipped with the strong operator topology. Any $\mathscr{G}$-twisted rigged Hilbert space bundle gives rise to a \emph{projective} Hilbert space bundle over $\mathcal{M}$, with transition functions in the group $P\U(\mathcal{H})$. Besides regularity aspects, the advantage of twisted Hilbert space bundles over projective Hilbert space bundles  is that the twist is treated separately as an object in its own right, and the bundle is considered relative to a fixed twist. 
\end{remark}

We observe that a rigged Hilbert space bundle $\mathcal{K}$ over $\mathcal{M}$ and a trivialization $\mathcal{T}$ of $\mathscr{G}$ with bundle isomorphism $\tau$ can be turned into a $\mathscr{G}$-twisted rigged Hilbert space bundle  $\mathcal{K}_{\mathcal{T}} \defeq \mathcal{T}_{\C} \otimes \pi^{*}\mathcal{K}$ with isometric  isomorphism  ix $\kappa_{\mathcal{T}} \defeq \tau_{\C} \otimes 1$ (using implicitly the canonical isomorphism $\pr_2^{*}\pi^{*}\mathcal{K}\cong \pr_1^{*}\pi^{*}\mathcal{K}$ over $\mathcal{Y}^{[2]}$). For fixed                                                     trivialization $\mathcal{T}$, this defines a \quot{twisting} functor
\begin{equation}
\label{eq:twisting}
\twist_{\mathcal{T}}:\VBdl(\mathcal{M}) \to \Mod{\mathscr{G}}\quad\text{;}\quad \mathcal{K} \mapsto \mathcal{K}_{\mathcal{T}}\text{.}
\end{equation}
We have the following result, see \cite{Bouwknegt2002,waldorf1} for finite-dimensional analogues. 

\begin{lemma}
\label{lem:twisting}
The functor $\twist_{\mathcal{T}}$ establishes an equivalence between the category $\VBdl(\mathcal{M})$ of rigged Hilbert space bundles over $\mathcal{M}$ and the category $\Mod{\mathscr{G}}$ of  $\mathscr{G}$-twisted rigged Hilbert space bundles.
\end{lemma}

The interesting point here is that the functor $\twist_{\mathcal{T}}$ is essentially surjective, which allows  to \quot{untwist} twisted rigged Hilbert space bundles using trivializations of the twist.
Let us recall the  argument for essential surjectivity.
If $\mathcal{K}'$ is a $\mathscr{G}$-twisted rigged Hilbert space bundle with isometric isomorphism $\kappa'$, then we consider the rigged Hilbert space bundle $\mathcal{T}_{\C}^{*} \otimes \mathcal{K}'$ over $\mathcal{Y}$, and the  isometric isomorphism
\begin{equation}
\label{eq:descent}
\xymatrix@C=4em{\pr_2^{*}(\mathcal{T}_{\C}^{*} \otimes \mathcal{K}') \cong \pr_2^{*}\mathcal{T}_{\C}^{*} \otimes \mathcal{Q}_{\C}^{*} \otimes \mathcal{Q}_{\C} \otimes \pr_2^{*}\mathcal{K}'\ar[r]^-{\tau_{\C}^{*-1} \otimes \kappa'} & \pr_1^{*}(\mathcal{T}_{\C}^{*}\otimes \mathcal{K}')\text{,}} 
\end{equation}
of rigged Hilbert space bundles over $\mathcal{Y}^{[2]}$, where $\tau_{\C}^{*}$ denotes the transposed map between dual bundles. It is straightforward to check that this map satisfies the cocycle condition over $\mathcal{Y}^{[3]}$, and hence equips $\mathcal{T}_{\C}^{*} \otimes \mathcal{K}'$ with a descent structure for the surjective submersion $\pi:\mathcal{Y} \to \mathcal{M}$. Since rigged Hilbert space bundles together with isometric isomorphisms form a stack (see \cref{re:presheafstack}), there exists a rigged Hilbert space bundle $\mathcal{K}$ over $\mathcal{M}$ together with an isometric isomorphism $\phi:\pi^{*}\mathcal{K} \to \mathcal{T}_{\C}^{*}\otimes \mathcal{K}'$ over $\mathcal{Y}$ such that $\pr_1^{*}\phi\circ  \pr_2^{*}\phi^{-1}$ is the map  \cref{eq:descent}. Tensoring with $\mathcal{T}_{\C}$ from the left, and using the canonical pairing between  between a hermitian line bundle and its dual, we obtain  an isometric isomorphism $\phi': \mathcal{K}_{\mathcal{T}} \to \mathcal{K}'$ of rigged Hilbert space bundles, which is in fact an isomorphism of $\mathscr{G}$-twisted rigged Hilbert space bundles; thus, $\mathcal{K}$ is an essential preimage for $\mathcal{K}'$.

After this introduction to twisted rigged Hilbert space bundles for general bundle gerbes, we shall turn to lifting bundle gerbes, which have particular twisted bundles. We consider again a central extension $\U(1) \to \widetilde{\mathcal{G}} \stackrel{q}{\to} \mathcal{G}$ of Fr\'echet Lie groups, a principal $\mathcal{G}$-bundle $\mathcal{\mathcal{P}}$ over a Fr\'echet manifold $\mathcal{M}$, and the lifting gerbe $\mathscr{L}_\mathcal{P}$. We recall that the principal $\U(1)$-bundle $\mathcal{Q}$ of $\mathscr{L}_\mathcal{P}$ is the pullback $\mathcal{Q} \defeq \delta^{*}\widetilde{\mathcal{G}}$, where $p_1\delta(p_1,p_2)=p_2$.  Suppose $\widetilde{\mathcal{G}} \times E \to E$ is a smooth representation on a rigged Hilbert space $E$, with the  property that the subgroup $\U(1)\subset \widetilde{\mathcal{G}}$ acts by scalar multiplication. 
Then, we consider the trivial rigged Hilbert space bundle $E(\mathcal{P})^{\twist} \defeq \mathcal{P} \times E$ over $\mathcal{P}$, and have the following result.

\begin{lemma}
\label{lem:universal}
 There exists a unique isometric  isomorphism 
\begin{equation*} 
\nu: \mathcal{Q}_{\C} \otimes \pr_2^{*}E(\mathcal{P})^{\twist} \to \pr_1^{*}E(\mathcal{P})^{\twist}
\end{equation*}
of rigged Hilbert space bundles over $\mathcal{P}^{[2]}$, such that $\nu([\tilde g,z],v) = z(\tilde g v)$ for all $\tilde g\in \widetilde{\mathcal{G}}$, $z\in\C$ and $v\in E$.
Moreover, $\nu$ turns $E(\mathcal{P})^{\twist}$ into an $\mathscr{L}_{\mathcal{P}}$-twisted rigged Hilbert space bundle. 
\end{lemma}

\begin{proof}
Uniqueness is clear, and that the given equation can be used to define $\nu$ as a bundle map follows since $\U(1)$ acts by scalar multiplication. That $\nu$ is smooth and isometric follows because the representation of $\widetilde{\mathcal{G}}$ on $E$ is smooth. In order to see that $\nu$ is an isomorphism, we notice that the inverse map  in the fibre over a point $(p_1,p_2)\in \mathcal{P}^{[2]}$ is given by $v \mapsto [\tilde g,1] \otimes \tilde g^{-1}v$, where $\tilde g\in \widetilde{\mathcal{G}}$ is an element projecting to $\delta(p_1,p_2)$. Since $\widetilde{\mathcal{G}} \to \mathcal{G}$ has local sections along $\delta:\mathcal{P}^{[2]} \to \mathcal{G}$, it is clear that this inverse map is smooth, and it is also isometric. This shows that $\nu$ is an isometric isomorphism. The compatibility condition with the isomorphism $\mu$ of $\mathscr{L}_{\mathcal{P}}$ is straightforward to prove. 
\end{proof}

The twisted rigged Hilbert space bundle $E(\mathcal{P})^{\twist}$ will be of importance below and in \cref{sec:twisted}, so that we shall summarize and give it a name. 
 
\begin{definition}
\label{def:canonicaltvb}
Let $\U(1) \to \widetilde{\mathcal{G}} \to \mathcal{G}$ be a central extension of Fr\'echet  Lie groups, let $\mathcal{P}$ be a Fr\'echet\ principal $\mathcal{G}$-bundle over a Fr\'echet manifold $\mathcal{M}$, and let $\mathscr{L}_\mathcal{P}$ be the associated lifting gerbe. For a smooth representation $\widetilde{\mathcal{G}} \times E \to E$ with $\U(1)$ acting by scalar multiplication,  $E(\mathcal{P})^{\twist}$ is called the \emph{universal $\mathscr{L}_\mathcal{P}$-twisted rigged Hilbert space bundle} with typical fibre  $E$. 
\end{definition}

The following result motivates this terminology, as the universal $\mathscr{L}_{\mathcal{P}}$-twisted rigged Hilbert space bundle $E(\mathcal{P})^{\twist}$ is a twisted version of a rigged Hilbert space bundle associated to \emph{any} solution of the lifting problem. 

\begin{proposition}
\label{lem:reptwistedvb}
Suppose $\mathcal{T}$ is a trivialization of $\mathscr{L}_\mathcal{P}$, corresponding to a lift $\widetilde{\mathcal{P}}$ under the isomorphism of \cref{prop:lifting}. 
Let $E(\widetilde{\mathcal{P}}) \defeq (\widetilde{\mathcal{P}} \times E)/\widetilde{\mathcal{G}}$ be the associated rigged Hilbert space bundle.
Then, there exists a canonical isometric isomorphism 
\begin{equation*}
\twist_{\mathcal{T}}(E(\widetilde{\mathcal{P}})) \cong E(\mathcal{P})^{\twist}
\end{equation*}
of $\mathscr{   L}_\mathcal{P}$-twisted rigged Hilbert space bundles.
\end{proposition}  

\begin{proof}
Employing the relation between $\mathcal{T}$ and $\widetilde{\mathcal{P}}$ we have have $E(\widetilde{\mathcal{P}})=(\mathcal{T} \times E)/\widetilde{\mathcal{G}}$ with $\tilde g \in \widetilde{\mathcal{G}}$ acting on $\mathcal{T} \times E$ by $(t,v) \mapsto (\tau(\tilde g\otimes t),\tilde g v)$, and the bundle projection is induced by $\mathcal{T} \to \mathcal{P} \to \mathcal{M}$. We claim that there exists a unique isometric isomorphism 
\begin{equation*}
E(\widetilde{\mathcal{P}})_{\mathcal{T}}=\mathcal{T}_{\C} \otimes \pi^{*}\mathcal{H} \stackrel{\phi}{\to} \mathcal{P}\times E=E(\mathcal{P})^{\twist}
\end{equation*}
of rigged Hilbert space bundles over $\mathcal{P}$ such that 
\begin{equation*}
\phi([t,z] \otimes [t,v]) = (p,zv)\text{,} 
\end{equation*}
where $t\mapsto p$ under the projection $\mathcal{T} \to \mathcal{P}$. Indeed, it is clear that every element of $\mathcal{T}_{\C} \otimes \pi^{*}E(\widetilde{\mathcal{P}})$ can be written as on the left hand side; this shows uniqueness. For existence, if $[t',z'] \otimes [t',v']$ represents the same element, then $t$ and $t'$ project to the same $p\in \mathcal{P}$, and thus $t'=tu$ for $u\in \U(1)$. Then, $[t',z']=[t,z]$ implies $z=uz'$ and $[t',v']=[t,v]$ implies $v'=uv$.
Thus, $zv=z'v'$, which shows the  existence of $\phi$. It is clear that $\phi$ is isometric and smooth, and a smooth isometric inverse bundle morphism can be constructed using the fact that $\mathcal{T}\to \mathcal{P}$ has smooth local sections.  
 It remains to prove that $\phi$ is compatible with the bundle isomorphisms $\kappa_{\mathcal{T}}$ of $\twist_{\mathcal{T}}(E(\widetilde{\mathcal{P}}))$ and $\nu$ of $E(\mathcal{P})^{\twist}$. The following calculation only using the definitions shows this:
\begin{align*}
\nu([\hat g,\hat z] \otimes \phi([t,z] \otimes [t,v])) &= \nu([\hat g,\hat z] \otimes  (p_2,zv))
\\&= (p_1,z\hat z\hat\rho(\hat g,v)) 
\\&=\phi([\tau(\hat g,t),z\hat z] \otimes [\tau(\hat g,t),\hat\rho(\hat g,v)] )
\\&=\phi(\kappa_{\mathcal{T}}([\hat g,\hat z] \otimes[t,z] \otimes [t,v] ))\text{.}\qedhere
\end{align*}
\end{proof}

Next we generalize the previous results to twisted rigged Hilbert space bundles that are  rigged module bundles for a rigged \cstar-algebra bundle $\mathcal{A}$ over $\mathcal{M}$, as introduced in \cref{def:repcstarbundle}.

\begin{definition}
\label{def:trmb}
For a bundle gerbe  $\mathscr{G}$ over $\mathcal{M}$ with surjective submersion $\pi:\mathcal{Y} \to \mathcal{M}$,  a \emph{$\mathscr{G}$-twisted rigged $\mathcal{A}$-module bundle} is a $\mathscr{G}$-twisted rigged Hilbert space bundle $\mathcal{K}$ that  is at the same time a rigged $\pi^{*}\mathcal{A}$-module bundle, and whose isomorphism $\kappa$ is $\mathcal{A}$-linear in the sense that $\kappa(\ell,a v)=a \kappa(\ell,v)$ for all $a\in \mathcal{A}_x$, $e\in \mathcal{K}_{y'}$ and $\ell \in \mathcal{Q}_{y,y'}$ with $\pi(y)=\pi(y')=x$. \end{definition}

In \cref{sec:HilbertBundle:2} we wrote $\VBdl^\mathcal{A}(\mathcal{M})$ for the category of rigged $\mathcal{A}$-module bundles over $\mathcal{M}$. Now we denote by $\Mod{\mathscr{G}}^{\mathcal{A}}$ the category of $\mathscr{G}$-twisted rigged $\mathcal{A}$-module bundles. 
The twisting functor of \cref{eq:twisting} extends to a functor
\begin{equation}
\label{eq:Atwisting}
\twist_{\mathcal{T}}:\VBdl^\mathcal{A}(\mathcal{M}) \to \Mod{\mathscr{G}}^{\mathcal{A}}\text{,}
\end{equation}
and \cref{lem:twisting} generalizes in a straightforward way to the following statement.

\begin{proposition}
\label{lem:twistingmod}
The functor $\twist_{\mathcal{T}}$ establishes an equivalence between the category $\VBdl^{\mathcal{A}}(\mathcal{M})$ of rigged $\mathcal{A}$-module bundles over $\mathcal{M}$ and the category $\Mod{\mathscr{G}}^{\mathcal{A}}$ of  $\mathscr{G}$-twisted rigged $\mathcal{A}$-module bundles.
\end{proposition}

In order to generalize the universal twisted rigged Hilbert space bundle of \cref{def:canonicaltvb,lem:reptwistedvb} to rigged module bundles we 
now consider a smooth representation of a central extension $\U(1) \to \widetilde{\mathcal{G}} \stackrel{q}{\to} \mathcal{G}$ on a rigged module $E$ for a rigged \cstar-algebra $A$, in the sense of \cref{def:repofce}. If then $\mathcal{P}$ is a Fr\'echet principal $\mathcal{G}$-bundle over $\mathcal{M}$, we have the associated rigged \cstar -algebra bundle $\mathcal{A} \defeq (\mathcal{P} \times A)/\mathcal{G}$ and the universal $\mathscr{L}_\mathcal{P}$-twisted rigged Hilbert space bundle $E(\mathcal{P})^{\twist}$ with typical fibre $E$.

\begin{lemma}
\label{lem:twistedmodulebundle}
The universal $\mathscr{L}_{\mathcal{P}}$-twisted rigged Hilbert space bundle $E(\mathcal{P})^{\twist}$ is an $\mathscr{L}_{\mathcal{P}}$-twisted rigged $\mathcal{A}$-module bundle. 
\end{lemma} 

\begin{proof}
We have a canonical global trivialization $\pi^{*}\mathcal{A}\to  \mathcal{P} \times A$ over $\mathcal{P}$, given in the fibre over $p\in \mathcal{P}$ by  $[p,a] \mapsto (p,a)$. It is clear that  $E(\mathcal{P})^{\twist}=\mathcal{P} \times E$ is a rigged module for  $\mathcal{P} \times A$, with typical fibre the rigged $A$-module $E$. Then, it is straightforward to verify using the relation between the involved actions of \cref{def:repofce} that the bundle isomorphism $\nu$ of \cref{lem:universal} is $\mathcal{A}$-linear; this shows the claim.
\end{proof}

In order to generalize \cref{lem:reptwistedvb}, we assume that  $\widetilde{\mathcal{P}}$ is a lift of the structure group of $\mathcal{P}$ along the central extension $\U(1) \to \widetilde{\mathcal{G}} \to \mathcal{G}$; then $(\widetilde{\mc{P}} \times E)/\widetilde{\mc{G}}$ is a rigged $\mathcal{A}$-module bundle by \cref{prop:InducedBundleRep}.

\begin{theorem}
\label{lem:reptwistedmb}
Suppose $\mathcal{T}$ is a trivialization of $\mathscr{L}_\mathcal{P}$, corresponding to a lift $\widetilde{\mathcal{P}}$ under the isomorphism of \cref{prop:lifting}. 
Let $E(\widetilde{\mathcal{P}}) \defeq (\widetilde {\mathcal{P}} \times E)/\widetilde{\mathcal{G}}$ be the associated rigged $\mathcal{A}$-module bundle. 
Then, the isometric isomorphism $\twist_{\mathcal{T}}(E(\widetilde{\mathcal{P}})) \cong E(\mathcal{P})^{\twist}$ of \cref{lem:reptwistedvb} is an isometric isomorphism of $\mathscr{L}_\mathcal{P}$-twisted rigged $\mathcal{A}$-module bundles.
\end{theorem}  

\begin{proof}
We only need to check that the isomorphism $\phi$ described in \cref{lem:reptwistedvb} is $\pi^{*}\mathcal{A}$-linear, which is a straightforward calculation.
\end{proof}

In the remainder of this section, we discuss isomorphisms between bundle gerbes and the induced relation between twisted rigged Hilbert space bundles. Such relations will be constructed in \cref{sec:LagrangianTwistedSpinorBundle,sec:lagrangiangerbe}.
For the purpose of this article, we can restrict our attention to the simplest kind of isomorphisms between bundle gerbes, called \emph{refinements}. 
\begin{definition}
If $\mathscr{G}=(\mathcal{Y},\mathcal{Q},\mu)$ and $\mathscr{G}'=(\mathcal{Y}',\mathcal{Q}',\mu')$ are bundle gerbes over a Fr\'echet manifold $\mathcal{M}$, then a \emph{refinement} $\mathscr{R}:\mathscr{G} \to \mathscr{G}'$ is a pair, $(r,\rho)$, consisting of a smooth map $r:\mathcal{Y} \to \mathcal{Y}'$, commuting with the projections to $\mathcal{M}$, and a smooth bundle morphism $\rho:\mathcal{Q} \to \mathcal{Q}'$ covering the induced map $r^{[2]}:\mathcal{Y}^{[2]} \to \mathcal{Y}'^{[2]}$, such that the diagram
\begin{equation*}
\xymatrix{\pr_{12}^{*}\mathcal{Q} \otimes \pr_{23}^{*}\mathcal{Q} \ar[d]_{\pr_{12}^{*}\rho \otimes \pr_{23}^{*}\rho} \ar[r]^-{\mu} & \pr_{13}^{*}\mathcal{Q} \ar[d]^{\pr_{13}^{*}\rho} \\ \pr_{12}^{*}\mathcal{Q}' \otimes \pr_{23}^{*}\mathcal{Q}' \ar[r]_-{\mu'} & \pr_{13}^{*}\mathcal{Q}'}
\end{equation*}
commutes.
\end{definition}

Any refinement between bundle gerbes induces an isomorphism in the usual bicategory of bundle gerbes (in fact, that bicategory can be obtained by formally inverting refinements). 
It is straightforward to check that trivializations can be pulled back along refinements, resulting into a functor
\begin{equation*}
\mathscr{R}^{*}:\Triv(\mathscr{G}') \to \Triv(\mathscr{G})\quad\text{;}\quad \mathcal{T} \mapsto r^{*}\mathcal{T}\text{.} 
\end{equation*}
One can show that this functor is an equivalence of categories and it has in fact a canonical inverse functor, though this will not be important here. 
Similarly, a refinement induces a functor between twisted rigged module bundles for a rigged \cstar-algebra bundle $\mathcal{A}$ over $\mathcal{M}$,
\begin{equation*}
\mathscr{R}^{*}: \mathscr{G}'\text{-}\VBdl^{\mathcal{A}} \to \mathscr{G}\text{-}\VBdl^{\mathcal{A}}\quad\text{;}\quad \mathcal{K} \mapsto r^{*}\mathcal{K}\text{.}
\end{equation*}
This functor is again an equivalence of categories. It allows us to compare twisted rigged module bundles for isomorphic bundle gerbes.

We have also been interested in twisting $\mathcal{A}$-module bundles (or rather, in untwisting twisted ones) using the functor $\twist_{\mathcal{T}}$ of \cref{eq:Atwisting}.
We have the following compatibility result between twisting and refining.

\begin{lemma}
\label{lem:refinementtriv}
Let $\mathscr{R}:\mathscr{G} \to \mathscr{G}'$ be a refinement between bundle gerbes over $\mathcal{M}$, and let $\mathcal{T}$ be a trivialization of $\mathscr{G}'$. Let $\mathcal{A}$ be a rigged \cstar-algebra bundle over $\mathcal{M}$. Then, the  diagram
\begin{equation*}
\xymatrix@R=1.3em@C=4em{& \Mod{\mathscr{G}'}^{\mathcal{A}} \ar[dd]^{\mathscr{R}^{*}}  \\\VBdl^\mathcal{A}(\mathcal{M}) \ar[ur]^-{\twist_{\mathcal{T}}}\ar[dr]_-{\twist_{\mathscr{R}^{*}\mathcal{T}}} & \\ & \Mod{\mathscr{G}}^{\mathcal{A}}}
\end{equation*}
 of functors commutes up to a canonical natural isometric isomorphism.
\end{lemma}

\begin{proof}
If $\mathcal{K}$ is a rigged $\mathcal{A}$-module bundle over $\mathcal{M}$, then the underlying rigged Hilbert space bundle over $\mathcal{Y}$ of  $\mathscr{R}^{*}(\twist_{\mathcal{T}}(\mathcal{K}))$ is  $r^{*}\mathcal{K}_{\mathcal{T}}=r^{*}(T_{\C} \otimes \pi'^{*}\mathcal{K})$, while the underlying rigged Hilbert space bundle of $\twist_{\mathscr{R}^{*}\mathcal{T}}(\mathcal{K})$ is $\mathcal{K}_{\mathscr{R}^{*}\mathcal{T}}=r^{*}T_{\C} \otimes \pi^{*}\mathcal{K}$. It is straightforward to show that the canonical isometric isomorphism between these rigged Hilbert space bundles is $\pi^{*}\mathcal{A}$-linear and respects the bundle isomorphisms $(r^{[2]})^{*}\kappa_{\mathcal{T}}$ and $\kappa_{\mathscr{R}^{*}\mathcal{T}}$.  \end{proof}

Refinements may be used to relate lifting gerbes in a common situation (see for example \cref{sec:impLiftingGerbe}).
To start, suppose that we are given a commutative diagram
\begin{equation*}
\xymatrix{\U(1) \ar@{=}[d] \ar[r] & \widetilde{\mathcal{G}} \ar[d]^{\tilde{\phi}} \ar[r] & \mathcal{G} \ar[d]^{\phi} \\ \U(1) \ar[r] & \widetilde{\mathcal{H}} \ar[r] & \mathcal{H}}
\end{equation*}
of Fr\'echet Lie groups and smooth group homomorphisms, whose horizontal sequences are central extensions. We suppose that $\mathcal{E}$ is a Fr\'echet principal  $\mathcal{G}$-bundle over $\mathcal{M}$, and $\mathcal{F}$ is a Fr\'echet principal  $\mathcal{H}$-bundle over $\mathcal{M}$, so that the lifting gerbes $\mathscr{L}_{\mathcal{E}}$ and $\mathscr{L}_{\mathcal{F}}$ are defined. We suppose further that $f: \mathcal{E} \to \mathcal{F}$ is a bundle morphism that is equivariant along $\phi$. Then, the diagram
\begin{equation*}
\xymatrix{\mathcal{E}^{[2]} \ar[d]_{f \times f} \ar[r]^{\delta} & \mathcal{G} \ar[d]^{\phi} \\ \mathcal{F}^{[2]} \ar[r]_{\delta} & \mathcal{H}}
\end{equation*}
is commutative. We may regard $\tilde\phi$ as a morphism of principal $\U(1)$-bundles covering $\phi:\mathcal{G} \to \mathcal{H}$, and hence obtain via pullback another another morphism $\rho\defeq \delta^{*}\tilde\phi: \delta^{*}\widetilde{\mathcal{G}}\to \delta^{*}\widetilde{\mathcal{H}}$ covering $f^{[2]}$. The fact that $\tilde\phi$ is also a group homomorphism implies, by construction of the lifting bundle gerbes, the following.

\begin{lemma}
\label{lem:refinement}
For every bundle morphism $f:\mathcal{E} \to \mathcal{F}$ that is equivariant along $\phi$, the pair 
$\mathscr{R}_f\defeq(f,\rho)$ is a refinement $\mathscr{R}_f:\mathscr{L}_{\mathcal{E}} \to \mathscr{L}_{\mathcal{F}}$. In particular, the two lifting problems are equivalent.
\end{lemma}

Next we suppose that we have a smooth representation of the central extension $\U(1) \to \widetilde{\mathcal{H}} \to \mathcal{H}$ on a rigged $A$-module $E$, in the sense of \cref{def:repofce}. Then we induce, along the Lie group homomorphisms $\tilde\phi$ and $\phi$, a smooth representation of the central extension $\U(1) \to \widetilde{\mathcal{G}} \to \mathcal{G}$ on $E$. We consider  the rigged \cstar-algebra bundles $\mathcal{A}\defeq (\mathcal{E} \times A)/\mathcal{G}$ and $\mathcal{B} \defeq (\mathcal{F} \times A)/\mathcal{H}$ over $\mathcal{M}$. In this situation, we have  two universal twisted module bundles: $E(\mathcal{E})^{\twist}$ is the universal $\mathscr{L}_{\mathcal{E}}$-twisted $\mathcal{A}$-module bundle with typical fibre $E$, and $E(\mathcal{F})^{\twist}$ is the universal $\mathscr{L}_{\mathcal{F}}$-twisted $\mathcal{B}$-module bundle with typical fibre $E$. We suppose again that $f: \mathcal{E} \to \mathcal{F}$ is a $\phi$-equivariant bundle morphism. By \cref{prop:InvarianceUnderLiftsCStar} it induces an isometric isomorphism $\mathcal{A} \to \mathcal{B}$, under which  $E(\mathcal{F})^{\twist}$ becomes an $\mathscr{L}_{\mathcal{F}}$-twisted $\mathcal{A}$-module bundle. 

\begin{proposition}
\label{prop:functorialityofuniversalmodule}
There exists a canonical isometric isomorphism
\begin{equation*}
\mathscr{R}_f^{*}E(\mathcal{F})^{\twist} \cong E(\mathcal{E})^{\twist}
\end{equation*} 
between $\mathscr{L}_{\mathcal{E}}$-twisted $\mathcal{A}$-module bundles.
\end{proposition}

\section{Rigged fermionic Fock spaces}\label{sec:FreeFermions}

\label{sec:freefermions}

In this section we consider the fermionic Fock space, which plays the role of the typical fibre of the Fock bundles we want to construct.
In \cref{sec:CliffFockAndImp} we discuss algebraic and analytical aspects of Clifford algebras and Fock spaces; this is a collection of more or less well-known material. In \cref{sec:SmoothRepresentations} we recast Clifford algebras and Fock spaces in the new setting of rigged \cstar-algebras and rigged Hilbert spaces we introduced in \cref{sec:HilbertBundle}, and we prove several new smoothness results about Fock spaces and Clifford multiplication.

\subsection{Lagrangians and Fock spaces}
\label{sec:CliffFockAndSpin}
\label{sec:CliffFockAndImp}

Concerning infinite-dimensional Clifford algebras and their representations on Fock spaces, we follow \cite{PR95}.
The results we require on implementers are more spread out across the literature, see for instance \cite{Ara85, Ott95, Ne09, Kristel2019}.
Throughout  \cref{sec:freefermions}, we consider  a complex Hilbert space $V$ equipped with a real structure $\alpha$, i.e., an anti-unitary involution $\alpha : V \rightarrow V$.

Given a unital \cstar-algebra $A$, we recall that a map $f:V \rightarrow A$ is a \emph{Clifford map} if the following equations are satisfied, for all $v,w \in V$:
\begin{align}\label{eq:CliffordMap}
 f(v)f(w) + f(w)f(v) &= 2 \langle v, \alpha(w) \rangle \mathds{1}, \quad f(v)^{*} = f(\alpha(v)).
\end{align}
The Clifford \cstar-algebra $\Cl(V)$ is the unique (up to unique isomorphism) unital \cstar-algebra equipped with a Clifford map $\iota: V \rightarrow \Cl(V)$, such that for each unital \cstar-algebra $A$ and each Clifford map $f: V \rightarrow A$, there exists a unique unital \cstar-algebra homomorphism $\Cl(f): \Cl(V) \rightarrow A$ with the property that $\Cl(f) \circ \iota = f$.
An explicit construction of $\Cl(V)$ is given in \cite[Section 1.2]{PR95}.

We define the orthogonal group of $V$, denoted $\O(V)$, to consist of those unitary transformations of $V$ that commute with the real structure.
If $g \in \O(V)$, then $\iota g: V \rightarrow \Cl(V)$ is a Clifford map.
We write $\theta_{g} \defeq \Cl( \iota g): \Cl(V) \rightarrow \Cl(V)$ for its extension to $\Cl(V)$.
The map $\theta_{g}$ is called the \emph{Bogoliubov automorphism} associated to $g$.
The map $\theta: g \mapsto \theta_{g}$ is  a continuous homomorphism from $\O(V)$ to $\Aut(\Cl(V))$, where $\O(V)$ is equipped with the operator norm topology, and $\Aut(\Cl(V))$ is equipped with the strong operator topology, see, e.g. \cite[Proposition 4.35]{Ambler2012}.

A \emph{Lagrangian} in $V$ is a subspace $L \subset V$ such that $V$ splits as the orthogonal direct sum $V = L \oplus \alpha(L)$.
If $L$ is a Lagrangian, then the \emph{Fock space} $\mc{F}_{L}$ is the Hilbert completion of the exterior algebra, $\Lambda L$, of $L$.
We identify $\alpha(L)$ with the dual of $L$ by identifying $w \in \alpha(L)$ with the linear map $L \ni v \mapsto \langle v, \alpha(w) \rangle$.
If $v \in L$ and $w \in \alpha(L) \simeq L^{*}$, then we write $c(v): \mc{F}_{L} \rightarrow \mc{F}_{L}$ for left multiplication with $v$, and $a(w) : \mc{F}_{L} \rightarrow \mc{F}_{L}$ for contraction with $w$.
The maps $c(v)$ and $a(v)$ are bounded operators on $\mc{F}_{L}$, and the map 
\begin{align*}
 \rho_{L}:V = L \oplus \alpha(L) &\rightarrow \mc{B}(\mc{F}_{L}), \quad (v,w) \mapsto \sqrt{2}(c(v) + a(w))
\end{align*}
is a Clifford map. This means that $\Cl(\rho_{L}): \Cl(V) \rightarrow \mc{B}(\mc{F}_{L})$ is a unital \cstar-algebra homomorphism; i.e., a representation of $\Cl(V)$ on $\mc{F}_{L}$.
This representation is irreducible \cite[Theorem 2.4.2]{PR95} and faithful; hence, we may identify $\Cl(V)$ with its image in $\mc{B}(\mc{F}_{L})$.
Whenever convenient, we adopt the notation $a \lact v = \Cl(\rho_{L})(a)(v)$ for $a \in \Cl(V)$ and $v \in \mc{F}_{L}$.

An element $g \in \O(V)$ is called \emph{implementable}, if there exists  $U \in \U(\mc{F}_{L})$ with the property that
\begin{equation}\label{eq:Implementer}
 \theta_{g}(a) = UaU^{*}
\end{equation}
for all $a \in \Cl(V)$;
the operator $U$ is said to \emph{implement} $g$. The problem to decide which $g\in \O(V)$ are implementable is called the \quot{implementability problem}.
 It is completely solved: an element $g \in \O(V)$ is implementable if and only if the operator $P_{L} g P_{L}^{\perp}$ is Hilbert-Schmidt, where $P_L$ is the orthogonal projection to $L$, see \cite[Theorem 3.3.5]{PR95} or \cite[Theorem 6.3]{Ara85}.
We write $\O_{\res}(V)$ for the set consisting of those $g \in \O(V)$ which are implementable; the set $\O_{\res}(V)$ is in fact a subgroup of $\O(V)$.

The group $\O(V)$ can be equipped with the structure of Banach Lie group in the standard way, with underlying topology the operator norm topology. The Lie algebra of $\O(V)$ is
\begin{equation*}
 \lie{o}(V) = \{ X \in \mc{B}(V) \mid [X,\alpha] = 0, X^{*} = -X \}.
\end{equation*}
The subgroup $\O_L(V)$ can also be equipped with the structure of a Banach Lie group,  whose underlying topology is given by the norm $\| g\|_{\mc{J}} = \| g\| + \| P_{L} g P_{L}^{\perp} \|_{2}$, where $\|g\|$ is the operator norm of $g$, and where $\| \cdot \|_{2}$ is the Hilbert-Schmidt norm, \cite[Section 3.4]{Kristel2019}. The inclusion $\O_L(V) \to \O(V)$ is smooth.
The Lie algebra of $\O_{\res}(V)$ is 
\begin{equation*}
 \lie{o}_{\res}(V) = \{ X \in \lie{o}(V) \mid \| P_{L} X P_{L}^{\perp} \|_{2} < \infty \}.
\end{equation*}

The \emph{group of implementers}, $\Imp_{\res}(V)$, is defined to be the subgroup of $\U(\mc{F}_{L})$ consisting of those operators $U \in \U(\mc{F}_{L})$, for which there exists a $g \in \O_{\res}(V)$ such that  \cref{eq:Implementer} holds.
If $U \in \Imp_{\res}(V)$, then the element $g \in \O_{\res}(V)$ that it implements is determined uniquely, and we obtain a group homomorphism $q:\Imp_{\res}(V) \rightarrow \O_{\res}(V)$.
Using the irreducibility of the representation of $\Cl(V)$ on $\mc{F}_{L}$ together with Schur's Lemma, we see that, for each $g \in \O_{\res}(V)$, the fibre $q^{-1}\{g\}$ is a $\U(1)$-torsor.
We equip the group $\Imp_{\res}(V)$ with the structure of Banach Lie group as in \cite[Theorem 3.15]{Kristel2019}, see also \cite{Wurzbacher2001}.
We then have that the exact sequence
\begin{equation*}
 \U(1) \rightarrow \Imp_{\res}(V) \xrightarrow{q} \O_{\res}(V) ,
\end{equation*}
is a central extension of Banach Lie groups.
It is important to note that the topology underlying the Banach Lie group structure on $\Imp_{\res}(V)$ is not the operator norm topology, and that the inclusion map $\Imp_{\res}(V) \rightarrow \U(\mc{F}_{L})$ is not continuous, let alone smooth.

The following results are  used in \cref{sec:lagrangiangerbe,sec:LagrangianTwistedSpinorBundle}, where we carry out the construction of twisted Fock bundles.
We write $\Lag(V)$ for the set of Lagrangians in $V$.
Let $L \in \Lag(V)$ be a Lagrangian, and let $\mathcal{J}_L:=i(P_L-P_L^{\perp})$ be the corresponding complex structure.
We  define the \emph{restricted unitary group}:
\begin{equation*}
 \U_{\res}(V) \defeq \{ g \in \U(V) \mid \|[g,\mc{J}_{L}]\|_{2} < \infty \}.
\end{equation*}
Suppose that $L' \in \Lag(V)$ is a second Lagrangian.
We may now ask if the corresponding Fock space representations are unitarily equivalent, in other words, if there exists a unitary operator $T: \mc{F}_{L} \rightarrow \mc{F}_{L'}$ that intertwines the $\Cl(V)$ representations. The problem to decide when this is the case is called the \quot{equivalence problem}. The following result shows that it is closely related to the implementability problem.

\begin{lemma}\label{lem:FockSpacesUnitarilyEquivalent}
For two Lagrangians $L,L'\in \Lag(V)$, the following are equivalent:
 \begin{enumerate}
  \item The representations $\mc{F}_{L}$ and $\mc{F}_{L'}$ are unitarily equivalent.
  \item The operator $P_{L}^{\perp}P_{L'}: V \rightarrow V$ is Hilbert-Schmidt.
  \item The operator $P_{L}- P_{L'}: V \rightarrow V$ is Hilbert-Schmidt.
  \item There exists an element $g \in \O_{\res}(V)$ such that $g(L) = L'$.
  \item There exists an element $g \in \U_{\res}(V)$ such that $g(L) = L'$.
 \end{enumerate}
\end{lemma}
\begin{proof}
 That 1-4 are equivalent is proven in \cite[Theorem 3.4.1 \& Theorem 3.4.2]{PR95}. That 4 implies 5 is obvious, after all $\O_{\res}(V) \subset \U_{\res}(V)$. Finally, assume that 5 holds.
 We shall prove that 3 then holds.
 First, we compute
 \begin{equation*}
  [g, \mc{J}_{L}] = i[g, P_{L} - P_{L}^{\perp}] = i[g, 2P_{L} - \mathds{1}] = 2i[g,P_{L}],
 \end{equation*}
 whence we conclude that $[P_{L},g]$ is Hilbert-Schmidt.
 By assumption, we have that $g(L) = L'$, hence
 \begin{equation*}
  P_{L} - P_{L'} = P_{L} - g P_{L} g^{-1} = (P_{L}g - gP_{L})g^{-1} = [P_{L},g] g^{-1}. 
 \end{equation*}
 Hence, $P_L - P_{L'}$ is Hilbert-Schmidt and 3 is satisfied. 
\end{proof}

If a pair of Lagrangians satisfies any of the equivalent conditions of \cref{lem:FockSpacesUnitarilyEquivalent} we say that the Lagrangians are \emph{equivalent}.
This defines an equivalence relation on the set $\Lag(V)$ of Lagrangians.
We write $\Lag_{\res}(V)$ for the set of Lagrangians equivalent to $L$.
The following \lcnamecref{lem:TheTwoU1Torsors} makes the relationship between the implementability problem and the equivalence problem  explicit.
Its proof is a straightforward computation employing the implementability condition \labelcref{eq:Implementer}.  

\begin{proposition}\label{lem:TheTwoU1Torsors}
Let $L_{1},L_{2} \in \Lag_L(V)$, and suppose $g_{1},g_{2} \in \O_{\res}(V)$ satisfy $g_{1}(L)=L_{1}$ and $g_{2}(L) = L_{2}$. Let $\U_{\Cl(V)}(\mc{F}_{L_{1}}, \mc{F}_{L_{2}})$ denote the set of unitary equivalences of $\Cl(V)$-representations, and let $\Imp_{L}(V)_{g}$ denote the set of implementers of $g\in \O_L(V)$. Then, the map
\begin{equation*}
\EqImp_{g_1,g_2}:  \U_{\Cl(V)}(\mc{F}_{L_{1}}, \mc{F}_{L_{2}}) \rightarrow \Imp_{L}(V)_{g_{2}^{-1}g_{1}}, \quad \quad T \mapsto \Lambda_{g_{2}}^{-1}T\Lambda_{g_{1}},
\end{equation*}
 is an isomorphism of $\U(1)$-torsors. 
Moreover, if $L_{3}\in \Lag_L(V)$ and $g_{3}\in \O_L(V)$ such that $g_{3}(L)=L_{3}$, then there is a commutative diagram
\begin{equation*}
\xymatrix@C=8em{\U_{\Cl(V)}(\mathcal{F}_{L_{2}},\mathcal{F}_{L_{3}})  \times\U_{\Cl(V)}(\mathcal{F}_{L_{1}},\mathcal{F}_{L_{2}}) \ar[r]^-{\EqImp_{g_2,g_3} \times \EqImp_{g_1,g_2}} \ar[d]_{\circ} & \Imp_L(V)_{g_{3}^{-1}g_{2}}  \times \Imp_L(V)_{g_{2}^{-1}g_{1}} \ar[d] \\
\U_{\Cl(V)}(\mathcal{F}_{L_{1}},\mathcal{F}_{L_{3}}) \ar[r]_-{\EqImp_{g_1,g_3}} & \Imp_L(V)_{g_{3}^{-1}g_{1}}}
\end{equation*}
\end{proposition}

Another aspect of the theory of Lagrangians that we need in \cref{sec:bundlegerbe} is the completion of subspaces to Lagrangians. The following sets up the relevant terminology and results.    
 
\begin{definition}
 A subspace $L' \subset V$ is a \emph{sublagrangian} if it is isotropic and if $L' \oplus \alpha(L') \subseteq V$ is of finite and even codimension.
\end{definition}

\begin{lemma}\label{lem:sublagrangians}
 Let $L' \subset V$ be a sublagrangian. Then there exists a Lagrangian  $L \subset V$ such that $L' \subseteq L$.
 If $L_{0}$ and $L_{1}$ are two Lagrangians which both contain the sublagrangian $L'$, then $L_{0}$ and $L_1$ are equivalent.
\end{lemma}

\begin{proof}
 Let $L'$ be a sublagrangian. Then, by definition, the orthogonal complement of $L' \oplus \alpha(L')$ is of finite and even dimension. Moreover, it is clearly preserved by $\alpha$, hence it contains a Lagrangian subspace $K \subset V$. It follows that $L' \oplus K \subset V$ is Lagrangian.
For the second part, a direct computation shows that $P_{L_{0}}^{\perp} P_{L_{1}}$ is Hilbert-Schmidt which suffices to prove the claim, according to \cref{lem:FockSpacesUnitarilyEquivalent}.
\end{proof}

Finally, in \cref{sec:LagrangianTwistedSpinorBundle} we need to study the unitary group $\U(L)$ of a Lagrangian $L\in \mathrm{Lag}(V)$. We embed  into $\O_{\res}(V)$ by sending an operator $T \in \U(L)$ to the operator that acts by $T$ on $L$ and by $\alpha T \alpha$ on $\alpha(L)$.

\begin{lemma}
\label{lem:UnitarySubgroup}
The homogeneous space  $\O_L(V)/U(L)$ carries a unique structure of a Banach manifold, such that the projection $\O_L(V) \to \O_L(V)/U(L)$ is a submersion.
\end{lemma}

\begin{proof}
We  claim that $\U(L)\subset \O_L(V)$ is a submanifold. Having this, the lemma is a standard result  \cite[Proposition 1.6.11 of Chapter III]{Bourbaki1998}.  
 To show the claim, we note that, as topological spaces, both $\U(L)$ and $\O_{\res}(V)$ are metric spaces, where the metric on $\U(L)$ is given by $d(T_{1},T_{2}) = \| T_{1} - T_{2} \|$ and the metric on $\O_{\res}(V)$ is given by $d(g_{1},g_{2}) = \|g_{1} - g_{2} \| + \|[g_{1} - g_{2}, \mc{J}_{L}] \|_{2}$.
 It is clear that the inclusion $\U(L) \rightarrow \O_{\res}(V)$ is an isometry and hence a topological embedding. 
 It remains to show that the inclusion is an immersion, this can be checked by computing the derivative of the inclusion at the identity.
\end{proof}

\cref{lem:FockSpacesUnitarilyEquivalent} tells us that the map $\O_{\res}(V) \rightarrow \Lag_{L}(V), g \mapsto g(L)$ is surjective.
It is clear that this map descends to a map on the quotient $\O_{\res}(V)/\U(L) \rightarrow \Lag_{\res}(V), [g]\mapsto g(L)$.
This map is injective, which we see as follows.
Let $L' \in \Lag_{\res}(V)$ be arbitrary, and let $g_{1},g_{2} \in \O_{\res}(V)$ with the property that $g_{1}(L) = g_{2}(L) = L'$.
The map $g_{1}^{-1}g_{2}$, then preserves $L$, and thus lies in $\U(L)$, which implies that $[g_{1}] = [g_{2}] \in \O_{\res}(V)/\U(L)$.
We conclude that 
\begin{equation*}
\O_{\res}(V)/\U(L) \rightarrow \Lag_{\res}(V), [g]\mapsto g(L)
\end{equation*}
is a bijection. We view $\Lag_{\res}(V)$ as a Banach manifold through the identification with $\O_{L}(V)/\U(L)$.
\begin{remark}
 The Banach manifold $\Lag_{L}(V)$ turns out to be a rich geometric object, and appears frequently in the literature, see among others \cite[Chapter 7]{PS86}, \cite{Borthwick1992, Furutani2004, Waldron2009}.
 In particular, it is even a Hilbert manifold, and it supports a \emph{Pfaffian line bundle} (see \cite{SW}).
We remark that the description of the Banach manifold structure on $\Lag_{L}(V)$ given in \cite{PS86,Borthwick1992} is, at first glance, different from the one we just gave, but one can show that they in fact agree \cite{kristel2019b}.
\end{remark}

We complete this section with an elementary result which will be used in \cref{sec:LagrangianTwistedSpinorBundle}.
\begin{lemma}
\label{lem:splittingoverUL}
The central extension $\U(1) \to \Imp_L(V) \to \O_L(V)$ splits over $\U(L)$. More precisely, the map 
 $\sigma: \U(L) \rightarrow \Imp_{L}(V), g \mapsto \Lambda_{g}$ is a group homomorphism and a smooth section.
\end{lemma}
\begin{proof}
 It is clear that $\sigma$ is a group homomorphism. That $\sigma(g)$ is an implementer and implements $g$ follows from the identity
 \begin{equation*}
  \theta_{g}(a) \lact v = \Lambda_{g} a \lact \Lambda_{g}^{-1}v, \quad a \in \Cl(V), v \in \mc{F}.
 \end{equation*}
 It remains to show that $\sigma$ is smooth. It suffices to show that the map $\sigma: \U(L) \rightarrow \Imp_{L}(V)$ coincides with the local section $\sigma$ defined in \cite[Section 3.5]{Kristel2019} on the intersection of their domains, because this would imply that $\sigma$ is smooth in a neighbourhood of the identity, and hence smooth everywhere.
 To that end, suppose that $g = \exp (A) \in \U(L)$, for $A \in \lie{o}_{\res}(V)$ small enough.
 We have that $A$, just like $g$, is block diagonal, which implies that the element $\tilde{A}_{1}$ defined in \cite[Section 3.5]{Kristel2019} is equal to zero.
 It then follows easily that $\exp(\tilde{A}_{0}) = \Lambda_{\exp(A)} = \Lambda_{g}$.
\end{proof}

\subsection{Smooth representations on Fock space}

\label{sec:SmoothRepresentations}

In \cref{sec:SpinorBundleOnLoopSpaceI} we construct rigged  Hilbert space bundles and rigged \cstar-algebra bundles over Fr\'echet manifolds. This requires a detailed study of smoothness properties of the  representations obtained in \cref{sec:CliffFockAndSpin}; this is the goal of this section. We continue working in the setting of a complex Hilbert space $V$ with a real structure $\alpha$, and a Lagrangian subspace $L\subset V$ with corresponding Fock space $\mathcal{F}_{L}$.

As  a subgroup of $\U(\mc{F}_{L})$, the group $\Imp_{\res}(V)$ of implementers comes equipped with a unitary representation on $\mc{F}_{L}$.
As is typical for infinite dimensional representations, the action map 
$\Imp_{\res}(V) \times \mc{F}_{L} \rightarrow \mc{F}_{L}$
is not smooth.
The \emph{subspace of smooth vectors} in $\mc{F}_{L}$ is defined to be
\begin{equation*}
 \mc{F}^{\smooth}_{L} \defeq \{ v \in \mc{F}_{L} \mid \Imp_{\res}(V) \rightarrow \mc{F}_{L}, U \mapsto Uv \text{ is smooth} \}.
\end{equation*}

\begin{lemma}
 The set of smooth vectors $\mc{F}_{L}^{\smooth}$ contains the exterior algebra $\Lambda L$, and is hence a dense subspace of $\mc{F}_{L}$.
\end{lemma}
This has been shown in \cite[Proposition 3.17]{Kristel2019}, see also \cite[Section 10.1]{Ne09}.
We remark that we have not been able to give a concrete description of the set $\mathcal{F}_L^{\infty}$ in any of our examples (described in \cref{sec:spinorbundles}).

By definition of smooth vectors, the Lie algebra $\lie{imp}(V)$ of $\Imp_L(V)$ acts infinitesimally on $\mathcal{F}_L^{\smooth}$; i.e., for $X \in \lie{imp}(V)$ and  $v \in \mc{F}^{\smooth}_{L}$, we may define
\begin{equation*}
 Xv \defeq \der{}{t} \bigg|_{t=0} \exp(t X)(v).
\end{equation*}
Let $\mc{P}(\mc{F}_{L})$ be the set of all continuous semi-norms on $\mc{F}_{L}$.
We define a topology on $\mc{F}_{L}^{\smooth}$ by the following family of semi-norms, \cite[Section 4]{neebdiffvect}:
\begin{equation*}
 p_{n}(v) = \sup \{ p( X_{1} \dots X_{n} v) \mid X_{i} \in \lie{imp}(V), \|X_{i}\| \leqslant 1 \}, \quad p \in \mc{P}(\mc{F}_{L}), n \in \mathbb{N}_{0}.
\end{equation*}
We write $s: \mc{F}_{L}^{\smooth} \rightarrow \R_{\geqslant 0}$ for the restriction of the norm on $\mc{F}_{L}$ to $\mc{F}_{L}^{\smooth}$.
Because $s= s_{0}$ is one of the semi-norms defining the topology on $\mc{F}_{L}^{\smooth}$ we have the following result.
\begin{lemma}\label{prop:SmoothFockSpace}
 The map $\mc{F}_{L}^{\smooth} \hookrightarrow \mc{F}_{L}$ is continuous.
\end{lemma}

The following result is \cite[Theorem 4.4 and Proposition 5.4]{neebdiffvect}.

\begin{lemma}\label{lem:ImpActsOnF}
Equipped with the semi-norms $p_{n}$, the space $\mc{F}_{L}^{\smooth}$ is complete, and hence Fr\'echet.
 Moreover, the action map $\Imp_{\res}(V) \times \mc{F}_{L}^{\smooth} \rightarrow \mc{F}_{L}^{\smooth}$ is smooth.
\end{lemma}

We may reformulate the last two lemmas in the language of rigged Hilbert spaces and smooth representations, as explained \cref{sec:HilbertBundle}, \cref{def:RiggedHilbertSpace,def:smoothreprhs}. 

\begin{proposition}
\label{prop:smoothFockSpace}
The space of smooth vectors $\mathcal{F}^{\smooth}_L$ is a rigged Hilbert space, and it carries a smooth  representation of the Banach Lie group $\Imp_{\res}(V)$.
\end{proposition}

Just like the action map $\Imp_{\res}(V) \times \mc{F}_{L} \rightarrow \mc{F}_{L}$ is not smooth, the  map $\O(V) \times \Cl(V) \rightarrow \Cl(V)$ for the action by Bogoliubov automorphisms is not smooth either, an issue that we handle in a similar way.
First, we introduce some notation.
We write $\Clsm$ for the subspace of smooth vectors in $\Cl(V)$, i.e.~
\begin{equation*}
 \Cl(V)^{\smooth} \defeq \{ a \in \Cl(V) \mid \O(V) \rightarrow \Cl(V), g \mapsto \theta_{g}(a) \text{ is smooth} \},
\end{equation*}
and consider  the infinitesimal action of $\lie{o}(V)$ on $\Clsm$, i.e.,
\begin{equation*}
 Ya \defeq \der{}{t}\bigg|_{t=0} \theta_{\exp(tY)}(a).
\end{equation*}
Let $\mc{R}(\Cl(V))$ be the set of continuous semi-norms on $\Cl(V)$.
The topology on $\Clsm$ is then defined by the family of semi-norms
\begin{equation*}
 r_{n}(a) = \sup \{ r( Y_{1} \dots Y_{n} a) \mid Y_{i} \in \lie{o}(V), \|Y_{i}\| \leqslant 1 \}, \quad r \in \mc{R}(\Cl(V)), n \in \mathbb{N}_{0}.
\end{equation*}
For $X = (Y,\lambda) \in \lie{imp}(V) = \lie{o}_{\res}(V) \oplus \R$ and $a \in \Clsm$ we define
\begin{equation*}
 Xa \defeq \der{}{t} \bigg|_{t=0} \exp(tX)a \exp(-tX)
\end{equation*}
Because $\exp(tX)$ implements $\exp(tY)$, we obtain
\begin{equation*}
 Xa = \der{}{t}\bigg|_{t=0} \exp(tX) a \exp(-tX) = \der{}{t}\bigg|_{t=0} \theta_{\exp(tY)}(a) = Ya
\end{equation*}
We then have, for all $a \in \Clsm$ and all $v \in \mc{F}_{L}^{\smooth}$,
\begin{equation}\label{eq:ActionIsADerivation}
 X (a \lact v) = Xa \lact v + a \lact Xv.
\end{equation}
We write $s: \Clsm \rightarrow \R_{\geqslant 0}$ for the restriction of the norm on $\Cl(V)$ to $\Clsm$.
We observe that if $r \in \mc{R}(\Cl(V))$, then $\overline{r}: \Cl(V) \rightarrow \R, a \mapsto r(a^{*})$ is a continuous semi-norm, because the involution is continuous.
Using the fact that $s=s_{0}$ is one of the semi-norms defining the topology on $\Clsm$, and the fact that $\O(V)$ acts on $\Cl(V)$ by \cstar-automorphisms, and hence that $r_{n}(a^{*}) = \overline{r}_{n}(a)$ for all $r \in \mc{R}(\Cl(V))$ and all $n \in \mathbb{N}_{0}$, we conclude that the following result holds.
\begin{lemma}\label{lem:CliffordFrechetAlgebra}
 The inclusion map $\Cl(V)^{\smooth} \hookrightarrow \Cl(V)$ is continuous.
 Moreover, the involution on $\Cl(V)$ restricts to a continuous involution on $\Cl(V)^{\smooth}$.
\end{lemma}
Using the fact that the action map $\O(V) \times \Cl(V) \rightarrow \Cl(V)$ is continuous, see \cite[Proposition 4.35]{Ambler2012}, together with \cite[Theorem 6.2]{neebdiffvect}, we obtain the following result.
\begin{lemma}
 Equipped with the semi-norms $r_{n}$, the algebra $\Clsm$ is a Fr\'echet algebra.
 Moreover, the action map $\O(V) \times \Clsm \rightarrow \Clsm$ is smooth.
\end{lemma}

Again, we  reformulate this in the language introduced in \cref{sec:HilbertBundle}, \cref{def:riggedcstaralgebra,def:smoothrepcstar}.  

\begin{proposition}
\label{prop:CliffordFrechetAlgebra}
The algebra of smooth vectors $\Clsm$ is a rigged \cstar-algebra, and it carries a smooth representation of the Banach Lie group $\mathrm{O}(V)$.
\end{proposition}

The following result tells us exactly in what sense the riggings  $\mc{F}_{L}^{\smooth}$ and $\Clsm$ are compatible.

\begin{proposition}\label{prop:SmoothCliffordActionOnF}
The Fock space representation $\Cl(V) \times \mc{F}_{L} \rightarrow \mc{F}_{L}$ restricts to a map $\Cl(V)^{\infty} \times \mc{F}_{L}^{\infty} \rightarrow \mc{F}_{L}^{\infty}$, and turns the rigged Hilbert space $\mathcal{F}_L^{\infty}$ into a rigged $\Cl(V)^{\infty}$-module. \end{proposition}

\begin{proof}
 A priori, the image of the action map is contained in $\mc{F}_{L}$ instead of $\mc{F}_{L}^{\smooth}$, however, a standard computation shows that if $a \in \Cl(V)^{\smooth}$ and $v \in \mc{F}_{L}^{\smooth}$, then the derivative of the map $\Imp_{L}(V) \rightarrow \mc{F}_{L}, U \mapsto U (a \lact v)$ at $\mathds{1} \in \Imp_{\res}(V)$ is given by $\lie{imp}(V) \rightarrow \mc{F}_{L}, X \mapsto X(a) \lact v + a \lact X(v)$, which implies that $a \lact v \in \mc{F}^{\smooth}_{L}$.
 To complete the proof, it suffices to show that the map $\Cl(V)^{\smooth} \times \mc{F}^{\smooth}_{L} \rightarrow \mc{F}^{\smooth}_{L}$ is continuous (see \cref{re:repriggedcstar}).
We start by deriving a useful estimate for $\| X_{1} \dots X_{n}(a \lact v)\|$. 
Equation \cref{eq:ActionIsADerivation} implies that for all $a \in \Clsm$, all $v \in \mc{F}_{L}^{\smooth}$, and all $X_{1},...,X_{n} \in \lie{imp}(V)$ we have the following non-commutative binomial expansion
\begin{equation}\label{eq:NCBinomial}
 X_{1}...X_{n}(a \lact v) = \sum_{k=0}^{n}\sum_{\sigma \in \lie{S}_{k}^{n}} X_{\sigma(1)}...X_{\sigma(k)}a \lact X_{\sigma(k+1)} ... X_{\sigma(n)} v,
\end{equation}
where $\lie{S}_{k}^{n}$ is the set of bijections $\sigma: \{1,...,n\} \rightarrow \{1,...,n\}$ such that the restrictions $\sigma: \{1,...,k\} \rightarrow \{ 1,...,n\}$ and $\sigma: \{k+1,...,n\}\rightarrow \{1,...,n\}$ are order preserving. We now suppose that $\| X_{i} \| \leqslant 1$ and compute for arbitrary $\sigma \in \lie{S}_{k}^{n}$
\begin{align*}
 \|  X_{\sigma(1)}...X_{\sigma(k)}a \lact X_{\sigma(k+1)} ... X_{\sigma(n)} v \| &\leqslant \| X_{\sigma(1)}...X_{\sigma(k)}a \| \| X_{\sigma(k+1)} ... X_{\sigma(n)} v  \| \leqslant s_{k}(a) s_{n-k}(v).
\end{align*}
From which we obtain the estimate
\begin{align*}
 \| X_{1} \dots X_{n} (a \lact v) \| \leqslant \sum_{k=0}^{n} \sum_{\sigma \in \lie{S}_{k}^{n}} s_{k}(a) s_{n-k}(v) = \sum_{k=0}^{n} \binom{n}{k} s_{k}(a) s_{n-k}(v).
\end{align*}
Let us write $\rho: \Clsm \times \mc{F}_{L}^{\smooth} \rightarrow \mc{F}_{L}^{\smooth}$ for the action map.
A subbasis of open neighbourhoods of zero in $\mc{F}^{\smooth}_{L}$ is given by the open neighbourhoods
 \begin{equation*}
  U(p,n,\eps) = \{ v \in \mc{F}_{L}^{\smooth} \mid p_{n}(v) < \eps \}, \quad p \in \mc{P}, \, n \in \mathbb{N}_{0}, \, \eps > 0.
 \end{equation*}
 We assume, without loss of generality, that $p \leqslant \| \cdot \|$.
 Hence, let $W \defeq U(p,n,\eps)$, with $p,n,\eps$ as above, be arbitrary.
 We shall prove that $\rho^{-1}(W)$ is open by finding for each $(a,v) \in \rho^{-1}(W)$ an open neighbourhood of $(a,v)$ contained in $\rho^{-1}(W)$.
 Indeed, for each $k =0,...,n$ we consider the open neighbourhoods
 $Z_{k} = \{ b \in \Clsm \mid s_{k}(b-a) < \eps \min \{ 2^{-n}, 2^{-n}/s_{n-k}(v) \} /3 \} $ of $a$, and $Z'_{k} = \{ w \in \mc{F}_{L}^{\smooth} \mid s_{n-k}(w-v) < \eps \min \{ 2^{-n}, 2^{-n}/s_{k}(a) \}/3 \}$ of $v$. The intersection $Z \defeq \cap_{k} Z_{k} \times Z'_{k}$ is then an open neighbourhood of $(a,v)$.
 Let $(b,w) \in \Clsm \times \mc{F}_{L}^{\smooth}$ be arbitrary, then we compute
 \begin{equation*}
  \rho(b,w) - \rho(a,v) = (b - a) \lact (w - v) + a \lact (w - v) + (b-a) \lact v.
 \end{equation*}
 We thus obtain the estimate, for $(b,w) \in Z$ and $X_{1}, \dots X_{n} \in \lie{imp}(V)$ with $\|X_{i} \| \leqslant 1$,
 \begin{align*}
  p( X_{1}...X_{n} (\rho(b,w) - \rho(a,v)))
  &\leqslant \| X_{1} \dots X_{n} ( \rho(b,w) - \rho(a,v)) \| \\
  &\leqslant \| X_{1} \dots X_{n} ((b-a) \lact (w-v)) \| + \|X_{1} \dots X_{n} (a \lact (w-v)) \| \\
  & \hspace{17em}+ \| X_{1} \dots X_{n} ((b-a) \lact v) \| \\
  & \leqslant \sum_{k=0}^{n} \binom{n}{k} 
  \big(   s_{k}(b-a)s_{n-k}(w-v) 
  +\, s_{k}(a) s_{n-k}(w-v) 
  \\[-1em]&\hspace{17em}+\, s_{k}(b-a) s_{n-k}(v)
   \big) \\
  & < \frac{\eps}{3} \sum_{k=0}^{n} \binom{n}{k} \bigg( \frac{2^{-2n}}{3} + 2^{-n} + 2^{-n} \bigg) \\
  &= \frac{\eps}{3} \left( 2 + \frac{2^{-n}}{3} \right) \\
  & < \eps\text{.}
 \end{align*}
 Which implies that
 \begin{equation*}
  p_{n}(\rho(b,w) - \rho(a,v)) = \sup \{ 
  p( X_{1}...X_{n} (\rho(b,w) - \rho(a,v))) \mid X_{i} \in \lie{imp}(V), \| X_{i} \| \leqslant 1 \} < \eps.
 \end{equation*}
 Hence $Z \subseteq \rho^{-1}(W)$, and we are done.
\end{proof}

The smooth representations of \cref{prop:smoothFockSpace,prop:CliffordFrechetAlgebra}, as well as the module structure of \cref{prop:SmoothCliffordActionOnF} all fit together and give the following result.

\begin{theorem}
\label{thm:repofimp}
The smooth representations of $\O_L(V)$ on $\Cl(V)^{\infty}$ and of $\Imp_L(V)$ on $\mathcal{F}_L^{\infty}$ form a smooth representation of the central extension $\U(1) \to \Imp_L(V)\to\O_L(V)$ on the rigged $\Cl(V)^{\infty}$-module $\mathcal{F}_L^{\infty}$ in the sense of \cref{def:repofce}.
\end{theorem}

\begin{proof}
Condition (a) of \cref{def:repofce} requires that $\U(1)\subset \Imp_L(V)$ act on $\mathcal{F}_L^{\infty}$ by scalar multiplication; this is clear as $\Imp_L(V)\subset \U(\mathcal{F})$. Condition (b) is the compatibility of the three involved actions: if $U\in \Imp_L(V)$ implements $g\in \O_L(V)$, $a\in \Cl(V)^{\infty}$, and $v\in \mathcal{F}_L^{\infty}$, the condition is 
\begin{equation*}
\theta_g(a) \lact Uv = U(a\lact v)\text{;}
\end{equation*}
this is precisely the condition that $U$ implements $g$ of \cref{eq:Implementer}.
\end{proof}

Finally, we provide two lemmas about the functoriality of the riggings, with respect to a change of the real Hilbert space.
These will be used in \cref{sec:impLiftingGerbe} in order to make the transition between the \emph{typical} fibre of a Fock bundle and its \emph{actual} fibres. Let $V'$ be another complex Hilbert space with real structure $\alpha'$, and let $\nu: V \to V'$ be an orthogonal transformation, i.e., a unitary isomorphism that commutes with the real structures.

\begin{lemma}
\label{lem:UniversalFrechetCliffordAlgebra}
 The isometric $\ast$-isomorphism $\Cl(\nu): \Cl(V) \rightarrow \Cl(V')$ restricts to an isometric  isomorphism $\Clsm \rightarrow \Cl(V')^{\smooth}$ of rigged \cstar-algebras with inverse (the restriction of) $\Cl(\nu^{-1})$.
\end{lemma}
\begin{proof}
 Let us show that if $a \in \Clsm$ then $\Cl(\nu)(a) \in \Cl(V')^{\smooth}$.
 To prove that, we need to show that, for all $a \in \Clsm$, the map $\O(V') \rightarrow \Cl(V')^{\smooth}, g \mapsto \theta_{g} \Cl(\nu)(a)$ is smooth.
 We decompose that map as follows
 \begin{equation*}
  \xymatrix@R=0.5em{
   \O(V') \ar[r] & \O(V) \ar[r] & \Cl(V) \ar[r] & \Cl(V') \\
   g \ar[r] & \nu^{-1} g \nu \ar[r] & \theta_{\nu^{-1}g\nu}(a) \ar[r] & \theta_{g} \Cl(\nu)(a)
  }
 \end{equation*}
 The first map is obviously smooth. The second map is smooth because $a \in \Clsm$. Writing
 \begin{equation*}
  \theta_{\nu^{-1}g\nu} = \Cl(\nu^{-1}) \theta_{g} \Cl(\nu),
 \end{equation*}
 we see that the last map is simply $\Cl(\nu)$, which is smooth because it is a linear isometry. Now, it remains to check that the resulting map $\Clsm \rightarrow \Cl(V')^{\smooth}$ is smooth; this is a standard computation.  
\end{proof}

The analogous result for the  smooth Fock space, proved in a similar way,  is the following. We notice that if $L\subset V$ is a Lagrangian, then  $\nu(L) \subset V'$ is a Lagrangian. We let $\Lambda_{\nu}: \mathcal{F}_L \to \mathcal{F}_{\nu(L)}$ denote the induced unitary map. 

\begin{lemma}
\label{lem:UnitaryRestrictedToSmoothVectors}
The unitary map $\Lambda_{\nu}: \mc{F}_{L} \rightarrow \mc{F}_{\nu(L)}$ restricts to an isometric  isomorphism
$\Lambda_{\nu}: \mc{F}_{L}^{\smooth} \rightarrow \mc{F}_{\nu(L)}^{\smooth}$ of rigged Hilbert spaces with inverse (the restriction of) $\Lambda_{\nu^{-1}}$. \end{lemma}

\begin{remark}
\label{re:intertwinernu}
The isometric isomorphism $\Cl(\nu):\Clsm \rightarrow \Cl(V')^{\smooth}$ of rigged \cstar-algebras  and the isometric isomorphism $\Lambda_{\nu}: \mathcal{F}_L^{\infty} \to \mathcal{F}_{\nu(L)}^{\infty}$ fit together into a commutative diagram
\begin{equation*}
\xymatrix{\Cl(V)^{\infty} \times \mathcal{F}_L^{\infty} \ar[r] \ar[d]_{\Cl(\nu) \times \Lambda_{\nu}} & \mathcal{F}_L^{\infty} \ar[d]^{\Lambda_{\nu}} \\ \Cl(V')^{\infty} \times \mathcal{F}_{\nu(L)}^{\infty} \ar[r] & \mathcal{F}_{\nu(L)}^{\infty}\text{.}}
\end{equation*}
In this sense, $\Lambda_{\nu}$ is an intertwiner along the algebra homomorphism $\Cl(\nu)$.
\end{remark}

\begin{remark}
\label{lem:SmoothIntertwiners}
The restriction to spaces of smooth vectors has no effect on the $\U(1)$-torsor of unitary equivalences between Fock space representations, considered in \cref{lem:TheTwoU1Torsors}. This will be interesting in \cref{sec:impLiftingGerbe}.
 To make the statement more precise, let $\U_{\Clsm}(\mathcal{F}^{\smooth}_{L_1},\mathcal{F}^{\smooth}_{L_2})$ denote the set of isometric isomorphisms $T: \mathcal{F}^{\smooth}_{L_1} \to \mathcal{F}^{\smooth}_{L_2}$ of rigged Hilbert spaces that commute with the action of all $a\in \Clsm$. Then, restriction and extension to completions result into a bijection
\begin{equation*}
\U_{\Clsm}(\mc{F}^{\smooth}_{L_1}, \mc{F}^{\smooth}_{L_2})\cong \U_{\Cl(V)}(\mc{F}_{L_1}, \mc{F}_{L_2}) \text{.}
\end{equation*}
Indeed, showing that extension $\U_{\Clsm}(\mc{F}^{\smooth}_{L_1}, \mc{F}^{\smooth}_{L_2})\to \U_{\Cl(V)}(\mc{F}_{L_1}, \mc{F}_{L_2})$ is well-defined is routine. Conversely, if $T \in  \U_{\Cl(V)}(\mc{F}_{L_1}, \mc{F}_{L_2})$, all we need to show is that $T(\mc{F}^{\smooth}_{L_1}) \subseteq \mc{F}_{L_2}^{\smooth}$.
 Hence, let $v \in \mc{F}^{\smooth}_{L_1}$.
 We claim that $Tv \in \mc{F}_{L_2}$ is a smooth vector.
 By \cite[Theorem 7.2]{neebdiffvect} it is sufficient to prove that the map $\Imp_{L_2}(V) \rightarrow \C, U \mapsto \langle U Tv, Tv \rangle = \langle T^{*}UTv,v \rangle$ is smooth in an open neighbourhood of the identity.
 In turn, it suffices to show that the map $\Imp_{L_2}(V) \rightarrow \Imp_{L_1}(V), U \mapsto T^{*}UT$ is smooth in an open neighbourhood of the identity.
 By \cref{lem:TheTwoU1Torsors} we may assume that $T = U_{0} \Lambda_{g}^{-1}$ for some $U_{0} \in \Imp_{L_2}(V)$ and $g\in \O_{L_2}(V)$ such that $g(L_2)=L_1$.
 Because $\Imp_{L_2}(V)$ is a Banach Lie group, the map $U \mapsto U_{0}^{*}U U_{0}$ is smooth.
 It remains to be shown that the map $C_{\Lambda_{g}}:\Imp_{L_2}(V) \rightarrow \Imp_{L_1}(V), U \mapsto \Lambda_{g} U \Lambda_{g}^{-1}$ is smooth in an open neighbourhood of the identity. This can be done by showing that it corresponds to the adjoint action $C_{g}: \lie{o}_{L_2}(V) \rightarrow \lie{o}_{L_1}(V)$ on the level of Lie algebras, which is smooth. 
\end{remark}

\section{Smooth Fock bundles}
\def\thedefinition{\thesection.\arabic{definition}}

\label{sec:smoothfockbundles}
\label{sec:spinstructures}
\label{sec:SpinorBundleOnLoopSpaceI}
\label{sec:TheAssociatedFockBundle}

The general setting in which we construct smooth Fock bundles consists of a Fr\'echet manifold $\mathcal{M}$ and a principal $\mathcal{G}$-bundle $\mathcal{E}$ over $\mathcal{M}$, for a Fr\'echet Lie group $\mathcal{G}$.
In order to bring this structure in contact with Fock spaces, we require the following structure.

\label{sec:generalsetup}

\begin{definition}
\label{def:fockextension}
A \emph{Fock extension} for a Fr\'echet Lie group  $\mathcal{G}$ is a quadruple $(V,\alpha,L,\omega)$ consisting of a complex Hilbert space $V$, a real structure $\alpha$ on $V$, a Lagrangian subspace $L\subset V$, and a smooth group homomorphism $\omega:\mathcal{G} \to \O_L(V)$.
\end{definition}

Suppose now given a Fock extension $(V, \alpha, L, \omega)$.
First, we obtain a smooth representation
\begin{equation*}
\mathcal{G} \times \Cl^{\infty}(V) \to \Cl^{\infty}(V)\text{,}
\end{equation*}
induced along the smooth group homomorphism $\omega$ from the smooth representation of $\O_L(V)$ on the rigged \cstar-algebra $\Cl^{\infty}(V)$, see \cref{prop:CliffordFrechetAlgebra}.
Using \cref{lem:RiggedCStarBundle}, we define the following.

\begin{definition}
\label{def:CliffordBundle}
The \emph{Clifford bundle} associated to $\mathcal{E}$ is the rigged \cstar-algebra bundle
\begin{equation*}
\Cl_V^{\infty}(\mathcal{E}) \defeq (\mathcal{E} \times \Cl^{\infty}(V))/\mathcal{G}\text{.}
\end{equation*}
\end{definition}
Second, we obtain a central extension of Fr\'echet Lie groups,
\begin{equation*}
1\to \U(1) \to \widetilde{\mathcal{G}} \to \mathcal{G} \to 1\text{,}
\end{equation*}
obtained by pullback of $\Imp_L(V)$ along $\omega$.
We may now consider the corresponding lifting problem (see \cref{def:lifts}).

\begin{definition}
\label{def:FockStructure}
A \emph{Fock structure} on $\mathcal{E}$ is a lift $\widetilde{\mathcal{E}}$ of the structure group of $\mathcal{E}$ to $\widetilde{\mathcal{G}}$.
\end{definition}
Third, the smooth representation $\Imp_L(V) \times \mathcal{F}_{L}^{\infty} \to \mathcal{F}_{L}^{\infty}$ of \cref{prop:smoothFockSpace} induces via the smooth group homomorphism $\widetilde{\mathcal{G}} \to \Imp_L(V)$, obtained from the definition of $\widetilde{\mathcal{G}}$ as a pullback, a smooth representation of $\widetilde{\mathcal{G}}$ on $\mathcal{F}^{\infty}_{L}$. Using \cref{lem:AssociatedHilbertBundle}, we give the main definition of this section. 
\begin{definition}
\label{def:fockbundle}
Suppose $\mathcal{E}$ is equipped with a Fock structure $\widetilde{\mathcal{E}}$.
The  \emph{Fock bundle} associated  $\widetilde{\mathcal{E}}$ is the rigged Hilbert space bundle
\begin{equation*}
\mathcal{F}^{\infty}(\widetilde{\mathcal{E}}) \defeq (\widetilde{\mathcal{E}} \times \mathcal{F}^{\infty}_L) / \widetilde{\mathcal{G}}
\end{equation*} 
over $\mathcal{M}$. 
\end{definition}

In this general setting, there is one important result about our Fock bundles, namely a version of Clifford multiplication of the Clifford bundle $\Cl_V^{\infty}(\mathcal{E})$ on the Fock bundle $\mathcal{F}^{\infty}(\widetilde{\mathcal{E}})$. It is a direct application of \cref{prop:InducedBundleRep}, only using the fact that the rigged $\Cl_V^{\infty}(\mathcal{E})$-module $\mathcal{F}^{\infty}$ is a smooth representation of the central extension $\U(1)\to \widetilde{\mathcal{G}} \to \mathcal{G}$ in the sense of \cref{def:repofce}. This, in turn, follows immediately from \cref{thm:repofimp} and the construction of $\widetilde{\mathcal{G}}$ as the pullback of $\Imp_L(V)$. Thus, we have the following result. 

\begin{theorem}
\label{th:cliffmult}
There is a unique map
\begin{equation*}
\Cl_V^{\infty}(\mathcal{E}) \times_{\mathcal{M}} \mathcal{F}^{\infty}(\widetilde{\mathcal{E}}) \to \mathcal{F}^{\infty}(\widetilde{\mathcal{E}})
\end{equation*}
such that
\begin{equation*}
([e,a],[\tilde e,v]) \mapsto [\tilde e,a\lact v]
\end{equation*}
for all $\tilde e \in \widetilde{E}$ projecting to $e\in \mathcal{E}$, and all $a\in \Cl^{\infty}(V)$ and $v\in \mathcal{F}^{\infty}_L$.
Moreover, this map is a smooth morphism of Fr\'echet vector bundles over $\mathcal{M}$, and   exhibits the Fock bundle $\mathcal{F}^{\infty}(\widetilde{\mathcal{E}})$ as a rigged $\Cl_V^{\infty}(\mathcal{E})$-module bundle.
\end{theorem}

We postpone examples and application to \cref{sec:spinorbundles}, so that we can immediately apply the twisted version of the theory developed in \cref{sec:twisted}.

\begin{remark}
Our theory of rigged Hilbert space bundles, rigged \cstar-algebra bundles, and smooth representations has the following nice implications for Fock bundles and their Clifford multiplication.
\begin{enumerate}[(a)]

\item 
For each point $p\in \mathcal{M}$, the fibre $\Cl_V^{\infty}(\mathcal{E})_{p}$ is a rigged \cstar-algebra, the fibre $\mathcal{F}^{\smooth}(\widetilde{\mathcal{E}})_{p}$ is a rigged Hilbert  space, and  the fibrewise Clifford multiplication
\begin{equation*}
\Cl_V^{\infty}(\mathcal{E})_{p} \times \mathcal{F}^{\infty}(\widetilde{\mathcal{E}})_{p} \to \mathcal{F}^{\infty}(\widetilde{\mathcal{E}})_{p}
\end{equation*}
turns $\mathcal{F}^{\smooth}(\widetilde{\mathcal{E}})_{p}$ into a rigged $\Cl_V^{\infty}(\mathcal{E})$-module (\cref{re:repcstarbundle}). 

\item
Taking the fibrewise Hilbert completion in the bundle $\mathcal{F}^{\infty}(\widetilde{\mathcal{E}})$ results by \cref{rem:ContinuousHilbertBundle} in a continuous Hilbert space bundle $\mathcal{F}(\widetilde{\mathcal{E}})$ with typical fibre the Fock space $\mathcal{F}_{L}$.
\item
Taking the fibrewise norm completion in the bundle $\Cl_V^{\infty}(\mathcal{E})$ results by \cref{rem:fibrewiseCstar} in a  continuous bundle $\Cl_V(\mathcal{E})$ of \cstar-algebras   with typical fibre the Clifford \cstar-algebra $\Cl(V)$. 

\item
\cref{lem:continuous smoothrep} assures that Clifford multiplication is a continuous bundle morphism
\begin{equation*}
\Cl_V(\mathcal{E}) \times_{\mathcal{M}} \mathcal{F}(\widetilde{\mathcal{E}}) \to \mathcal{F}(\widetilde{\mathcal{E}})\text{,}
\end{equation*}  
which restrict  over each point $p$ to a representation of the Clifford \cstar-algebra $\Cl_V(\mathcal{E})_{p}$ on the Hilbert space $\mathcal{F}(\widetilde{\mathcal{E}})_{p}$.

\end{enumerate}
\end{remark}



\section{The twisted perspective to Fock bundles}
\def\thedefinition{\thesubsection.\arabic{definition}}

\label{sec:twisted}
\label{sec:twistedspinorbundles}

In this section we consider a Fr\'echet manifold $\mathcal{M}$ and a principal $\mathcal{G}$-bundle $\mathcal{E}$ over $\mathcal{M}$, for a Fr\'echet Lie group $\mathcal{G}$, and we consider a Fock extension  $(V,\alpha,L,\omega)$ for $\mathcal{G}$. We construct from this data the following three bundle gerbes:
\begin{itemize}

\item 
The \emph{Fock lifting gerbe}, which is the lifting gerbe for   Fock structures in the sense of \cref{def:FockStructure}. Its universal twisted rigged Hilbert space bundle is the \emph{twisted Fock bundle}.

\item
The \emph{implementer lifting gerbe}, which represents the obstruction against  a Fock bundle obtained  from fibrewise defined Fock spaces.    

\item
The \emph{Lagrangian Gra\ss mannian gerbe}, which represents the obstruction against a Fock bundle obtained from fibrewise Lagrangians.

\end{itemize}
Each of these bundle gerbes comes equipped with a version of a  twisted Fock bundle. We prove that all three bundle gerbes are canonically isomorphic, and that the isomorphisms exchange the individual twisted Fock bundles. If a trivialization of any of these bundle gerbes is provided, then each  twisted Fock bundle untwists precisely to the Fock bundle of \cref{def:fockbundle}.

\subsection{The twisted Fock bundle}

\label{sec:Focklifting}

In this section we  directly apply the abstract framework of lifting gerbes explained in \cref{sec:twistedriggedhb} to Fock structures. 
Associated to the problem of lifting the structure group of $\mathcal{E}$ along the central extension
$\U(1) \to \widetilde{\mathcal{G}} \to \mathcal{G}$
is a lifting gerbe, the \emph{Fock lifting gerbe} $\mathscr{L}_{\mathcal{E}}$. As a consequence of  \cref{prop:lifting} we have the following result.

\begin{proposition}
\label{prop:fockstructureslifting}
The category of Fock structures  on $\mathcal{\mathcal{E}}$ is canonically isomorphic to the category of trivializations of the Fock lifting gerbe $\mathscr{L}_{\mathcal{E}}$.
\end{proposition}

We consider the smooth representation $\widetilde{\mathcal{G}} \times \mathcal{F}_{L}^{\infty} \to \mathcal{F}_{L}^{\infty}$, for which we note that the central subgroup $\U(1) \subset \widetilde{\mathcal{G}}$ acts by scalar multiplication.

\begin{definition}
\label{def:twistedfock}
The \emph{twisted Fock bundle of $\mc{E}$} is the universal $\mathscr{L}_\mathcal{E}$-twisted rigged Hilbert space bundle $\mathcal{F}^{\infty}(\mathcal{E})^{\twist}$ associated to the smooth representation $\tilde{\mc{G}} \times \mc{F}_{L}^{\infty} \rightarrow \mc{F}_{L}^{\infty}$ according to \cref{def:canonicaltvb}.
\end{definition}

As explained in \cref{sec:SpinorBundleOnLoopSpaceI}, we also have a smooth representation $\mathcal{G} \times \Cl(V)^{\infty} \to \Cl(V)^{\infty}$, and the associated Clifford bundle $\Cl_V^{\infty}(\mathcal{E})$.
Together with the Fock space representation $\Cl(V)^{\infty} \times \mathcal{F}_{L}^{\infty} \to \mathcal{F}_{L}^{\infty}$, we have a smooth representation of the central extension $\U(1) \to \widetilde{\mathcal{G}} \to \mathcal{G}$.
Then, by \cref{lem:twistedmodulebundle},  the twisted Fock bundle $\mathcal{F}^{\infty}(\mathcal{E})^{\twist}$ is  an $\mathscr{L}_{\mathcal{E}}$-twisted rigged $\Cl_V^{\infty}(\mathcal{E})$-module bundle, and \cref{lem:reptwistedmb} implies the following result.

\begin{theorem}
\label{th:twistedspinor}
 Suppose $\mathcal{T}$ is a trivialization of the Fock lifting gerbe $\mathscr{L}_\mathcal{E}$ corresponding to a Fock structure $\widetilde{\mathcal{E}}$ under the isomorphism of \cref{prop:fockstructureslifting}, and let $\mathcal{F}^{\infty}(\widetilde{\mathcal{E}})$ be the associated Fock bundle. Then, there exists a canonical isometric isomorphism
\begin{equation*}
\twist_{\mathcal{T}}(\mathcal{F}^{\infty}(\widetilde{\mathcal{E}})) \cong \mathcal{F}^{\infty}(\mathcal{E})^{tw}
\end{equation*}   
of $\mathscr{L}_\mathcal{E}$-twisted rigged $\Cl_V^{\infty}(\mathcal{E})$-module bundles. In other words, untwisting the twisted Fock bundle $\mathcal{F}^{\infty}(\mathcal{E})^{\twist}$ gives the Fock bundle $\mathcal{F}^{\infty}(\widetilde{\mathcal{E}})$. 
\end{theorem}

\subsection{The frame-dependent Fock bundle}

\label{sec:impLiftingGerbe}
\label{sec:lagrangiangerbe}

This and the subsequent section are motivated by the following situation. We suppose that we have a  continuous Hilbert space bundle $\mathcal{H}$ over $\mathcal{M}$ with typical fibre $V$, see the explanations before \cref{rem:ContinuousHilbertBundle}. We assume that $\mathcal{H}$ is equipped with a real structure, i.e., a fibrewise anti-unitary continuous bundle morphism $\mathcal{H}\to \mathcal{H}$ that corresponds under local trivializations to the real structure $\alpha$ on $V$. We will call this a \emph{continuous real Hilbert space bundle with typical fibre $V$}.

The aim is to construct a continuous bundle of Fock spaces from $\mathcal{H}$.
For this purpose, one might want to choose, for each $p \in \mc{M}$, a Lagrangian $L_p\subset \mathcal{H}_p$, in such a way that the corresponding Fock spaces $\mathcal{F}_{L_p}$ form again a continuous Hilbert space bundle over $\mathcal{M}$. It turns out that this is too restrictive, and that it is more appropriate to choose, for each $p \in \mc{M}$, an \emph{equivalence class} of Lagrangians in $\mc{H}_{p}$.
However, by passing to equivalence classes, we lose the right to speak about \emph{the} Fock space $\mc{F}_{L_{p}}$, the result of which is that we shall obtain a projective, or twisted, Fock bundle.

There are many methods of controlling the choice of equivalence class of  Lagrangians; in the present section, we attack this problem by considering orthogonal frames of $\mathcal{H}_p$, i.e., orthogonal transformations $\nu:V \to \mathcal{H}_p$.
Given such a frame, we consider the Lagrangian $\nu(L) \subset \mathcal{H}_{p}$ and the corresponding Fock space $\mathcal{F}_{\nu(L)}$, which is a representation of the Clifford \cstar-algebra  $\Cl(\mathcal{H}_{p})$.
However, if $\nu'$ is another orthogonal frame in $\mathcal{H}_p$, then the two Fock spaces $\mathcal{F}_{\nu(L)}$ and $\mathcal{F}_{\nu'(L)}$ are  unitarily equivalent if and only if the element $g\defeq \nu^{-1}\circ\nu\in \O(V)$ is implementable, i.e., $g\in \O_L(V)$, see \cref{lem:FockSpacesUnitarilyEquivalent}.
In other words, for each $p \in \mc{M}$ we must choose a collection of frames of $\mc{H}_{p}$ on which $\O_{\res}(V)$ acts transitively.
Roughly speaking, this means that we need to reduce the structure group of the orthogonal frame bundle of $\mathcal{H}$ from $\O(V)$ to $\O_L(V)$.

As announced above, this will not be enough to construct a bundle Fock spaces from $\mc{H}$, indeed, we see from \cref{lem:TheTwoU1Torsors} that the set $\U_{\Cl(\mathcal{H}_{p})}(\mathcal{F}_{\nu'(L)},\mathcal{F}_{\nu(L)})$ of possible unitary equivalences is a $\U(1)$-torsor and canonically isomorphic to the fibre of $\Imp_L(V) \to \O_L(V)$ over $g$. Again, roughly speaking, this means that we need to lift the structure group of the reduced orthogonal frame bundle from $\O_L(V)$ to $\Imp_L(V)$.

In this article, we are interested in \emph{smooth} bundles of Fock spaces, and so must impose certain smoothness conditions on what should be the orthogonal frame bundle of $\mathcal{H}$.
It turns out that the following setting is appropriate. 

\begin{definition}
\label{def:srofb}
Let $\mathcal{H}$ be a continuous real Hilbert space bundle over a Fr\'echet manifold $\mathcal{M}$ with typical fibre $V$.  A \emph{smooth  orthogonal frame bundle} for $\mathcal{H}$ is a Fr\'echet principal $\O(V)$-bundle $\O(\mathcal{H})$ over $\mathcal{M}$ together with a  fibre-wise isometric  isomorphism 
\begin{equation*}
(\O(\mathcal{H}) \times V)/\O(V) \cong \mathcal{H}
\end{equation*}
of continuous real Hilbert space bundles over $\mathcal{M}$.

\end{definition}

\begin{remark}
\label{re:frameid}
\label{lem:framebundle}
A smooth orthogonal frame bundle $\O(\mathcal{H})$ is indeed (a version of) the orthogonal frame bundle of $\mathcal{H}$, in the sense that the fibre $\O(\mathcal{H})_p$ can be identified with the set $\O(V,\mathcal{H}_p)$ of orthogonal frames in $\mathcal{H}$ at $p\in \mathcal{M}$. To see this, denote by $\psi$ the isomorphism $(\O(\mathcal{H}) \times V)/\O(V) \cong \mathcal{H}$. If $\nu\in \O(\mathcal{H})$, then $v\mapsto \psi([\nu,v])$ is an orthogonal frame. Conversely, if $\varphi:V \to \mathcal{H}_p$ is an orthogonal frame, then -- since the $\O(V)$-action on $V$ is free -- there exists a unique $\nu\in \O(\mathcal{H})$ with $\psi^{-1}(\varphi(v))=[\nu,v]$ for all $v\in V$. In the following we will identify elements $\nu\in \O(\mathcal{H})_p$ with orthogonal frames $\nu:V \to \mathcal{H}_p$ without further notice.   
\end{remark}

\begin{remark}
Given a smooth orthogonal frame bundle $\O(\mathcal{H})$, it is clear what a reduction to  $\O_L(V)\subset \O(V)$ is: a Fr\'echet principal $\O_L(V)$-bundle $\O_L(\mathcal{H})$ over $\mathcal{M}$ together with a smooth $\O_L(V)$-equivariant bundle map $\O_L(\mathcal{H}) \to \O(\mathcal{H})$.  
\end{remark}

\begin{remark}
\label{re:Hfromgeneralsetting}
We recall our general setting, in which we start with a Fr\'{e}chet principal $\mc{G}$-bundle $\mc{E}$ over $\mc{M}$, and a Fock extension $(V,\alpha,L,\omega)$ for $\mathcal{G}$. Then, we consider the continuous representation $\mathcal{G} \times V \to V$  induced along the smooth group homomorphism $\mathcal{G} \to \O_L(V) \to \O(V)$ and consider the associated continuous Hilbert space bundle
\begin{equation}
\label{eq:hsbdlimp}
\mathcal{H}_{\mathcal{E}} \defeq (\mathcal{E} \times V)/\mathcal{G}\text{.}
\end{equation}
Its typical fibre is $V$, and it comes equipped with a real structure induced from $\alpha$. Next we provide a smooth  orthogonal frame bundle for $\mathcal{H}$, by considering the Fr\'echet principal $\O(V)$-bundle
\begin{equation*}
\O(\mathcal{H}_{\mathcal{E}}) \defeq \mathcal{E} \times_{\mathcal{G}} \O(V),
\end{equation*}
i.e.,  the extension of the structure group of $\mathcal{E}$ along the smooth group homomorphism $\mathcal{G} \to \O_L(V) \to \O(V)$, see \cref{lem:bundleext}.
By \cref{prop:InvarianceUnderLiftsHilbert} we have a canonical isomorphism
\begin{equation*}
(\O(\mathcal{H}_{\mathcal{E}}) \times V)/\O(V)\cong (\mathcal{E} \times V)/\mathcal{G} = \mathcal{H}_{\mathcal{E}}\text{.}
\end{equation*}
The identification of elements of $\O(\mathcal{H}_{\mathcal{E}})_p$ with orthogonal frames $V \to (\mathcal{H}_{\mathcal{E}})_p$ of \cref{re:frameid} becomes now the following: for $[e, g] \in \O(\mc{H}_{\mathcal{E}})_{p}$, the frame is $v \mapsto [e, gv]$.
Finally, $\O(\mathcal{H})$ comes equipped with a reduction to $\O_L(V)$ given by the Fr\'echet principal $\O_L(V)$-bundle
\begin{equation}
\label{def:olh}
\O_L(\mathcal{H}_{\mathcal{E}}) \defeq \mathcal{E} \times_{\mathcal{G}} \O_L(V)\text{,}
\end{equation}
the extension of the structure group of $\mathcal{E}$ along the smooth group homomorphism $\mathcal{G} \to \O_L(V)$. By construction, $\O_L(\mathcal{H}_{\mathcal{E}})$ is then a reduction of $\O(\mathcal{H}_{\mathcal{E}})$.  
\end{remark}

In the following, we continue with the assumption that our continuous real Hilbert space bundle $\mathcal{H}$ is equipped with a smooth orthogonal frame bundle $\O(\mathcal{H})$ as in \cref{def:srofb} and a reduction $\O_L(\mathcal{H})$, and will now explain the construction of a corresponding smooth twisted Fock bundle. 
The idea to define a bundle  of Fock spaces using orthogonal frames leads us to a rigged Hilbert space bundle $\mathcal{F}\frm{\mathcal{H}}$ over $\O_L(\mathcal{H})$, whose fibre over a frame $\nu\in \O_L(\mathcal{H})$ is $\mathcal{F}^{\infty}_{\nu(L)}$. We will call $\mathcal{F}\frm{\mathcal{H}}$ the \emph{frame-dependent Fock bundle}. 

In order to properly construct this bundle, we define the set $\mathcal{F}\frm{\mathcal{H}}$ to be the disjoint union of the fibres $\mathcal{F}^{\infty}_{\nu(L)}$. \cref{lem:UnitaryRestrictedToSmoothVectors} provides a bijection
\begin{equation}
\label{eq:71phi}
\phi:\O_L(\mathcal{H}) \times \mathcal{F}_{L}^{\infty} \to \mathcal{F}\frm{\mathcal{H}}: (\nu,v) \mapsto (\nu,\Lambda_{\nu}(v))
\end{equation}
which is fibrewise an isometric isomorphism; this equips $\mathcal{F}\frm{\mathcal{H}}$ with the structure of a  rigged Hilbert space bundle over $\O(\mathcal{H})$, with  global trivialization $\phi$. 

By construction, each fibre $\mathcal{F}^{\infty}_{\nu(L)}$ is a rigged $\Cl(\mathcal{H}_p)^{\infty}$-module, as $\nu(L)\subset \mathcal{H}_p$ is a Lagrangian. In order to capture the global aspects of this module structure, we introduce the rigged \cstar-algebra bundle
\begin{equation*}
\Cl^{\infty}(\mathcal{H}) \defeq (\O(\mathcal{H}) \times \Cl(V)^{\infty})/\O(V)
\end{equation*}
over $\mathcal{M}$. In a fibre over $p \in \mathcal{M}$, we obtain an isometric isomorphism of rigged \cstar-algebras
\begin{equation}
\label{eq:clid}
\Cl^{\infty}(\mathcal{H})_p \to \Cl(\mathcal{H}_p)^{\infty}: [\nu,a] \mapsto \Cl(\nu)(a)
\end{equation}
provided by \cref{lem:UniversalFrechetCliffordAlgebra}.
This shows that the bundle $\Cl^{\infty}(\mathcal{H})$ combines in an appropriate way the fibrewise defined rigged \cstar-algebras $\Cl(\mathcal{H}_p)^{\infty}$. In particular, the fibre $\Cl^{\infty}(\mathcal{H})_p$ acts on $\mathcal{F}^{\infty}_{\nu(L)}$.

Given the reduction $\O_L(\mathcal{H}) \to \O(\mathcal{H})$, \cref{prop:InvarianceUnderLiftsCStar} yields a canonical isomorphism
\begin{equation*}
\Cl(\mathcal{H})^{\infty} \cong (\O_L(\mathcal{H}) \times \Cl(V)^{\infty})/\O_L(V)\text{,}
\end{equation*}
of rigged \cstar-algebra bundles, which we shall use in the following.

\begin{lemma}
\label{lem:fockframemodule}
The action of the  fibre $\Cl^{\infty}(\mathcal{H})_p$ on $\mathcal{F}^{\infty}_{\nu(L)}$ exhibits the frame-dependent Fock bundle $\mathcal{F}\frm{\mathcal{H}}$ as a rigged $\pi^{*}\Cl^{\infty}(\mathcal{H})$-module bundle, where $\pi: \O_L(\mathcal{H}) \to \mathcal{M}$ is the bundle projection.
Explicitly, the action  in the fibre over a frame $\nu\in \O_L(\mathcal{H})$ is given by
\begin{equation*}
([\nu,a], v) \mapsto \Cl(\nu)(a)\lact v\text{,}
\end{equation*}
for $a\in \Cl(V)^{\infty}$ and $v\in \mathcal{F}^{\infty}_{\nu(L)}$.  
\end{lemma}

\begin{proof}
Only the smoothness of the action is to check. 
We observe that the pullback of $\Cl^{\infty}(\mathcal{H})$ along $\pi:\O_L(\mathcal{H}) \to \mathcal{M}$ is canonically isomorphic to the trivial  bundle $\O_L(\mathcal{H}) \times \Cl(V)^{\infty}$. Hence, the trivial bundle $\O_L(\mathcal{H}) \times \mathcal{F}^{\infty}$ is a rigged $\pi^{*}\Cl^{\infty}(\mathcal{H})$-module bundle. We claim that the isometric isomorphism $\phi: \O_L(\mathcal{H}) \times \Cl(V)^{\infty} \to \mathcal{F}\frm{\mathcal{H}}$ exchanges this $\pi^{*}\Cl^{\infty}(\mathcal{H})$-module structure with the one in question; this proves the lemma. The claim follows directly from \cref{re:intertwinernu}.  
\end{proof}

Next, we discuss the transformation behaviour of the frame-dependent Fock bundle under a change of frame. Since we have used the reduced frame bundle $\O_L(\mathcal{H})$, the fibres $\mathcal{F}\frm{\mathcal{H}}_{\nu_1}$ and $\mathcal{F}\frm{\mathcal{H}}_{\nu_2}$ over two frames $\nu_1,\nu_2\in \O_L(\mathcal{H})_{p}$ are unitarily equivalent $\Cl^{\infty}(\mathcal{H})_p$-representations. However, as mentioned above, no canonical equivalence is given; instead, consistent choices of unitary equivalences would require a lift of the structure group of $\O_L(\mathcal{H})$ along the central extension
$\U(1) \to \Imp_L(V) \to \O_L(V)$.
The existence of such a lift is a strong condition, and we will first continue without one.
Thus, we consider the lifting bundle gerbe $\mathscr{L}_{\mathcal{H}} \defeq \mathscr{L}_{\O_L(\mathcal{H})}$
for the problem of lifting the structure group of $\O_L(\mathcal{H})$ to $\Imp_L(V)$. We will call $\mathscr{L}_{\mathcal{H}}$ the \emph{implementer lifting gerbe} associated to $\mathcal{H}$.

We  now exhibit the  frame-dependent Fock bundle $\mathcal{F}\frm{\mathcal{H}}$ as an $\mathscr{L}_{\mathcal{H}}$-twisted rigged Hilbert space bundle, which allows us to employ the theory developed in \cref{sec:twistedriggedhb}. 
We recall that the principal $\U(1)$-bundle of $\mathscr{L}_{\mathcal{H}}$ over $\O_L(\mathcal{H})^{[2]}$ is $\mathcal{Q}\defeq \delta^{*}\Imp_L(V)$, and we recall from \cref{lem:TheTwoU1Torsors,lem:SmoothIntertwiners} that we have canonical bijections
\begin{equation*}
\mathcal{Q}_{\nu,\nu'}=\Imp_L(V)_{\delta(\nu,\nu')}\cong \U_{\Cl(\mathcal{H}_p)}(\mathcal{F}_{\nu'(L)},\mathcal{F}_{\nu(L)})\cong \U_{\Cl(\mathcal{H}_p)^{\infty}}(\mathcal{F}^{\infty}_{\nu'(L)},\mathcal{F}^{\infty}_{\nu(L)})\text{.}
\end{equation*}
Under these bijections, we have a map
\begin{equation*}
\varphi_{\nu,\nu'}:\mathcal{Q}_{\nu,\nu'} \times \mathcal{F}\frm{\mathcal{H}}_{\nu'} \to \mathcal{F}\frm{\mathcal{H}}_{\nu}: (U,v) \mapsto Uv
\end{equation*}
which induces the structure of a bundle morphism $\varphi: \mathcal{Q}_{\C} \otimes \pr_2^{*}\mathcal{F}\frm{\mathcal{H}} \to \pr_1^{*}\mathcal{F}\frm{\mathcal{H}}$.

\begin{lemma}
\label{lem:fockframe}
The bundle morphism $\varphi$ is an isometric isomorphism of rigged Hilbert space bundles, and  turns the frame-dependent Fock bundle $\mathcal{F}\frm{\mathcal{H}}$  into an $\mathscr{L}_{\mathcal{H}}$-twisted $\Cl^{\infty}(\mathcal{H})$-module bundle.
Moreover, the bundle morphism $\phi$ of \cref{eq:71phi} establishes an isometric isomorphism between $\mathcal{F}\frm{\mathcal{H}}$ and the universal $\mathscr{L}_{\mathcal{H}}$-twisted $\Cl^{\infty}(\mathcal{H})$-module bundle $\mathcal{F}^{\infty}(\O_L(\mathcal{H}))^{\twist}$ with fibre $\mathcal{F}_{L}^{\infty}$.
\end{lemma}

\begin{proof}
Let us first note that behind the stage is the smooth representation of the central extension $\U(1)\to \Imp_L(V) \to \O_L(V)$ on the rigged $\Cl(V)^{\infty}$-module $\mathcal{F}_L^{\infty}$ from \cref{thm:repofimp}.
Then, we recall that the universal $\mathscr{L}_{\mathcal{H}}$-twisted $\Cl^{\infty}(\mathcal{H})$-module bundle $\mathcal{F}^{\infty}(\O_L(\mathcal{H}))^{\twist}$ consists of the trivial bundle $\O_L(\mathcal{H}) \times \mathcal{F}^{\infty}$ together with the isometric isomorphism
$\nu: \mathcal{Q}_{\C} \otimes \pr_2^{*}\mathcal{F}^{\infty}(\O_L(\mathcal{H}))^{\twist} \to \pr_1^{*}\mathcal{F}^{\infty}(\O_L(\mathcal{H}))^{\twist}$
over $\O_L(\mathcal{H})^{[2]}$  characterized by $\nu([U,z],v) = z(U v)$ for all $U\in \Imp_L(V)$, $z\in\C$ and $v\in \mathcal{F}^{\infty}$.
We notice by inspection of the definitions of the involved maps that the diagram
\begin{equation*}
\xymatrix@C=4em{\mathcal{Q}_{\C} \otimes \pr_2^{*}\mathcal{F}^{\infty}(\O_L(\mathcal{H}))^{\twist} \ar[r]^-{\nu} \ar[d]_{1 \otimes \pr_2^{*}\phi} & \pr_1^{*}\mathcal{F}^{\infty}(\O_L(\mathcal{H}))^{\twist} \ar[d]^{\pr_1^{*}\phi} \\ \mathcal{Q}_{\C} \otimes \pr_2^{*}\mathcal{F}\frm{\mathcal{H}} \ar[r]_-{\varphi} & \pr_1^{*}\mathcal{F}\frm{\mathcal{H}}}
\end{equation*}
is commutative. Since $\phi$ is a global trivialization of $\mathcal{F}\frm{\mathcal{H}}$, this proves all claims at once.
\end{proof}

We recall from \cref{sec:twistedriggedhb} that any trivialization~$\mathcal{T}$ of  the implementer lifting gerbe $\mathscr{L}_{\mathcal{H}}$  provides an equivalence
\begin{equation*}
\twist_{\mathcal{T}}:\VBdl^{\Cl^{\infty}(\mathcal{H})}(\mathcal{M}) \to \Mod{\mathscr{L}_{\mathcal{H}}}^{\Cl^{\infty}(\mathcal{H})}\text{,}
\end{equation*}
between rigged $\Cl^{\infty}(\mathcal{H})$-module bundles over $\mathcal{M}$ and $\mathscr{L}_{\mathcal{H}}$-twisted rigged  $\Cl^{\infty}(\mathcal{H})$-module bundles. In particular, we have the following result:

\begin{theorem}
\label{th:unwistfdfb}
Let $\mathcal{H}$ be a continuous real Hilbert space bundle over $\mathcal{M}$ with typical fibre $V$, equipped with a smooth  orthogonal frame bundle and a reduction to $\O_L(V)$.
Let $\mathcal{T}$ be a trivialization of the implementer lifting gerbe $\mathscr{L}_{\mathcal{H}}$. Then, there exists a unique (up to unique isometric isomorphism) rigged $\Cl^{\infty}(\mathcal{H})$-module bundle $\mathcal{F}\frm{\mathcal{H},\mathcal{T}}$ over $\mathcal{M}$ that is an untwisted version of the frame-dependent Fock bundle $\mathcal{F}\frm{\mathcal{H}}$:
\begin{equation*}
\twist_{\mathcal{T}}(\mathcal{F}\frm{\mathcal{H},\mathcal{T}})\cong \mathcal{F}\frm{\mathcal{H}}\text{.}
\end{equation*} 
\end{theorem}

The theorem tells us that the goal of constructing a smooth bundle of Fock spaces from the given real Hilbert space bundle $\mathcal{H}$ is achieved when a trivialization $\mathcal{T}$ of the implementer lifting gerbe $\mathscr{L}_{\mathcal{H}}$  is provided. Tracing back through the relevant definitions, the fibre of the rigged Hilbert space bundle $\mathcal{F}\frm{\mathcal{H},\mathcal{T}}$ over a point $p\in \mathcal{M}$ is canonically 
\begin{equation*}
\mathcal{F}\frm{\mathcal{H},\mathcal{T}}_p \cong (\mathcal{T}_{\C}^{*})_{\nu} \otimes \mathcal{F}^{\infty}_{\nu(L)}    
\end{equation*} 
for any  frame $\nu\in \O_L(\mathcal{H})_p$.
Thus, the fibre we tried to use naively, $\mathcal{F}^{\infty}_{\nu(L)}$, is corrected with a complex line coming from the trivialization $\mathcal{T}$.

We conclude this section by proving that the frame-dependent Fock bundle nicely agrees with the twisted Fock bundle of \cref{sec:Focklifting}.
We consider a Fr\'echet principal $\mathcal{G}$-bundle $\mathcal{E}$ over $\mathcal{M}$, and induce the continuous real Hilbert space bundle $\mathcal{H}_{\mathcal{E}}$ and its reduced smooth  orthogonal frame bundle $\O_L(\mathcal{H}_{\mathcal{E}}) \to \O(\mathcal{H}_{\mathcal{E}})$ as described in \cref{re:Hfromgeneralsetting}.
Then, we only have to notice two facts: first,  we have a commutative diagram
\begin{equation*}
\xymatrix{\U(1) \ar@{=}[d] \ar[r] & \widetilde{\mathcal{G}} \ar[d] \ar[r] & \mathcal{G} \ar[d]^{\omega} \\ \U(1) \ar[r] & \Imp_{\res}(V) \ar[r] & \O_L(V)}
\end{equation*}
of central extensions, under which the smooth representation of $\U(1)\to\Imp_{\res}(V) \to \O_L(V)$ we have used in the present section induces the smooth representation of $\U(1)\to\widetilde{\mathcal{G}} \to \mathcal{G}$ used in \cref{sec:Focklifting}. Second, the projection $\mathcal{E} \to \O_L(\mathcal{H}_{\mathcal{E}})$ is $\omega$-equivariant. Now, \cref{lem:refinement} implies that this projection extends to a refinement refinement $\mathscr{R}: \mathscr{L}_{\mathcal{E}} \to \mathscr{L}_{\mathcal{H}_{\mathcal{E}}}$ between the Fock lifting gerbe and the  implementer lifting gerbe. Further, the Clifford bundles $\Cl_V^{\infty}(\mathcal{E})$ of \cref{def:CliffordBundle} and $\Cl^{\infty}(\mathcal{H}_{\mathcal{E}})$ are canonically isomorphic. Now,   \cref{prop:functorialityofuniversalmodule} implies the following result.

\begin{theorem}
\label{prop:refinement}
Let $\mathcal{E}$ be a Fr\'echet principal $\mathcal{G}$-bundle over $\mathcal{M}$, and let $\mathcal{H}_{\mathcal{E}}$ be the corresponding continuous real Hilbert space bundle. Then, the following statements hold.\begin{enumerate}[(a)]
\item 
The Fock lifting gerbe $\mathscr{L}_{\mathcal{E}}$ and the implementer lifting gerbe $\mathscr{L}_{\mathcal{H}_{\mathcal{E}}}$ are canonically isomorphic via the refinement  $\mathscr{R}: \mathscr{L}_{\mathcal{E}} \to \mathscr{L}_{\mathcal{H}_{\mathcal{E}}}$.

\item
The frame-dependent Fock bundle $\mathcal{F}\frm{\mathcal{H}_{\mathcal{E}}}$ corresponds under the refinement $\mathscr{R}$   to the twisted Fock bundle $\mathcal{F}^{\infty}(\mathcal{E})^{\twist}$ of \cref{def:twistedfock}. More precisely, there exists a canonical isometric isomorphism 
\begin{equation*}
\mathscr{R}^{*}\mathcal{F}\frm{\mathcal{H}_{\mathcal{E}}} \cong \mathcal{F}^{\infty}(\mathcal{E})^{\twist}
\end{equation*}
of $\mathscr{L}_{\mathcal{E}}$-twisted $\Cl^{\smooth}_V(\mathcal{E})$-module bundles. 
\end{enumerate}
\end{theorem}

Thus, both  twisted versions of smooth Fock bundles are equivalent. 
Using \cref{lem:refinementtriv,th:twistedspinor} we  deduce that the untwisted versions are equivalent, too.

\begin{corollary}
\label{co:untwisting1}
Let $\mathcal{T}$ be a trivialization of the implementer lifting gerbe, and let $\widetilde{\mathcal{E}}$ be the Fock structure on $\mathcal{E}$ that corresponds to the trivialization $\mathscr{R}^{*}\mathcal{T}$ of the Fock lifting gerbe under the isomorphism of \cref{prop:fockstructureslifting}. Then,
the untwisted frame-dependent Fock bundle $\mathcal{F}\frm{\mathcal{H}_{\mathcal{E}},\mathcal{T}}$ of \cref{th:unwistfdfb} is canonically isomorphic to the Fock bundle $\mathcal{F}^{\infty}(\widetilde{\mathcal{E}})$ of \cref{def:fockbundle} as rigged $\Cl_V^{\infty}(\mathcal{E})$-module bundles,
\begin{equation*}
\mathcal{F}\frm{\mathcal{H}_{\mathcal{E}},\mathcal{T}}\cong \mathcal{F}^{\infty}(\widetilde{\mathcal{E}})\text{.} 
\end{equation*}
\end{corollary}

\subsection{The polarization-dependent Fock bundle}

\label{sec:LagrangianTwistedSpinorBundle}

In this section we pick up the idea of the previous section to form a bundle of Fock space from a continuous real Hilbert space bundle $\mathcal{H}$ over $\mathcal{M}$ with typical fibre $V$, equipped with a smooth  orthogonal frame bundle $\O(\mathcal{H})$ and a reduction $\O_L(\mathcal{H}) \to \O(\mathcal{H})$, see \cref{def:srofb}.
In \cref{sec:impLiftingGerbe}, we used restricted orthogonal frames in order to induce the Lagrangians in this equivalence class.
In this section, we work with these Lagrangians directly, without reference to frames.

We recall that, for any $p\in \mathcal{M}$ and $\nu_1,\nu_2\in \O_L(\mathcal{H})_p$, the Lagrangians $\nu_1(L)$ and $\nu_2(L)$ in $\mathcal{H}_p$ are equivalent, and so represent a distinguished equivalence class $\mathcal{P}_p$ of Lagrangians in $\mathcal{H}_p$. We denote by $\O_{L,\mathcal{P}_p}(V,\mathcal{H}_p) \subset \O(V,\mathcal{H}_p)$ the subset of orthogonal frames $\nu$ in $\mathcal{H}_p$ such that $\nu(L)\in \mathcal{P}_p$, and obtain immediately the following result.

\begin{lemma}
The bijection $\O(\mathcal{H})_p \cong \O(V,\mathcal{H}_p)$ of \cref{lem:framebundle} restricts to a bijection $\O_L(\mathcal{H}_p) \cong \O_{L,\mathcal{P}_p}(V,\mathcal{H}_p)$.  
\end{lemma}

Before we continue, we recall from \cref{sec:CliffFockAndImp} that the homogeneous space $\O_L(V)/\U(L)$ is a Banach manifold and in bijection with the set $\Lag_L(V)$ of Lagrangians in $V$ that are equivalent to $L$. We used this bijection to induce a Banach manifold structure on $\Lag_L(V)$.  In the present situation, we consider the quotient
\begin{equation*}
\Lag_{L}(\mc{H}) \defeq \O_{\res}(\mc{H})/ \U(L)\text{.} 
\end{equation*}

\begin{lemma}
\label{lem:lagrangianprojection}
There exists a unique Fr\'echet manifold structure on $\Lag_{L}(\mathcal{H})$
such that the canonical projection $\O_L(\mathcal{H}) \to \Lag_L(\mathcal{H})$ is a submersion. Moreover, $\Lag_L(\mathcal{H})\to \mathcal{M}$ is a locally trivial fibre bundle over $\mathcal{M}$ with typical fibre $\Lag_L(V)$. 
\end{lemma}

\begin{proof}
Any local trivialization $\Phi: \O_L(\mathcal{H})|_{U} \to U \times \O_L(V)$ induces immediately a bijection $\Lag_L(\mathcal{H})|_{U} \cong U \times \Lag_L(V)$; these bijections equip $\Lag_L(\mathcal{H})$ with a Fr\'echet manifold structure, and all  claims are then straightforward consequences. 
\end{proof}

The importance of the fibre bundle $\Lag_L(\mathcal{H})$ lies in the following fact.

\begin{lemma}
\label{lem:lagrangianclass}
For each point $p\in \mathcal{M}$, the map $\Lag_L(\mathcal{H})_p \to \mathcal{P}_p:[\nu]\mapsto \nu(L)$ is a well-defined bijection. In other words, the fibre bundle $\Lag_L(\mathcal{H})$ parametrizes in a smooth way the distinguished equivalence class $\mathcal{P}_p$ of Lagrangians in $\mathcal{H}$. 
\end{lemma}

\begin{proof}
Well-definedness is clear, since elements of $\U(L)$ preserve $L$, and different frames $\nu$ lead to equivalent Lagrangians. If $L_p\in \mathcal{P}_p$ is any representative, we choose  $\nu\in \O_L(\mathcal{H})_p$ and obtain  two equivalent Lagrangians $L$ and $\nu^{-1}(L_p)$ in $V$. By \cref{lem:FockSpacesUnitarilyEquivalent} there exists $g\in \O_L(V)$ such that $g(L)=\nu^{-1}(L_p)$. We send $L_p \mapsto [\nu g]\in \Lag_L(\mathcal{H})_p$, and claim that this is well-defined. It is easy to see that $[\nu g]$ is independent of the choice of $\nu$. 
If $g'\in \O_L(V)$ also satisfies $g'(L)=\nu^{-1}(L)=g(L)$, then it follows from the fact that $\Lag_L(V)=\O_L(V)/\U(L)$ that $g'=g h$, for $h\in \U(L)$. This shows that $[\nu g]=[\nu g']$ in $\Lag_L(\mathcal{H})$.
Hence, the assignment  $L_p \mapsto [\nu g]$ is well-defined. Since it is inverse to the map $[\nu]\mapsto \nu(L)$, we are done.
\end{proof}

In the following we shall see that all structure that we discussed in \cref{sec:impLiftingGerbe} is $\U(L)$-equivariant and descends along the projection $\O_L(\mathcal{H}) \to \Lag_L(\mathcal{H})$. In particular, the frame-dependent Fock bundle $\mathcal{F}\frm{\mathcal{H}}$ descends to a rigged Hilbert space bundle that we call the \emph{polarization-dependent Fock bundle} and that we denote by $\mathcal{F}\pol{\mathcal{H}}$. Its fibre over a Lagrangian $L_p$ is the Fock space $\mathcal{F}^{\infty}_{L_p}$. For the construction, we consider the trivial rigged Hilbert space bundle $\O_L(\mathcal{H}) \times \mathcal{F}_L^{\infty}$ over $\O_L(\mathcal{H})$, to which the $\U(L)$-action lifts using the  section $\sigma:\U(L) \rightarrow \Imp_{\res}(V)$ of \cref{lem:splittingoverUL},  setting
\begin{equation}
\label{eq:actUL}
(\nu,v) \cdot g \defeq (\nu g,\sigma(g^{-1})v)\text{.}
\end{equation}
This action is smooth, as the section $\sigma$ is smooth and $\Imp_L(V) \times \mathcal{F}_L^{\infty} \to \mathcal{F}_L^{\infty}$ is smooth. Hence, the quotient bundle $(\O_L(\mathcal{H}) \times \mathcal{F}_L^{\infty})/\U(L)$ is  a rigged Hilbert space bundle over $\Lag_L(\mathcal{H})$.

Now, we let $\mathcal{F}\pol{\mathcal{H}}$ be the disjoint union of the fibres $\mathcal{F}_{L_p}^{\infty}$, and consider the map
\begin{equation*}
\psi: (\O_L(\mathcal{H}) \times \mathcal{F}_L^{\infty})/\U(L) \to \mathcal{F}\pol{\mathcal{H}}:[\nu,v] \mapsto (\nu(L),\Lambda_{\nu}(v))\text{,}
\end{equation*}
which is a well-defined bijection and turns $\mathcal{F}\pol{\mathcal{H}}$ into a rigged Hilbert space bundle. 
We obtain a commutative diagram
\begin{equation}
\label{eq:focklagdescend}
\xymatrix{\O_L(\mathcal{H}) \times \mathcal{F}_L^{\infty} \ar[r]^-{\phi} \ar[d] & \mathcal{F}\frm{\mathcal{H}} \ar[d] \\ (\O_L(\mathcal{H})\times \mathcal{F}^{\infty}_L)/\U(L) \ar[r]_-{\psi} & \mathcal{F}\pol{\mathcal{H}}}
\end{equation}
whose first row consists of rigged Hilbert space bundles over $\O_L(\mathcal{H})$, and whose second row consists of rigged Hilbert space bundles over $\Lag_{L}(\mathcal{H})$, and the vertical arrows are the projections in the quotients by $\U(L)$-actions. In particular, $\mathcal{F}\pol{\mathcal{H}} = \mathcal{F}\frm{\mathcal{H}}/\U(L)$, where the $\U(L)$-action on $\mathcal{F}\frm{\mathcal{H}}$ is induced from \cref{eq:actUL} along $\phi$; explicitly, the action of $g\in \U(L)$ takes $v\in \mathcal{F}^{\infty}_{\nu(L)}=\mathcal{F}\frm{\mathcal{H}}_{\nu}$ to $\Lambda_{\nu g^{-1}\nu^{-1}}(v)\in \mathcal{F}^{\infty}_{\nu(L)}=\mathcal{F}\frm{\mathcal{H}}_{\nu g}$.

The next step is to exhibit the polarization-dependent Fock bundle $\mathcal{F}\pol{\mathcal{H}}$ as a rigged $\pi^{*}\Cl^{\infty}(\mathcal{H})$-module bundle, where $\pi: \Lag_L(\mathcal{H}) \to \mathcal{M}$. Since the fibre $\Cl^{\infty}(\mathcal{H})_p$ can be identified canonically with $\Cl(\mathcal{H}_p)^{\infty}$, see \cref{eq:clid}, each fibre $\mathcal{F}^{\infty}_{L_p}$ is naturally a rigged $\Cl(\mathcal{H}_p)^{\infty}$-module. Moreover, this module structure coincides under the projection $\mathcal{F}\frm{\mathcal{H}}\to \mathcal{F}\pol{\mathcal{H}}$ to the $\Cl(\mathcal{H}_p)^{\infty}$-module structure of $\mathcal{F}\frm{\mathcal{H}}$ of \cref{lem:fockframemodule}. Hence, we conclude that the following result holds.

\begin{lemma}
\label{lem:lagfockmod}
The polarization-dependent Fock bundle $\mathcal{F}\pol{\mathcal{H}}$ is a rigged $\pi^{*}\Cl^{\infty}(\mathcal{H})$-module bundle, in such a way  that all maps in Diagram \cref{eq:focklagdescend} are intertwiners.
\end{lemma}

In order to exhaust all properties of the polarization-dependent Fock bundle $\mathcal{F}\pol{\mathcal{H}}$, the last step is to exhibit it as a twisted module bundle. In contrast to the twisted Fock bundle of \cref{sec:Focklifting} and the frame-dependent Fock bundle of \cref{sec:impLiftingGerbe}, the twisting bundle gerbe is not a lifting gerbe. It will be called  the \emph{Lagrangian Gra\ss mannian gerbe} $\mathscr{G}_{\mathcal{H}}$, and we shall construct it now. 
The first ingredient is the fibre bundle projection $\Lag_{\res}(\mc{H}) \rightarrow \mathcal{M}$, of which we have seen in \cref{lem:lagrangianprojection} that it is a surjective submersion. 

We note that $\Lag_L(\mathcal{H})^{[k]} = \O_L(\mathcal{H})^{[k]}/\U(L)^k$, where $\U(L)^k$ acts factor-wise.
We consider the the principal $\U(1)$-bundle $\mathcal{Q}=\delta^{*}\Imp_{L}(V)$ over $\O_L(\mathcal{H})^{[2]}$, which is the $\U(1)$-bundle of the implementer lifting gerbe $\mathscr{L}_{\mathcal{H}}$. 
The $\U(L)^2$-action lifts from  $\O_L(\mathcal{H})^{[2]}$ to $\mathcal{Q}$ using the section $\sigma:\U(L) \rightarrow \Imp_{\res}(V)$ of \cref{lem:splittingoverUL}, by setting
\begin{equation*}
(\nu_1,\nu_2,U)\cdot (g_1,g_2) := (\nu_1 g_1,\nu_2g_2,\sigma(g_1)^{-1}U\sigma(g_2))\text{.}
\end{equation*}
We remark that $U$ implements the unique $g\in \O_L(V)$ satisfying $\nu_2=\nu_1 g$. From this one can see that  $\sigma(g_1)^{-1}U\sigma(g_2)$ implements $g_1^{-1}gg_2$. Since $\nu_2 g_2 =\nu_1 g_1g_1^{-1}g g_2$, this shows that the action is well-defined.
Moreover, it preserves the $\U(1)$-action; hence, the bundle $\mathcal{Q}$ descends to a principal $\U(1)$-bundle $\mathcal{Q}' = \mathcal{Q}/\U(L)^2$ over $\Lag_L(\mathcal{H})^{[2]}$. 
Since the projections $\pr_{ij}:\O_L(\mathcal{H})^{[3]} \to \O_L(\mathcal{H})^{[2]}$ are equivariant along $\pr_{ij}:\U(L)^3 \to \U(L)^2$, we obtain  $\U(L)^{3}$-actions on the pullbacks $\pr_{ij}^{*}\mathcal{Q}$ such that $\pr_{ij}^{*}\mathcal{Q}'= \pr_{ij}^{*}\mathcal{Q}/\U(L)^{3}$.
It is straightforward to check that the isomorphism
\begin{equation*}
\mu: \pr_{12}^{*}\mathcal{Q} \otimes \pr_{23}^{*}\mathcal{Q} \to \pr_{13}^{*}\mathcal{Q}
\end{equation*}
of the implementer lifting gerbe,
which was defined using the multiplication in $\Imp_L(V)$, is $\U(L)^3$-equivariant.
Hence it descends to an isomorphism
\begin{equation*}
\mu': \pr_{12}^{*}\mathcal{Q}' \otimes \pr_{23}^{*}\mathcal{Q}' \to \pr_{13}^{*}\mathcal{Q}'
\end{equation*}
of principal $\U(1)$-bundles over $\Lag_L(\mathcal{H})^{[3]}$.
The associativity condition for $\mu$ implies the one for $\mu'$. This completes the construction of the Lagrangian Gra\ss mannian  gerbe $\mathscr{G}_{\mathcal{H}}$ over $\mathcal{M}$.

\begin{remark}
\label{re:intertwinerbundle}
The principal $\U(1)$-bundle $\mathcal{Q}'$ of $\mathscr{G}_{\mathcal{H}}$ can be interpreted nicely using \cref{lem:TheTwoU1Torsors,lem:SmoothIntertwiners}. Indeed, these results provide isomorphisms of $\U(1)$-torsors 
\begin{equation*}
\mathcal{Q}'_{L_{1},L_{2}}\cong \U_{\Cl(\mathcal{H}_{p})}(\mathcal{F}_{L_{1}},\mathcal{F}_{L_{2}})\cong \U_{\Cl(\mathcal{H}_p)^{\infty}}(\mathcal{F}^{\infty}_{L_1},\mathcal{F}^{\infty}_{L_2})\text{,}
\end{equation*}
under which the bundle gerbe product $\mu'$ becomes the composition of unitary equivalences. Using these isomorphisms, one can in fact equip the disjoint unions of the the torsors $\U_{\Cl(\mathcal{H}_{p})}(\mathcal{F}_{L_{1}},\mathcal{F}_{L_{2}})$ or $\U_{\Cl(\mathcal{H}_p)^{\infty}}(\mathcal{F}^{\infty}_{L_1},\mathcal{F}^{\infty}_{L_2})$
over all $(L_1, L_2)\in \Lag_L(\mathcal{H})^{[2]}$ with the structure of  Fr\'echet principal $\U(1)$-bundles, and then identify these with $\mathcal{Q}'$.   
  
\end{remark}

Finally, we consider the isometric isomorphism $\varphi: \mathcal{Q}_{\C} \otimes \pr_2^{*}\mathcal{F}\frm{\mathcal{H}} \to \pr_1^{*}\mathcal{F}\frm{\mathcal{H}}$ of \cref{lem:fockframe}, which turned the frame-dependent Fock bundle into an $\mathscr{L}_{\mathcal{H}}$-twisted rigged Hilbert space bundle. It is straightforward to verify that
$\varphi$ is $\U(L)^2$-equivariant,
and hence descends to an isometric isomorphism
\begin{equation*}
\varphi':\mathcal{Q}'_{\C} \otimes \pr_2^{*}\mathcal{F}\pol{\mathcal{H}} \to \pr_1^{*}\mathcal{F}\pol{\mathcal{H}}
\end{equation*}
over $\Lag_L(\mathcal{H})$, which moreover turns  $\mathcal{F}\pol{\mathcal{H}}$ into an $\mathscr{G}_{\mathcal{H}}$-twisted rigged Hilbert space bundle. We have verified in \cref{lem:fockframe} that the isomorphism $\varphi$ respects the $\pi^{*}\Cl^{\infty}(\mathcal{H})$-actions. It follows that the same holds for $\varphi'$. In the end, we have the following result.

\begin{proposition}
\label{prop:lagfocktwist}
The isometric isomorphism $\varphi'$ turns the polarization-dependent Fock bundle $\mathcal{F}\pol{\mathcal{H}}$ into an $\mathscr{G}_{\mathcal{H}}$-twisted $\Cl^{\infty}(\mathcal{H})$-module bundle. 
\end{proposition}

We now have the following consequence of general bundle gerbe theory, telling us how to construct smooth Fock bundles in the present context. 

\begin{theorem}
\label{th:unwistfdfbLag}
Let $\mathcal{H}$ be a continuous real Hilbert space bundle over $\mathcal{M}$ with typical fibre $V$, equipped with a smooth  orthogonal frame bundle and a reduction to $\O_L(V)$.
Let $\mathcal{T}$ be a trivialization of the Lagrangian Gra\ss mannian  gerbe $\mathscr{G}_{\mathcal{H}}$. Then, there exists a unique (up to unique isometric isomorphism) rigged $\Cl^{\infty}(\mathcal{H})$-module bundle $\mathcal{F}\pol{\mathcal{H},\mathcal{T}}$ over $\mathcal{M}$ that is an untwisted version of the polarization-dependent Fock bundle $\mathcal{F}\pol{\mathcal{H}}$:
\begin{equation*}
\twist_{\mathcal{T}}(\mathcal{F}\pol{\mathcal{H},\mathcal{T}})\cong \mathcal{F}\pol{\mathcal{H}}\text{.}
\end{equation*}
\end{theorem}

Explicitly, the fibre of the rigged Hilbert space bundle $\mathcal{F}\pol{\mathcal{H},\mathcal{T}}$ over a point $p\in \mathcal{M}$ is canonically 
\begin{equation*}
\mathcal{F}\pol{\mathcal{H},\mathcal{T}}_p \cong (\mathcal{T}_{\C}^{*})_{L_p} \otimes \mathcal{F}^{\infty}_{L_p}    
\end{equation*} 
for any Lagrangian  $L_p \subset \mathcal{H}_p$ in $\Lag_L(\mathcal{H})$.
We see again that the fibre we tried to use naively, $\mathcal{F}^{\infty}_{L_p}$, is corrected with a complex line coming from the trivialization $\mathcal{T}$.

It remains to establish  the precise relation between the polarization-dependent Fock bundle $\mathcal{F}\pol{\mathcal{H}}$ and the frame-dependent Fock bundle $\mathcal{F}\frm{\mathcal{H}}$, which in turn is equivalent to the twisted Fock bundle $\mathcal{F}^\infty(\mathcal{E})^\twist$ by \cref{prop:refinement}. Due to the definition of $\mathscr{G}_{\mathcal{H}}$ and $\mathcal{F}\pol{\mathcal{H}}$ as quotients of $\mathscr{L}_{\mathcal{H}}$ and $\mathcal{F}\frm{\mathcal{H}}$, the following is in fact immediately true.

\begin{theorem}\
\label{prop:equivambler}
\begin{enumerate}[(a)]

\item 
The canonical projections $\O_L(\mathcal{H}) \to \Lag_L(\mathcal{H})$ and $\mathcal{Q} \to \mathcal{Q}'$ form a refinement $\mathscr{P}: \mathscr{L}_{\mathcal{H}} \to \mathscr{G}_{\mathcal{H}}$ from the implementer lifting gerbe to the Lagrangian Gra\ss mannian gerbe; in particular, these bundle gerbes over $\mathcal{M}$ are canonically isomorphic. 

\item
The canonical projection $\mathcal{F}\frm{\mathcal{H}} \to \mathcal{F}\pol{\mathcal{H}}$ induces an isometric isomorphism 
\begin{equation*}
\mathcal{F}\frm{\mathcal{H}} \cong \mathscr{P}^{*}\mathcal{F}\pol{\mathcal{H}}
\end{equation*}
of $\mathscr{L}_{\mathcal{H}}$-twisted $\Cl^{\infty}(\mathcal{H})$-module bundles. 
\end{enumerate}
\end{theorem} 

We close with the following for the case that the continuous real Hilbert space bundle $\mathcal{H}$ is induced from a Fr\'echet principal $\mathcal{G}$-bundle $\mathcal{E}$ as in \cref{re:Hfromgeneralsetting}. The result follows from the general theory of \cref{lem:refinementtriv} and from \cref{th:twistedspinor,prop:refinement}.

\begin{corollary}
\label{co:untwisting2}
Let $\mathcal{T}$ be a trivialization of the Lagrangian Gra\ss mannian gerbe, let $\mathscr{P}^{*}\mathcal{T}$ be the pullback trivialization of the implementer lifting gerbe $\mathscr{L}_{\mathcal{H}_{\mathcal{E}}}$, and let $\widetilde{\mathcal{E}}$ be the Fock structure on $\mathcal{E}$ that corresponds to the trivialization $\mathscr{R}^{*}\mathscr{P}^{*}\mathcal{T}$ of the Fock lifting gerbe under the isomorphism of \cref{prop:fockstructureslifting}. Then, we have canonical isomorphisms between all three versions of Fock bundles,
\begin{equation*}
\mathcal{F}^{\infty}(\widetilde{\mathcal{E}}) \cong \mathcal{F}^{\infty}(\mathcal{H}_{\mathcal{E}},\mathscr{P}^{*}\mathcal{T})^{\mathrm{frm}}\cong \mathcal{F}\pol{\mathcal{H}_{\mathcal{E}},\mathcal{T}}\text{,}
\end{equation*}
as rigged $\Cl_V^{\infty}(\mathcal{E})$-module bundles over $\mathcal{M}$.
\end{corollary}

\section{Spinor bundles on loop space}

\label{sec:spinorbundles}
\label{sec:DiracLagrangians}
\label{sec:AnotherBundle}
\label{sec:bundlegerbe}

In this section we provide an application of the material we developed in \cref{sec:smoothfockbundles,sec:twisted}, in the context of 2-dimensional non-linear sigma models. In the simplest case, such models describe the motion of closed strings in a smooth manifold $M$. A closed string is a smooth map $\tau:S^1 \to M$, and thus just a  
point in free loop space of $M$,
\begin{equation*}
LM \defeq C^{\infty}(S^1,M)\text{.}
\end{equation*}     
Viewing closed strings in $M$ as points in $LM$ brings us into the well-studied context of point-like particles moving in a manifold, even if that manifold is now the infinite-dimensional Fr\'echet manifold $LM$. This analogy has proved to be very useful for the understanding of 2-dimensional sigma models.

For example, the discussion of point-like fermions in a finite-dimensional manifold $M$ requires a lift of the structure group of the frame bundle of $M$ to the spin group.
Killingback argued in \cite{killingback1} that fermionic strings in $M$, viewed as point-like particles in $LM$, require a similar lift of the structure group of the frame bundle of $LM$. Such a lift is called a \emph{loop spin structure on $M$} or \emph{spin structure on $LM$}. We  recall and explain the details in the first subsection below, and describe how our general theory of smooth Fock bundles can be used to construct  spinor bundles on the loop space $LM$.

\subsection{Loop spin structures}

\label{sec:loopspin}

We will discuss loop spin structures in a slightly more general setting, assuming that $E$ is a real vector bundle over $M$ of finite rank $d$, equipped with a Riemannian metric. In the application to fermionic strings, $E=TM$ is the tangent bundle of a Riemannian manifold  $M$. As a preliminary assumption, we suppose that $E$ is equipped with a spin structure, which we regard as a  principal $\Spin(d)$-bundle $\Spin(E)$ over $M$ together with a smooth map $\Spin(E) \to \O(E)$ that is equivariant along the group homomorphism $\Spin(d) \to \SO(d) \subset \mathrm{O}(d)$. 
We note (see, e.g., \cite[Proposition 1.8]{SW}) that taking free loops in $\Spin(E)$ results in a Fr\'echet principal $L\Spin(d)$-bundle $L\Spin(E)$ over $LM$, where $L\Spin(d)$ is the loop group of $\Spin(d)$. We recall that $L\Spin(d)$ has a distinguished central extension
\begin{equation*}
\U(1) \to \widetilde{L\Spin}(d) \to L\Spin(d)\text{,}
\end{equation*}   
the \emph{basic central extension} \cite{PS86}.
The following definition is Killingback's \cite{killingback1} version of a spin structure on loop space.

\begin{definition}
\label{def:spinstructureLM}
Let $E$ be a real vector bundle of rank $d$ over a smooth manifold $M$ equipped with a Riemannian metric and  a spin structure $\Spin(E)$.
A \emph{loop spin structure on $E$}  is a lift of  the structure group of   $L\Spin(E)$ along the basic central extension
\begin{equation*}
 \U(1) \rightarrow \BCE \rightarrow L\Spin(d).
\end{equation*}
\end{definition}

Thus, a loop spin structure on $E$ is a Fr\'echet principal $\BCE$-bundle $\widetilde{L\Spin}(E)$ over $LM$ together with a bundle map
$\widetilde{L\Spin}(E) \to L\Spin(E)$
that intertwines the group actions along the projection $\smash{\widetilde{L\Spin}}(d) \to L\Spin(d)$. 
The following result was proved by McLaughlin \cite{mclaughlin1}. 

\begin{proposition}
\label{prop:classspinLM}
$E$ admits a loop spin structure if its first fractional Pontryagin class vanishes, 
\begin{equation*}
\frac{1}{2}p_1(E)=0\text{.}
\end{equation*}
\end{proposition}

\begin{remark}
The fact that in \cref{prop:classspinLM} we only have \quot{if} but not \quot{iff}  indicates that Killingback's notion of a loop spin structure is incomplete. Indeed, the condition $\frac{1}{2}p_1(E)=0$ is equivalent to the existence of a \emph{string structure} on $E$, which can be seen as a lift of the structure group of $\Spin(E)$ to the string group, or, equivalently, as a trivialization of the Chern-Simons 2-gerbe associated to $\Spin(E)$, see \cite{waldorf8,Nikolausa}. It has been shown in \cite{Waldorfb} that a complete loop space-theoretic definition involves two  structures in addition to a loop spin structure: a \emph{fusion product} and a \emph{thin homotopy equivariant structure}.
The inclusion of both additional structures makes loop spin structures on $E$ equivalent to string structures on $E$, and makes the claim of \cref{prop:classspinLM} an \quot{iff}. For the present article,  these additional structures are not relevant though. 
However, we remark that a string structure on $E$ induces canonically a loop spin structure on $E$ using a procedure called transgression. 
\end{remark}

In order to bring loop spin structures into the context of \cref{sec:smoothfockbundles,sec:twisted}, we
consider the Fr\'echet manifold $\mathcal{M}:=LM$, the Fr\'echet Lie group $\mathcal{G}:=L\Spin(d)$, and the Fr\'echet principal $\mathcal{G}$-bundle  $\mathcal{E}=L\Spin(E)$. We look at Fock extensions (\cref{def:fockextension}) whose corresponding central extension of $L\Spin(d)$ is the basic central extension. The following hold for all such Fock extensions:
\begin{itemize}

\item 
Fock structures $\widetilde{L\Spin}(E)$ as defined in \cref{def:FockStructure} are precisely the same as loop spin structures  defined in \cref{def:spinstructureLM}. 

\item
The Fock lifting gerbe $\mathscr{L}_{L\Spin(E)}$ as defined in \cref{sec:Focklifting} is precisely the bundle gerbe that appeared under the name of the \emph{spin lifting gerbe of $L\Spin(E)$} \cite{Waldorfa,Waldorfb,Hanisch2017}; its trivializations are the loop spin structures.  \cref{prop:fockstructureslifting} reproduces the statement that trivializations of the  spin lifting gerbe are the same as loop spin structures on $E$ \cite[Section 4.1]{Waldorfa}.

\end{itemize}

All further structure, like the Fock bundle associated to loop spin structure,  the twisted Fock bundle of \cref{sec:Focklifting}, or the bundle gerbes of \cref{sec:impLiftingGerbe,sec:LagrangianTwistedSpinorBundle}, depend on representation-theoretic data that  varies with the choice of Fock extension.
In the following two subsections we discuss in detail two examples, $(V_{ev},\alpha_{ev},L_{ev},\omega_{ev})$ and $(V_{odd},\alpha_{odd},L_{odd},\omega_{odd})$, that have appeared in the literature before.
The construction of $(V_{ev},\alpha_{ev},L_{ev},\omega_{ev})$ requires $d$ to be even; while for $(V_{odd},\alpha_{odd},L_{odd},\omega_{odd})$ there is no such restriction on $d$.
Under a choice of a loop spin structure $\widetilde{L\Spin}(E)$ on $E$, these Fock extensions give rise to two versions of Fock bundles in the sense of \cref{def:fockbundle},
\begin{equation*}
\mathbb{S}^{\infty}_{ev/odd}(\widetilde{L\Spin}(E))\defeq \mathcal{F}_{L_{ev/odd}}^{\infty}(\widetilde{L\Spin}(E))
\text{,}
\end{equation*}  
the \emph{even} and \emph{odd spinor bundle} on $L\Spin(E)$, respectively. According to our general theory, they are rigged Hilbert space bundles over $LM$ with typical fibre the rigged Fock space $\mathcal{F}^{\infty}_{L_{ev/odd}}$. Moreover, they are rigged module bundles over  the Clifford bundles of \cref{def:CliffordBundle},
\begin{equation*}
\Cl_{ev/odd}^{\infty}(L\Spin(E))\defeq \Cl^{\infty}_{V_{ev/odd}}(\mathcal{E})\text{,}
\end{equation*} 
see \cref{th:cliffmult}.

One does not need the existence of a loop spin structure to apply \cref{prop:refinement,prop:equivambler} to the bundles $\mc{H}_{ev/odd} \defeq \mc{E} \times_{\mc{G}} V_{ev/odd}$ and obtain a diagram of isomorphisms of bundle gerbes
\begin{equation}
\label{eq:diagrambundlegerbes}
\xymatrix@R=0.5em@C=3em{& \mathscr{L}_{\mathcal{H}_{ev}} \ar[r] & \mathscr{G}_{\mathcal{H}_{ev}} \\ \mathscr{L}_{L\Spin(E)} \ar[dr] \ar[ur] \\ & \mathscr{L}_{\mathcal{H}_{odd}} \ar[r] & \mathscr{G}_{\mathcal{H}_{odd}}}
\end{equation}
where $\mathscr{L}_{\mathcal{H}_{ev/odd}}$ is the implementer lifting gerbe, and $\mathscr{G}_{\mathcal{H}_{ev/odd}}$ is the Lagrangian Gra{\ss}mannian gerbe (as explained in more detail in \cref{sec:impLiftingGerbe,sec:LagrangianTwistedSpinorBundle}).
The bundles $\mathcal{H}_{ev}$ and $\mathcal{H}_{odd}$ have concrete meanings in each case, as we explain in the following subsections.

All bundle gerbes in diagram \cref{eq:diagrambundlegerbes} come equipped with twisted module bundles of various geometric flavours, canonically isomorphic to either the even twisted Fock bundle $\mathcal{F}^{\infty}_{L_{ev}}(L\Spin(E))^{\twist}$  or the odd twisted Fock bundle $\mathcal{F}^{\infty}_{L_{odd}}(L\Spin(E))^{\twist}$ by \cref{prop:equivambler,prop:refinement}.
Let us explain what these twisted bundles are good for.  
First of all, they exist no matter if $E$ admits a loop spin structure or string structure. They are rigged Hilbert space bundles,  and come equipped with a twisted version of Clifford multiplication by $\Cl^{\infty}_{ev}(L\Spin(E))$ or $\Cl^{\infty}_{odd}(L\Spin(E))$, respectively. 
Upon providing a loop spin structure on $E$, they all untwist to either the even spinor bundle  or the odd one. 
Secondly, the problem of untwisting these twisted bundles  is reduced to finding some trivialization of some of the bundle gerbes in diagram \cref{eq:diagrambundlegerbes}. 

The third advantage is that one may use our twisted spinor bundles even when $E$ does not admit a loop spin structure.
This, for example, might be interesting for applications in $M$-theory. There, spacetime manifolds carry a principal $E_8$-bundle $F$, whose second Chern class (i.e., the second Chern class of the vector bundle associated to it along the adjoint representation $\rho:E_8 \to \mathrm{SU}(248)$) equals the first fractional Pontryagin number,
\begin{equation*}
c_2(F) = \frac{1}{2}p_1(M)\text{.}
\end{equation*}
In particular, $M$ is typically not a string manifold, and $TM$ is typically not equipped with a loop spin structure. 
Both classes can be realized by Chern-Simons 2-gerbes constructed from the basic gerbes over $\mathrm{SU}(248)$ and $\Spin(d)$, respectively (see \cref{sec:lagrangiangerbe}). An \emph{$F$-twisted string structure} is then an isomorphism between these 2-gerbes \cite{Sati2012}. One may observe that the pullback of the basic gerbe over $\mathrm{SU}(248)$ along $\rho$ is the 60th power of the basic gerbe over $E_8$, which implies that the first Chern-Simons 2-gerbe is isomorphic to one constructed from the bundle $F$ and $\mathcal{G}_{E_8}^{60}$. Under transgression to the loop space, an $F$-twisted spin structure is then an isomorphism between a bundle gerbe $\mathcal{G}_F$ and the spin lifting gerbe $\mathcal{S}_{LM}$, where $\mathcal{G}_F$ is the lifting gerbe associated to the problem of lifting the structure group of the looped bundle $LF$ from $LE_8$ to the 60th power of its basic central extension. The isomorphism $\mathcal{G}_F \cong \mathcal{S}_{LM}$ allows then to transfer the twisted spinor bundle into a twisted vector bundle for the bundle gerbe $\mathcal{G}_F$, hence only depending on the M-theoretical input data of the bundle $F$.

\subsection{The odd spinor bundle on loop space}

\label{sec:oddspinorbundle}

In this section we set up the Fock extension $(V_{odd},\alpha_{odd},L_{odd},\omega_{odd})$  of $L\Spin(d)$. We drop the subscript $(..)_{odd}$ in the following. We let $\mathbb{S} \rightarrow S^{1}$ be the odd spinor bundle on the circle, i.e., the one associated  to the odd (i.e., the connected, or bounding) spin structure on the circle.
We consider the Hilbert space $V \defeq L^{2}(S^{1}, \mathbb{S}) \otimes \C^{d}$ of odd spinors on the circle. Component-wise complex conjugation on $\C^d$ yields a real structure  $\alpha: V \rightarrow V$.
In \cite[Section 2]{Kristel2019} it is explained how the space of smooth $\C^d$-valued 2$\pi$-antiperiodic functions on the real line can be identified with the dense subspace $\Gamma(S^{1},\mathbb{S} )\otimes \C^{d} \subset V$.
Under this identification, $V$ has an orthonormal basis $\{ \xi_{n,j} \}_{n \in \mathbb{N}, j =1,...,d}$ by setting
\begin{equation}\label{eq:antiperiodicbasis}
 \xi_{n,j}(t) = e^{-i \left(n+ \frac{1}{2} \right) t} \; e_{j},
\end{equation}
where $\{ e_{j} \}_{j=1,...,d}$ is the standard basis of $\C^{d}$.
It is further shown that $\alpha(\xi_{n,j}) = \xi_{-n-1,j}$ for all $n$ and all $j$.
It follows that the closed linear span
\begin{equation*}
 L \defeq \operatorname{span} \{ \xi_{n,j} \mid n \geqslant 0, \, j=1,...,d \}
\end{equation*}
is a Lagrangian in $V$.

We note that the group $L\SO(d)$ acts on $V=L^{2}(S^{1}, \mathbb{S} )\otimes \C^{d}$ via its natural action on $\C^d$.
It is a standard result that this action factors through $\O_{\res}(V) \to \O(V)$ see \cite[Lemma 3.21]{Kristel2019} or \cite[Proposition 6.3.1]{PS86}.
Moreover, the map $L\SO(d) \rightarrow \O_{\res}(V)$ is smooth, see \cite[Lemma 3.22]{Kristel2019}.
We thus obtain a smooth group homomorphism $\omega:L\Spin(d) \to \O_L(V)$ by composition with the projection $L\Spin(d) \to L\SO(d)$.
This completes the construction of a Fock extension of $L\Spin(d)$.

\begin{proposition}[{{\cite[Theorem 3.26]{Kristel2019}}}]
\label{prop:ceodd}
For   $1\leq d \neq 4$, the pullback of the central extension $\U(1) \to\Imp_{\res}(V) \xrightarrow{q} \O_{\res}(V)$ along $\omega$ is the basic central extension of $L\Spin(d)$.
\end{proposition}

We continue under the assumption of \cref{prop:ceodd}, i.e.,   $1\leq d \neq 4$. Hence, a Fock structure on $L\Spin(E)$ is again the same thing as a loop spin structure on $E$. The corresponding Fock bundle
\begin{equation*}
\mathbb{S}^{\infty}_{odd}(\widetilde{L\Spin}(E)) \defeq \mathcal{F}^{\infty}(\widetilde{L\Spin}(E))
\end{equation*}
is called the \emph{odd loop spinor bundle} of $E$ associated to the loop spin structure $\widetilde{L\Spin}(E)$.
We recall that this is a rigged Hilbert space bundle with typical fibre $\mathcal{F}^{\infty}$, and that it is  a rigged 
module bundle for the Clifford bundle $\Cl^{\infty}(L\Spin(E))$.

Whether or not $E$ admits a loop spin structure, we have from \cref{sec:twisted} a sequence
\begin{equation*}
\mathscr{L}_{L\Spin(E)} \to \mathscr{L}_{\mathcal{H}} \to \mathscr{G}_{\mathcal{H}}
\end{equation*}
of bundle gerbes and refinements over $LM$, where $\mathcal{H}$ is the continuous real Hilbert space bundle 
\begin{equation*}
\mathcal{H} \defeq (L\Spin(E) \times V)/L\Spin(d)
\end{equation*}
defined in \cref{re:Hfromgeneralsetting}, whose real structure is induced from the one of $V$. In the present situation, this is a bundle with interesting fibres, as the following lemma shows.

\begin{lemma}
\label{lem:fibresoddspinor}
For each loop $\gamma\in LM$, there is a canonical isomorphism of real Hilbert spaces
\begin{equation*}
\mathcal{H}_{\gamma} \cong L^2(S^1,\mathbb{S} \otimes \gamma^{*}E_{\C})\text{,}
\end{equation*}
characterized uniquely by the condition that
$[\varphi,s \otimes w] \mapsto s \otimes \varphi_{\C} (w)$
for all $\varphi\in L\Spin(E)$, smooth sections $s$ of $\mathbb{S}$ and $w\in \C^d$; where $\varphi_{\C}(w)$ denotes the pointwise application of the complexification $\varphi(t)_{\C}$ of the frame $\varphi(t):\R^n \to E_{\gamma(t)}$ to $w$. 
\end{lemma}

\begin{proof}
It is clear that the formula characterizes the map uniquely, and the way $\Spin(d)$ acts in $V$ shows that it defines a map on the level of smooth sections. Since the frame $\varphi(t)$ is orthogonal, it is straightforward to see that this map is an isometry with respect to the $L^2$ scalar products, and hence extends to an isometric transformation of Hilbert spaces. It is easy to define an inverse map and to show that we have an isometric isomorphism. Finally, it is clear that the real structures are preserved under the complexification of a real-linear map. 
\end{proof}

By the lemma we may view our bundle $\mathcal{H}$ as the continuous real Hilbert space bundle with fibres $\mathcal{H}_{\gamma}=L^2(S^1,\mathbb{S} \otimes \gamma^{*}E_{\C})$. This bundle has been considered by Stolz and Teichner in their outline \cite[Section 3]{stolz3} of a construction of a spinor bundle on loop space, in the case $E=TM$. In order to form a bundle of Fock spaces, they proposed to consider an equivalence class of Lagrangians  by looking at the positive spectrum of the Dirac operator acting on $\mathcal{H}_{\gamma}$.

In the following discussion we will show that their outline is parallel to our construction of the polarization-dependent Fock bundle $\mathcal{F}\pol{\mathcal{H}}$  in \cref{sec:LagrangianTwistedSpinorBundle}. An important step in the construction of this Fock bundle was the reduction of the structure group of $\mathcal{H}$ from $\O(V)$ to $\O_L(V)$.
We observed that in the situation that $\O(\mathcal{H}) = L\Spin(E) \times_{L\Spin(d)} \O(V)$, there is a canonical reduction, denoted $\O_L(\mathcal{H})$.
We have then seen that this reduction selects a distinguished equivalence class of Lagrangians in the fibre $\mathcal{H}_{\gamma}$ over each loop $\gamma$; namely, the class consisting of the images $\nu(L)$ under all restricted orthogonal frames $\nu\in \O_L(\mathcal{H})$. In \cite{stolz3}, another procedure was described to obtain a distinguished equivalence class of Lagrangians. We will recall this and then prove that the two equivalence classes coincide.

We consider the twisted Dirac operator $\D_{\gamma}: \mc{H}_{\gamma} \rightarrow \mc{H}_{\gamma}$, which is defined by
\begin{equation*}
        \D_{\gamma} = i \dd{}{t} \otimes \nabla,
\end{equation*}
where $\nabla$ is the pullback of the covariant derivative corresponding to the Levi-Civita connection on $TM$.
We write $\operatorname{pt}_{\gamma}(t)\in \SO(T_{\gamma(0)}M, T_{\gamma(t)}M)$ for parallel transport along $\gamma$ for time $t$.
Let $\{v_{j}\}_{j=1,...,d}$ be a collection of linearly independent vectors in $T_{\gamma(0)}M_{\C}$ that diagonalizes $\operatorname{pt}_{\gamma}(2\pi) \in \SO( T_{\gamma(0)}M_{\C})$, with eigenvalues $e^{i \ph_{j}} \in S^{1}$. We choose $\ph_{j}$ to be in $[-\pi , \pi )$.
A straightforward computation shows the following result.

\begin{lemma}\label{lem:DiagonalizeDirac}
        The family of vectors $\eta_{n,j} \in \mc{H}_{\gamma}$ defined by
        \begin{equation*}
                \eta_{n,j}(t) = e^{-i (n + \frac{1}{2})t -i \frac{\ph_{j}}{2\pi}t} \otimes \operatorname{pt}_{\gamma}(t) v_{j},
        \end{equation*}
        diagonalizes $\D_{\gamma}$ with eigenvalues
        \begin{equation*}
                \lambda_{n,j} = n + \frac{1}{2} + \frac{\ph_{j}}{2\pi}.
        \end{equation*}
\end{lemma}

\begin{lemma}\label{lem:SplittingLagrangian}
        The subspace $\text{Eig}_{>0}(\D_{\gamma}) \subset \mc{H}_{\gamma}$ is a sublagrangian, and  a Lagrangian if and only if  $\text{Eig}_{=0}(\D_{\gamma})=\{0\}$.
\end{lemma}
\begin{proof}
        A computation using the basis of \cref{lem:DiagonalizeDirac} shows that $\alpha(\text{Eig}_{>0}(\D_{\gamma})) = \text{Eig}_{<0}(\D_{\gamma})$.
        The fact that $\text{Eig}_{>0}(\D_{\gamma})$ is isotropic follows from this and the fact that the basis $\eta_{n,j}$ is orthonormal. We generally have a decomposition
        \begin{equation*}
                \mc{H}_{\gamma} = \text{Eig}_{>0}(\D_{\gamma}) \oplus \alpha(\text{Eig}_{>0}(\D_{\gamma})) \oplus \text{Eig}_{=0}(\D_{\gamma}),
        \end{equation*}
        If the zero eigenspace $\text{Eig}_{\gamma = 0}(\D_{\gamma})$ is non-trivial, then $\text{Eig}_{>0}(\D_{\gamma})$ is not Lagrangian. However if $K \subset \text{Eig}_{=0}(\D_{\gamma})$ is Lagrangian, then $L_{\gamma} \defeq \text{Eig}_{>0}(\D_{\gamma}) \oplus K \subset \mc{H}_{\gamma}$ is Lagrangian.
        Such a choice of $K$ always exists, because $\text{Eig}_{=0}(\D_{\gamma})$ is even-dimensional, ($\lambda_{n,j} = 0$ corresponds to $n = 0$ and $\ph_{j} = -\pi$ and $\text{Eig}_{-1 }(\operatorname{pt}_{\gamma}(2\pi))$ is even dimensional).
\end{proof}

The sublagrangian $\mathrm{Eig}_{>0}(\D_{\gamma})$ is called the \emph{Dirac sublagrangian}.
By \cref{lem:sublagrangians} the Dirac sublagrangian determines an equivalence class of Lagrangians in $\mathcal{H}_{\gamma}$.

\begin{proposition}
\label{prop:equivalenceclasslag}
The two equivalence classes of Lagrangians in $\mathcal{H}_\gamma$, i.e., the one determined by the reduction $\O_L(\mathcal{H})$ and the one determined by the Dirac sublagrangian, are equal.
\end{proposition}

\begin{proof}
We consider the complexified tangent bundle $TM_{\C}$ and the bundle
$\U(TM_{\C})$ of unitary frames. We let $L\U(TM_{\C})$ be its looping, which is a Fr\'echet principal $L\U(d)$-bundle over $LM$. A standard result, \cite[Proposition 6.3.1]{PS86}, tells us that $L\U(d)$ acts in $V$ through $\U_{\res}(V)$. We do not care about the regularity of this action, and only consider for a single loop $\gamma\in LM$ the set
 \begin{equation*}
  \U_{\res}(\mc{H})_{\gamma} \defeq L\U(TM_{\C})_{\gamma} \times_{L\U(d)} \U_{\res}(V).
\end{equation*}
We may interpret elements $[\varphi,g]$ as unitary maps $\psi: V \to \mathcal{H}_{\gamma}$, $v\mapsto [\varphi,gv]$, making up a map $\U_L(\mathcal{H})_{\gamma} \to \U(V,\mathcal{H}_{\gamma})$.  
By construction, the set $\U_{\res}(\mc{H})_{\gamma}$ is a $\U_L(V)$-torsor, and contains $\O_L(\mathcal{H})_{\gamma}\cong L\SO(TM)_{\gamma} \times_{L\SO(d)}\O_L(V)$ under the inclusions $L\SO(TM) \subset L\U(TM_{\C})$ and $\O_L(V) \subset \U_L(V)$. We have a commutative diagram
\begin{equation*}
\xymatrix{\U_L(\mathcal{H})_{\gamma} \ar[r] & \U(V,\mathcal{H}_{\gamma})\\\O_L(\mathcal{H})_{\gamma} \ar@{^(->}[u] \ar[r] & \O(V,\mathcal{H}_{\gamma})\text{,}\ar@{^(->}[u]}
\end{equation*}
where the map at the bottom is $\O_L(\mathcal{H})_{\gamma} \subset \O(\mathcal{H})_{\gamma}\cong \O_L(V,\mathcal{H}_{\gamma})$, see \cref{lem:framebundle}.  

We construct a particular element in $\U_L(\mathcal{H})_{\gamma}$. To this end, let $\ph_{j}\in [-\pi,\pi)$ and $v_{j}\in T_{\gamma(0)}M_{\C}$ be as in \cref{lem:DiagonalizeDirac}. For each $t \in S^{1}$ define $\ph_{t} \in \U(M)_{\gamma(t)}$ by
 \begin{equation*}
  \ph_{t}(e_{j}) = e^{-i \frac{\ph_{j}}{2\pi}t}  \operatorname{pt}_{\gamma}(t) v_{j}.
 \end{equation*}
 This defines an element $\ph \in L\U(TM_{\C})$ over $\gamma$, and thus an element $\psi \defeq [\ph, \mathds{1}] \in \U_{\res}(\mc{H})_{\gamma}$.
 Let $\eta_{n,j}$ be as in \cref{lem:DiagonalizeDirac}.
 A computation shows that $\psi(\xi_{n,j}) = \eta_{n,j}$.
  Now we consider the Dirac sublagrangian $\operatorname{Eig}_{>0}(\D_{\gamma})$. From the characterization
 \begin{equation*}
  \operatorname{Eig}_{>0}(\D_{\gamma}) = \operatorname{span} \{ \eta_{n,j} \mid n \geqslant 1 , j=1,...,d \} \oplus \operatorname{span} \{ \eta_{n,j} \mid n = 0, j: \ph_{j} \neq -\pi \},
 \end{equation*}
 it follows readily that $\psi^{-1}(\operatorname{Eig}_{>0}(\D_{\gamma})) \subset V$ is a sublagrangian which is contained in our standard Lagrangian $L$.  For any $\nu\in \O_L(\mathcal{H})_p$, it follows that $\nu\psi^{-1}(\operatorname{Eig}_{>0}(\D_{\gamma}))\subset \nu(L)$. Since $\psi\in U_L(\mathcal{H})_{\gamma}$ and $\nu\in \O_L(\mathcal{H})_{\gamma}\subset \U_L(\mathcal{H})_{\gamma}$, it follows that $\nu\psi^{-1}\in \U_L(V)$. Then, by \cref{lem:FockSpacesUnitarilyEquivalent}, it follows that any completion of $\operatorname{Eig}_{>0}(\D_{\gamma})$ is equivalent to $\nu(L)$.
\end{proof}

By \cref{prop:equivalenceclasslag} we see that the Fock spaces $\mathcal{F}_{\nu(L)}$, with $\nu\in \O_L(\mathcal{H})_{\gamma}$, are precisely the Fock spaces considered by Stolz and Teichner in \cite{stolz3}. The problem of combining these fibres into a Fock bundle has not been further pursued in \cite{stolz3}, though the importance of string structures on $TM$ has been observed. Our treatment in \cref{sec:impLiftingGerbe} provides a complete solution to this problem, even on the level of rigged Hilbert bundles instead of just continuous ones.
First of all, we have seen that the fibres $\mathcal{F}^{\infty}_{\nu(L)}$ combine into a rigged Hilbert space bundle over $\O_L(\mathcal{H})$, the frame-dependent Fock bundle $\mathcal{F}\frm{\mathcal{H}}$.
Moreover, this bundle has the structure of a $\mathscr{L}_{\mathcal{H}}$-twisted rigged $\Cl^{\infty}(L\Spin(E))$-module.
We recall that $\mathscr{L}_{\mathcal{H}}$ is canonically isomorphic to the spin lifting gerbe $\mathscr{L}_{L\Spin(E)}$, and that the frame-dependent Fock bundle corresponds under this isomorphism to the twisted Fock bundle $\mathcal{F}^{\infty}(L\Spin(E))^{\twist}$, see \cref{prop:refinement}.
Now, if a loop spin structure $\widetilde{L\Spin}(E)$ on $E$ is provided -- for instance, via a string structure on $E$, it corresponds to a trivialization of $\mathscr{L}_{L\Spin(E)}$ under the isomorphism of \cref{prop:fockstructureslifting}. 
With this trivialization, the twisted Fock bundle untwists precisely to the odd spinor bundle $\mathbb{S}^{\infty}_{odd}(\widetilde{L\Spin}(E))$. This completes the construction of a smooth bundle of Fock space with typical fibre $\mathcal{F}_L^{\infty}$ using a loop spin structure on $E$.

\subsection{The even spinor bundle on loop space}

In this section we set up the Fock extension $(V_{ev},\alpha_{ev},L_{ev},\omega_{ev})$   of $L\Spin(d)$, for which  we suppose that $d$ is even.
We drop the subscript $(..)_{ev}$ in the following. We consider the Hilbert space $V \defeq L^{2}(S^{1},\C^{d})$, and equip it with the real structure given by pointwise complex conjugation.
Let $\{e_{j}\}_{j=0,...d-1,}$ be the standard basis of $\C^{d}$, and let $L_{0}$ be the complex linear span of the vectors $e_{j}+ie_{j+1}$, for even $j < d$.
The space $L_{0}$ is a Lagrangian, commonly called the ``standard Lagrangian''.
We now define the Lagrangian $L \subset L^{2}(S^{1},\C^{d})$ to be the direct sum of the constant functions with values in $L_{0}$, and the functions whose positive Fourier coefficients are zero.
We note that the group $L\SO(d)$ acts on $V=L^2(S^1,\C^d)$ by pointwise multiplication.
In fact, $L\SO(d)$ acts in $V$ through $\O_{L}(V)$ and the corresponding group homomorphism $L\SO(d) \rightarrow \O_{L}(V)$ is continuous (\cite[Prop. 6.25]{Ambler2012}) and thus smooth (\cite[Thm. IV.1.18]{Neeb2006}).
We write $\omega:L\Spin(d) \rightarrow \O_{L}(V)$ for the composition of the pointwise projection with the above map.
This completes the description of the even Fock extension $(V,\alpha,L,\omega)$ of $L\Spin(d)$.

\begin{proposition}\label{prop:EvenBasicCE}
 If $d \geqslant 6$, then $\omega^{*}\Imp_{L}(V)$ is the basic central extension of $L\Spin(d)$.
\end{proposition}

\begin{proof}
 The proof is analogous to the proof of \cite[Prop. 6.7.1]{PS86}, see also \cite[Thm. 3.26]{Kristel2019}.
 The result follows from a computation which shows that the 2-cocycle which characterizes the central extension $\Imp_{L}(V)$ pulls back to the 2-cocycle which characterizes the basic central extension of $L\Spin(d)$ along the map $\omega$.
 In fact, the only place where the proof of \cite[Thm. 3.26]{Kristel2019} must be modified so that it can be applied to the current situation is in the proof of \cite[Lem. 3.24]{Kristel2019}.
 This lemma states that for all $f_{1},f_{2} \in L \lie{so}(d)$ we have
  \begin{equation}\label{eq:CocycleIdentity}
  \mathrm{trace}([a_{1},a_{2}] - a_{3}]) = - \frac{1}{2\pi i} \int_{0}^{2\pi} \mathrm{trace}(f_{1}(t)f'_{2}(t)) \d \! t,
 \end{equation}
 where $A_{1} = \omega(f_{1})$ and $A_{2} = \omega(f_{2})$, and $A_{3} = [A_{1},A_{2}]$ and $a_{i} = P_{L} A_{i} P_{L}$.
 This equation still holds, as can be proven by an explicit computation using the definition
 \begin{equation*}
  \mathrm{trace}([a_{1},a_{2}] - a_{3}]) = \sum_{n,j} \langle z_{n} \otimes l_{j}, ([a_{1},a_{2}] - a_{3}) z_{n} \otimes l_{j} \rangle,
 \end{equation*}
 where $z_{n}(t) = e^{-int}$, and where $\{l_{j}\}_{j=1,...,d/2}$ forms a basis for $L_{0}$ and $\{l_{j}\}_{j=d/2+1,...,d}$ forms a basis for $\alpha(L_{0})$.
\end{proof}
Thus, a Fock structure on $L\Spin(E)$ is precisely the same thing as a loop spin structure on $E$.
The corresponding Fock bundle 
\begin{equation*}
\mathbb{S}^{\infty}_{even}(\widetilde{L\Spin}(E))\defeq \mathcal{F}^{\infty}(\widetilde{L\Spin}(E))
\end{equation*}
is called the \emph{even loop spinor bundle} of $E$  associated to the loop spin structure $\widetilde{L\Spin}(E)$.
Whether or not $E$ admits a loop spin structure, we have from \cref{sec:twisted} a sequence
\begin{equation*}
\mathscr{L}_{L\Spin(E)} \to \mathscr{L}_{\mathcal{H}} \to \mathscr{G}_{\mathcal{H}}
\end{equation*}
of bundle gerbes and refinements over $LM$, where $\mathcal{H}$ is the continuous real Hilbert space bundle 
\begin{equation*}
\mathcal{H} \defeq (L\Spin(E) \times V)/L\Spin(d)
\end{equation*}
defined in \cref{eq:hsbdlimp}, whose real structure is induced from the one of $V$. In the present situation, this is a bundle with interesting fibres, as the following lemma shows.

\begin{lemma}
\label{lem:fibresevenspinor}
For each loop $\gamma\in LM$, there is a canonical isomorphism of real Hilbert spaces
\begin{equation*}
\mathcal{H}_{\gamma} \cong L^2(S^1,\gamma^{*}E_{\C})\text{,}
\end{equation*}
characterized uniquely by the condition that
$[\varphi, \sigma] \mapsto  \varphi_{\C} (\sigma)$
for all $\varphi\in L\Spin(E)$ and $\sigma\in C^{\infty}(S^1,\C^d)$; where $\varphi_{\C}(\sigma)(t)\defeq \varphi(t)_{\C}( \sigma(t))$, using the complexification of the frame $\varphi(t):\R^n \to E_{\gamma(t)}$. 
\end{lemma}

\begin{proof}
It is clear that the formula characterizes the map uniquely, and the way $\Spin(d)$ acts in $V$ shows that it defines a map on the level of smooth sections. Since the frame $\varphi(t)$ is orthogonal, it is straightforward to see that this map is an isometry with respect to the $L^2$ scalar products, and hence extends to an isometric transformation of Hilbert spaces. It is easy to define an inverse map and to show that we have an isometric isomorphism. Finally, it is clear that the real structures are preserved under the complexification of a real-linear map. 
\end{proof}

The problem (or impossibility) of choosing Lagrangians $L_{\gamma}\subset \mathcal{H}_{\gamma}$ depending continuously on $\gamma\in LM$
has been addressed and (almost completely) solved by Ambler \cite{Ambler2012}. Our treatment of the polarization-dependent Fock bundle in \cref{sec:LagrangianTwistedSpinorBundle} provides a generalization of Ambler's work from a continuous to a smooth (rigged) setting. In more details, this generalization proceeds in the following steps:
\begin{itemize}

\item 
The Fr\'echet fibre bundle $\Lag_L(\mathcal{H})$ over $LM$ of \cref{lem:lagrangianprojection,lem:lagrangianclass} is a smooth version of Ambler's continuous \quot{polarization class bundle} $Y$ \cite[Chapter 7]{Ambler2012}.

\item The polarization-dependent Fock bundle $\mathcal{F}\pol{\mathcal{H}}$ of \cref{sec:LagrangianTwistedSpinorBundle} is a rigged version of Ambler's Fock space bundle $FY$ \cite[Chapter 9]{Ambler2012}.
The action of the Clifford bundle $\pi^{*}\Cl^{\infty}(L\Spin(E))$ asserted in \cref{lem:lagfockmod} is \cite[Lemma 9.3]{Ambler2012}.

\item
The Lagrangian Gra\ss mannian gerbe $\mathscr{G}_{\mathcal{H}}$ of \cref{sec:LagrangianTwistedSpinorBundle} is a Fr\'echet version of the continuous bundle gerbe of \cite[Chapter 11]{Ambler2012}. The coincidence of the $\U(1)$-bundles in both constructions can be seen using \cref{re:intertwinerbundle}.

\item
The fact that the polarization-dependent Fock bundle is a $\mathscr{G}_{\mathcal{H}}$-twisted vector bundle (\cref{prop:lagfocktwist}) has not been discussed by Ambler. Instead, the implication  that a trivialization of $\mathscr{G}_{\mathcal{H}}$ lets the polarization-dependent Fock bundle descend to a Fock bundle on $LM$ (\cref{th:unwistfdfbLag}) is proved manually \cite[Theorem 13.10]{Ambler2012}. 
\end{itemize} 
One thing that has apparently not been seen in \cite{Ambler2012} is the relation to loop spin structures, or string structures on $E$. We have established this relation by showing that the Lagrangian Gra\ss mannian gerbe $\mathscr{G}_{\mathcal{H}}$ is canonically isomorphic to the spin lifting gerbe $\mathscr{L}_{\Spin(E)}$ (\cref{prop:refinement,prop:equivambler}), whose trivializations are precisely the loop spin structures on $E$.

\newcommand{\etalchar}[1]{$^{#1}$}

\def\kobiburl#1{
   \IfSubStr
     {#1}
     {://arxiv.org/abs/}
     {\kobibarxiv{#1}}
     {\kobiblink{#1}}}
\def\kobibarxiv#1{\href{#1}{\texttt{[arxiv:\StrGobbleLeft{#1}{21}]}}}
\def\kobiblink#1{
  \StrSubstitute{#1}{\~{}}{\string~}[\myurl]
  \StrSubstitute{#1}{_}{\underline{\;\;}}[\mylink]
  \StrSubstitute{\mylink}{&}{\&}[\mylink]
  \StrSubstitute{\mylink}{/}{/\allowbreak}[\mylink]
  \newline Available as: \mbox{\;}
  \href{\myurl}{\texttt{\mylink}}}

\raggedright
\addcontentsline{toc}{section}{\refname}
\small
\bibliographystyle{kobib}
\bibliography{bibfile}

Peter Kristel ({\it peter.kristel@uni-greifswald.de})

\smallskip

Konrad Waldorf ({\it konrad.waldorf@uni-greifswald.de})

\medskip

Universit\"at Greifswald\\
Institut f\"ur Mathematik und Informatik\\
Walther-Rathenau-Str. 47\\
17487 Greifswald\\
Germany

\end{document}